\documentclass[12pt]{amsart}
\usepackage[margin=1in]{geometry}
\usepackage{dsfont}

\usepackage[sc,osf]{mathpazo} %
\usepackage[utf8]{inputenc}
\usepackage[T1]{fontenc}

\RequirePackage{mathtools}
\RequirePackage{amsthm,amsmath,amsfonts,amssymb}

\RequirePackage[colorlinks,citecolor=blue,urlcolor=blue]{hyperref}
\RequirePackage{hyperref}
\RequirePackage{graphicx}

\usepackage[numbers,sort]{natbib}

\usepackage{cleveref}
\usepackage{mathrsfs}
\usepackage{needspace}
\usepackage{newtxtext}

\theoremstyle{plain}
\newtheorem{thm}{Theorem}
\newtheorem{prop}{Proposition}
\newtheorem{lemma}{Lemma}

\theoremstyle{definition}
\newtheorem{assumption}{Assumption}
\newtheorem{remark}{Remark}

\newcommand{\defeq}{\coloneqq}

\newcommand{\eqstop}{.}
\newcommand{\eqcomma}{,}

\newcommand{\ie}{\textit{i.e.}}

\newcommand{\dd}{\mathrm{d}}
\newcommand{\dtv}{\dd_{\mathsf{TV}}}

\renewcommand{\vec}[1]{\boldsymbol{#1}}

\crefname{thm}{theorem}{theorems}
\Crefname{thm}{Theorem}{Theorems}
\crefname{prop}{proposition}{propositions}
\Crefname{prop}{Proposition}{Propositions}
\crefname{lemma}{lemma}{lemmas}
\Crefname{lemma}{Lemma}{Lemmas}
\crefname{cor}{corollary}{corollaries}
\Crefname{cor}{Corollary}{Corollaries}
\crefname{corollary}{corollary}{corollaries}
\Crefname{corollary}{Corollary}{Corollaries}
\crefname{remark}{remark}{remarks}
\Crefname{remark}{Remark}{Remarks}
\crefname{assumption}{assumption}{assumptions}
\Crefname{assumption}{Assumption}{Assumptions}
\title[Cutoff in the TK model]{Cutoff profiles and boundary effects in the Togashi--Kaneko model of autocatalytic reactions with weak mutations} 

\author{Wasiur R. KhudaBukhsh}
\thanks{Wasiur R. KhudaBukhsh was partially supported by the Engineering and Physical
Sciences Research Council (EPSRC) of UK Research and Innovation
(UKRI) [grant number EP/Y027795/1].
}
\address{Wasiur R. KhudaBukhsh\\
School of Mathematical Sciences\\
University of Nottingham\\
University Park\\
Nottingham NG7 2RD\\
United Kingdom
}
\email{wasiur.khudabukhsh@nottingham.ac.uk}
\urladdr{https://www.wasiur.xyz/}

\author{Yangrui Xiang}
\thanks{Yangrui Xiang was supported in part by the Engineering and Physical
Sciences Research Council (EPSRC) of UK Research and Innovation
(UKRI) [grant number EP/Y027795/1] and by the National Science
Foundation (NSF) [grant number DMS-2246815].
}
\address{Yangrui Xiang\\
109 Lockett Hall\\
Department of Mathematics\\
Louisiana State University\\
Baton Rouge, LA 70803, United States of America  
}
\email{yxiang8@lsu.edu}

\usepackage{etoolbox}
\makeatletter
\patchcmd{\@setaddresses}{\def\\{\unskip, \ignorespaces}}{\def\\{\newline\hspace*{\parindent}}}{}{\PackageWarning{main}{address patch failed}}
\makeatother

\begin{document}

\begin{abstract}
We determine explicit bulk and worst-case total variation cutoff
profiles for the configuration law of a two-species
Togashi--Kaneko model with a fixed population of $n$ particles
and weak mutations. Boundary initial conditions introduce
an additional logarithmic growth phase and change both the leading
mixing-time coefficient and the limiting cutoff profile.

Let $\kappa_1,\kappa_2>0$ be the autocatalytic rate constants and
put $\lambda_*=\kappa_1\kappa_2/(\kappa_1+\kappa_2)$.
For deterministic initial configurations whose density converges
to $m\in(0,1)$, the total variation distance to equilibrium at time
$t_n(s)=\frac{\log n}{2\lambda_*}+\frac{s}{\lambda_*}$
converges to $\mathcal G(r(m)e^{-s})$ for every fixed
$s\in\mathbb R$, where
$\mathcal G(r)=\dtv(\mathcal N(r,1),\mathcal N(0,1))$, $r(m)>0$ is explicit, and $\mathcal N(a,b^2)$ denotes the Gaussian distribution with mean $a$ and variance $b^2$.
For the worst-case distance $\mathscr D_n(t)$ over all initial
configurations, we prove
$\mathscr D_n(T_n+s/\lambda_*)\to\mathcal P(s)$, where
$T_n=\left(
\frac{1}{\min\{\kappa_1,\kappa_2\}}
+\frac{1}{2\lambda_*}
\right)\log n$.
The profile $\mathcal P$ is the total variation distance between
the equilibrium Gaussian and a Gaussian mixture whose random
displacement is determined by a gamma-distributed amplitude
arising from early growth at the boundary.
Unlike the bulk profile, $\mathcal P$ depends on the mutation
parameters.
Both cutoffs have windows of constant order, and we identify
the asymptotically worst boundary configurations.
\end{abstract}
\maketitle

\tableofcontents

\section{Introduction}
\label{sec:intro}
We determine explicit total variation cutoff profiles for the
configuration law of a two-species Togashi--Kaneko model with fixed-population, weak mutations, and nonlinear kinetics
(see \Cref{sec:TK_model_desrciption}).
The density-dependent rates produce a unique stable interior
coexistence equilibrium for the limiting density equation.
In our model, weak mutations allow a missing species
to reappear from a boundary state, initiating a growth phase
driven by autocatalysis and further mutations. Our approach connects this stochastic growth phase
with the subsequent relaxation towards equilibrium to quantify
how boundary initial conditions affect the mixing time and cutoff
profile of the full configuration law.

Our results cover deterministic bulk initial configurations and
the worst-case distance over all initial states, including the
identification of asymptotically worst starting configurations.
Both cutoff windows have constant order, but their locations have
different leading coefficients in front of $\log n$.
The bulk profile is the total variation distance between a Gaussian
law and a deterministic translate of that law. From the boundary,
early growth produces a random time shift that persists at the
equilibrium fluctuation scale, giving a worst-case profile expressed
as the distance between a Gaussian mixture and the equilibrium
Gaussian. Thus, boundary initial conditions change both the leading
mixing time and the shape of the cutoff profile.

The cutoff phenomenon
\cite{Aldous_Diaconis_cutoff,Diaconis1996cutoff}
describes an abrupt transition to equilibrium over a time window
of smaller order than the mixing time.
The cutoff profile describes the limiting total variation distance
within this window and can depend on the initial condition.
Lacoin \cite{Lacoin_profile} computed an explicit Gaussian cutoff
profile for the simple exclusion process on the circle.
Further explicit profiles have been obtained for random walks on
sparse directed graphs \cite{BordenaveCaputoSalez2018},
random transpositions \cite{Teyssier2020Profile}, and the symmetric
exclusion process with reservoirs \cite{gonccalves2021sharp}.

\subsection{The Togashi--Kaneko model}
\label{sec:TK_model_desrciption}
The TK model describes autocatalytic reactions involving
$d$ species $S_1,\ldots,S_d$, where $d$ is referred to as the dimension. In the two-species case,
the reaction network is
\begin{align}\label{eq:simpleTK}
S_1 + S_2 &\longrightarrow 2 S_2, \quad 
S_1 + S_2 \longrightarrow 2 S_1, \quad 
\emptyset \longrightarrow S_1 \longrightarrow \emptyset, \quad  
\emptyset \longrightarrow S_2 \longrightarrow \emptyset \eqstop 
\end{align}  
The reactions of the form $\emptyset \longrightarrow S_1 \longrightarrow \emptyset, \,  
\emptyset \longrightarrow S_2 \longrightarrow \emptyset 
$ describe the production of species $S_1$ and $S_2$ from the environment (``birth'' reactions) and their eventual degradation (``death'' reactions). On the other hand, the reactions of the form  $S_1 + S_2 \longrightarrow 2 S_2, \,
S_1 + S_2 \longrightarrow 2 S_1 
$
are autocatalytic since the presence of $S_2$ promotes the production of additional $S_2$ in the first reaction while $S_1$ promotes its own replication in the second. 

The interest in the reaction network \eqref{eq:simpleTK}  can be attributed to its peculiar switching dynamics \cite{Bibbona2020Interface,TK2001Transitions}, called ``discreteness-induced transitions'' in the physics and mathematical biology literature. When the autocatalytic reactions are much faster relative to the production and degradation reactions (\ie, when the autocatalytic reaction rate constants are much larger than the production and the degradation reaction rate constants), the system is observed to  spend most of its time near states where one of the species is absent, switching rapidly from one such state to another. Togashi and Kaneko observed this switching behaviour
numerically in a four-species TK model
\cite{TK2001Transitions,Togashi2003}:
\begin{align*}
S_2 + S_3 \longrightarrow 2 S_3, \,
S_3 + S_4 \longrightarrow 2 S_4, \, S_4 + S_1\longrightarrow 2 S_1,\,
\emptyset \longrightarrow S_3 \longrightarrow \emptyset, \, 
\emptyset \longrightarrow S_4 \longrightarrow \emptyset \eqstop 
\end{align*}  
In an attempt to explain the switching behaviour, the authors in \cite{Bibbona2020Interface} derived the stationary distribution for the $d$-dimensional {TK} model with a loop of autocatalytic reactions and showed that the stationary distribution was indeed bimodal with almost all mass concentrated on states where one of the species is absent, under some conditions on the reaction rate constants. 

In this paper, we will consider the 2-dimensional {TK} model with mutations, which comprises the following reactions: 
\begin{equation}\label{eq:2dTKmutations}
\begin{gathered}
S_1+S_2\longrightarrow 2S_2,\quad
S_1+S_2\longrightarrow 2S_1,\quad
\emptyset\longrightarrow S_1\longrightarrow\emptyset,\quad
\emptyset\longrightarrow S_2\longrightarrow\emptyset\eqstop\\
S_1\longrightarrow S_2\eqcomma\quad
S_2\longrightarrow S_1\eqstop
\end{gathered}
\end{equation}
Here, the last two reactions describe mutations. When the autocatalytic reactions and the mutation reactions are fast (of the order $\varepsilon^{-1}$ for some small positive parameter $\varepsilon$) compared to the production and the degradation reactions, the individual species counts, when modelled as a continuous-time Markov chain, are fast and have rapid jumps. On the other hand, the total population size, which jumps only when there is a birth or a death reaction, is slow.  In a recent work \cite{FuEtAl2026Fragility}, the authors established  a rigorous stochastic averaging principle for this model, which allows one to describe the dynamics of the slow variable (total population size) in terms of ergodic moments of the fast species. See \cite[Theorem 3.1]{FuEtAl2026Fragility}.

Stochastic averaging describes the effective slow dynamics through
equilibrium averages of the fast dynamics with the slow variable
held fixed; see
\cite{Hasminskii1966a,Khasminskii1966b,Ganguly2026tQSSA,
Kurtz1992Averaging,Ball2006AAM,Ganguly2026StatInference,
Elsayed2024StochasticAveraging}.
For the TK model, the averaging result of
\cite[Theorem 3.1]{FuEtAl2026Fragility} motivates a quantitative
study of how quickly the species dynamics approach equilibrium
when the total population is held fixed.
In this paper, we modify the model slightly (see \Cref{subsec:TK_model_def} and \Cref{remark:difference2Fupaper} below) and focus on the dynamics of the fast species (with the slow variable, total population size, frozen) and  report the cutoff phenomenon.

Consider a fixed-population TK-type model with $n$ 
particles, each of type $0$ or $1$. When the fraction of type $1$
particles is $m$, the rates at which an individual particle changes type
are
\[
0\to1:\quad \kappa_1m(1-m)+\frac{\iota_1}{n}(1-m),
\qquad
1\to0:\quad \kappa_2m(1-m)+\frac{\iota_2}{n}m,
\]
where all four parameters are fixed and strictly positive. Mutations
have total rate of order one and allow a missing species to reappear.
These rates differ from the frozen mass-action rates in
\cite{FuEtAl2026Fragility}; the precise comparison is given
in Remark~\ref{remark:difference2Fupaper}.
The factors $1-m$ in the $0\to1$ rate and $m$ in the $1\to0$
rate decrease as the species being produced becomes more abundant.
Together, they give the limiting density equation a stable
interior coexistence equilibrium. We study this modification
to understand how the random early growth of a species introduced
by mutation affects the mixing time and cutoff profile.
The limiting global density satisfies
\[
  \dot m
  =m(1-m)\left(\kappa_1-(\kappa_1+\kappa_2)m\right).
\]
Its stable interior equilibrium is
$ \alpha_* = \frac{\kappa_1}{\kappa_1+\kappa_2}$.

\subsection{Convergence profiles and the effect of the boundary}

Our results identify two distinct cutoff profiles for the configuration law: a Gaussian translation profile in the bulk and a
mutation-dependent Gaussian-mixture profile in the worst case. The respective
cutoff locations and their common window scale are
\[
 t_n=\frac{\log n}{2\lambda_*},\qquad
 T_n=t_n+\frac{\log n}{\min\{\kappa_1,\kappa_2\}},\qquad
 w_n=\lambda_*^{-1}.
\]

The two locations are separated by
    $(\min\{\kappa_1,\kappa_2\})^{-1}\log n$, although both
windows are of order one. Moreover, for every fixed bulk initial densities in an interval
$[\varepsilon_0,1-\varepsilon_0]$, with $\varepsilon_0\in(0,1/2)$,
and every fixed mixing level $\epsilon\in(0,1)$,
$
 t_{\mathrm{mix},\varepsilon_0}^{(n)}(\epsilon)=t_n+O(1),
 \
 \mathscr D_n\!\left(
 t_{\mathrm{mix},\varepsilon_0}^{(n)}(\epsilon)
 \right)\to 1,
$
where $\mathscr D_n$ denotes the worst-case total variation distance.
Thus, the entire fixed bulk class reaches level $\epsilon$ while the
worst-case distance remains asymptotically maximal.
Figure~\ref{fig:bulk_worst_cutoff} summarises this contrast.

\begin{figure}[!t]
    \centering
    \includegraphics[width=\linewidth]{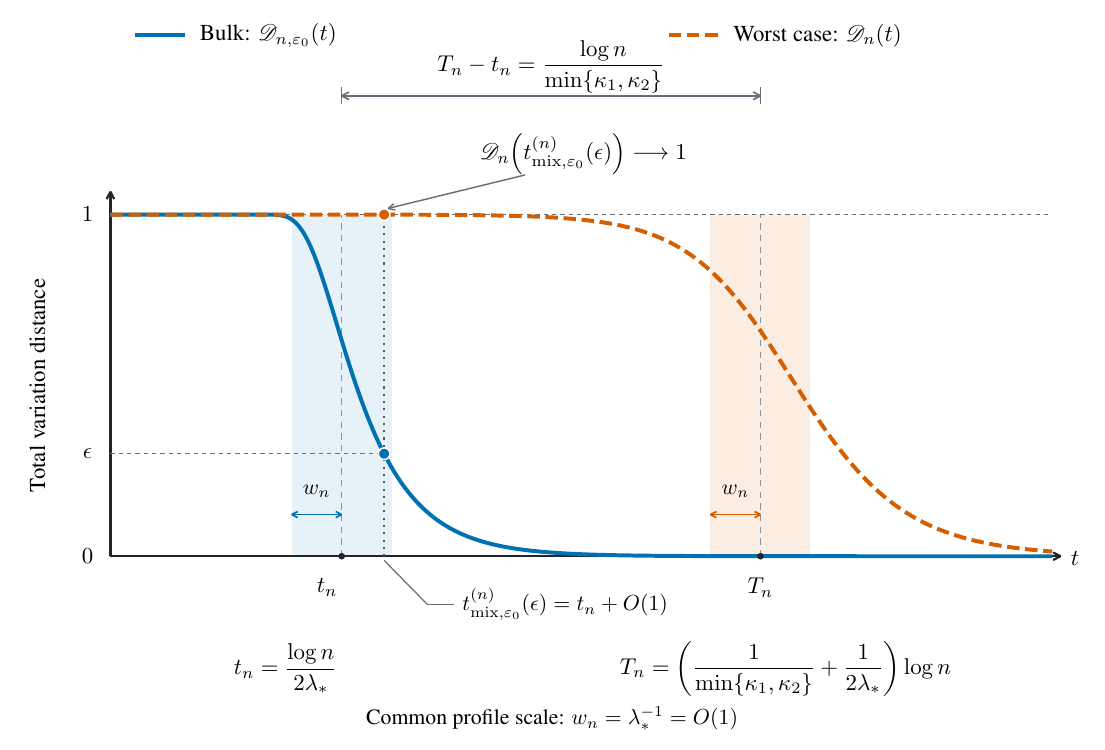}
    \caption{
    Bulk (solid) and worst-case (dashed) total variation
    distances, for fixed $\varepsilon_0\in(0,1/2)$.
    The cutoff locations are separated by
    $(\min\{\kappa_1,\kappa_2\})^{-1}\log n$, while both cutoff windows have scale $w_n=\lambda_*^{-1}$.
   At the marked bulk mixing time, the bulk distance drops below $\epsilon\in(0,1)$, whereas the worst-case
    distance tends to one.
    }
    \label{fig:bulk_worst_cutoff}
\end{figure}

The bulk profile makes the role of the initial condition explicit.
For a deterministic initial sequence whose global density converges
to $m\in(0,1)$, the total variation distance to equilibrium at time
$t_n+s w_n$ converges to $\mathcal G(r(m)e^{-s})$, where 
$
 \mathcal G(r)=2\Phi(r/2)-1,
$ 
$\Phi$ is the standard normal cumulative distribution function and the function $r(m)$
is the explicit amplitude given in \eqref{def:initial_profile_amplitude}.
The convergence is uniform in $s$ in compact intervals. Maximizing
over the fixed bulk interval gives the bulk cutoff profile and
mixing-time asymptotics; see
Theorems~\ref{thm_convergence_profile} and~\ref{cutoff_profile_tk} for precise statements.

Of particular interest is $r(\alpha_*)=1$: a deterministic
configuration with limiting density $\alpha_*$ still has a
nontrivial cutoff profile around $t_n$.
We describe this dependence on the initial state by following
the two blocks defined by the initial species.
Their deterministic deviations from equilibrium remain nonzero
on the $n^{-1/2}$ fluctuation scale, even though their weighted
contributions to the global density cancel in the limit. The bulk profile quantifies the logarithmic mixing
time even when the initial global density converges to its equilibrium value.

The boundary profile arises from a different mechanism. Starting from either boundary configuration $\vec{\mathbf0}$ or
$\vec{\mathbf1}$, the emergence and early growth of the absent
species are approximated by a linear birth process with immigration. Its gamma-distributed growth amplitude produces a
random time delay. Subsequent deterministic relaxation converts this
delay into a random displacement at the equilibrium fluctuation scale.
Theorem~\ref{thm:worst_case_cutoff} identifies the resulting limit
$
 \mathscr D_n(T_n+s w_n)\longrightarrow\mathcal P(s),
$
uniformly in $s$ in compact intervals, and gives the worst-case
mixing time. The limiting rescaled count law at the worst endpoint is a Gaussian mixture, obtained by averaging Gaussian laws whose
means depend on the gamma growth amplitude.
The profile $\mathcal P$ is the total variation distance between this mixture and the equilibrium Gaussian.
The mutation parameters $\iota_1,\iota_2$ enter through the gamma
growth amplitudes and affect both the worst-case profile and the
order-one correction to the worst-case mixing time. They do not enter the limiting bulk profile, and the leading
logarithmic cutoff coefficients depend only on the autocatalytic
parameters $\kappa_1,\kappa_2$.

The worst-case profile is attained asymptotically from
$\vec{\mathbf0}$ if $\kappa_1<\kappa_2$ and from
$\vec{\mathbf1}$ if $\kappa_2<\kappa_1$.
When $\kappa_1=\kappa_2$, it is attained from $\vec{\mathbf0}$
if $\iota_1<\iota_2$ and from $\vec{\mathbf1}$ if
$\iota_2<\iota_1$; both boundary configurations attain it
if $\iota_1=\iota_2$ as well.
For every fixed bulk class and mixing level, the ratio of
worst-case to bulk mixing times tends to $2$ when
$\kappa_1=\kappa_2$, and to a value strictly between $2$ and $3$
otherwise.

\subsection{Relation to earlier work and proof strategy}
Barbour, Brightwell and Luczak \cite{BBL2025DensityCutoff}
establish constant-window total variation cutoff for
density-dependent population processes near stable equilibria.
Our count process has the same density-dependent leading
structure, with bounded mutation terms that would require a
perturbative treatment to apply their framework.
He, Luczak and Ross \cite{HeLuczakRoss2025SIS} identify the
cutoff location and window for the infected-count process of
the logistic SIS model with a fixed self-infection rate per
individual.
We determine explicit profiles for the full configuration law,
covering both bulk initial configurations and growth from
boundary states.

Barbour, Hamza, Kaspi and Klebaner
\cite{BarbourHamzaKaspiKlebaner2015Escape}
describe boundary escape through a branching approximation
followed by deterministic evolution with a random time shift.
In our model, the initial approximation is a birth process
with immigration. We show how its gamma-distributed growth
amplitude determines the random displacement near equilibrium
and hence the Gaussian-mixture cutoff profile. Unimodality of the birth--death likelihood ratios, together with
an estimate for the erasure of initial labels, yields the limiting
total variation distance between the full configuration law and
its invariant measure. Together with the bulk results, estimates for fixed and growing
initial minority populations identify the worst-case profile
over all configurations.

Our initial-block decomposition and quantitative Gaussian
approximation are closely related to the noisy voter analysis
of Aljovin, Jara and Xiang \cite{JaraXiang}.
Their weak-noise results distinguish logarithmic thermalisation
from equilibration on the time scale $n$, in Kantorovich distance.
Here, the two block densities contribute to the bulk total
variation profile on the same logarithmic scale.
The Gaussian approximation uses Barbour's semigroup approach
to Stein's method \cite{Barbour1990Diffusion}.

Our product-measure comparison follows the entropy method
of Gon{\c{c}}alves, Jara, Marinho and Menezes
\cite{gonccalves2021sharp}, who obtain explicit Gaussian profiles
for exclusion with reservoirs from suitable product initial laws.
For our deterministic initial configurations, a short restart
combines entropy estimates with a coupling of the block counts
to obtain a total variation approximation.
Long-time entropy estimates control the subsequent evolution
of the reference measure.
We use Yau's relative entropy method \cite{Yau-relative-entropy};
related applications appear in
\cite{jara2018non,jara_amorim_xiang}.

\subsection{Organisation of the paper}
Section~\ref{sec:Preliminaries} defines the model and the distances. Section~\ref{sec:main_results} states the cutoff results and the
quantitative approximations. Sections~\ref{sec:tk_approx}--\ref{sec:cutoff_profile}
prove the bulk profile. Section~\ref{sec:worst_case} identifies the
boundary profile and compares worst-case and bulk mixing.

\section{Preliminaries and notations}
\label{sec:Preliminaries}

\subsection{A two-species Togashi--Kaneko model with weak mutations}
\label{subsec:TK_model_def}
We consider a two-species, fixed-population version of the Togashi--Kaneko stochastic autocatalytic reaction network. We focus on the fast variable in the stochastic averaging result \cite[Theorem 3.1]{FuEtAl2026Fragility} with slightly modified  kinetics in this paper. As such, the jump rates, which are specified below, differ from the mass-action rates of the fast process (with the slow variable frozen) in
\cite{FuEtAl2026Fragility}. Let
\(n\in\mathbb N\) be the total number of particles, and define
\[
    \Lambda_n\defeq \{1,\ldots,n\},\qquad Q\defeq \{0,1\},\qquad
    \Omega_n\defeq Q^{\Lambda_n}.
\]
The two elements of \(Q\) label the molecular species \(X_0\) and \(X_1\) (denoted by $S_1$ and $S_2$ in \cite{FuEtAl2026Fragility}).  A configuration \(\sigma=(\sigma_x;x\in\Lambda_n)\in\Omega_n\) records the species type of each particle: \(\sigma_x=i\) means that particle \(x\) is of species \(X_i\). The index \(x\) is only a label for the particles. Let \(\vec{\mathbf 0}\) and \(\vec{\mathbf 1}\) be the constant configurations
with values \(0\) and \(1\), respectively, and define
    $\Omega_n^*\defeq \Omega_n\setminus\{\vec{\mathbf 0},\vec{\mathbf 1}\}$.
We use \(\Omega_n^*\) when excluding boundary initial configurations. The Markov chain itself is defined on all of \(\Omega_n\). For \(x\in\Lambda_n\) and \(\sigma\in\Omega_n\), let
\(\sigma^x\in\Omega_n\) be obtained from \(\sigma\) by switching the state at
\(x\):
    $$(\sigma^x)_y\defeq
    \begin{cases}
        1-\sigma_x, & y=x,\\
        \sigma_y, & y\ne x.
    \end{cases}$$
The \emph{global density} of species \(X_1\) in a configuration $\sigma$ is defined as
\begin{align}
\label{def:global_density}
    M^{(n)}(\sigma)\defeq \frac1n\sum_{x\in\Lambda_n}\sigma_x .
\end{align}
Thus,  \(1-M^{(n)}(\sigma)\) is the density of species \(X_0\). Now, fix parameters \(\kappa_1,\kappa_2,\iota_1,\iota_2>0\). For
\(x\in\Lambda_n\), define the \emph{transition rate}
\begin{equation}
\label{def:TK_flip_rate}
\begin{aligned}
    c_x(\sigma)
    &\defeq (1-\sigma_x)
    \left(
        \kappa_1M^{(n)}(\sigma)\bigl(1-M^{(n)}(\sigma)\bigr)
        +\frac{\iota_1}{n}\bigl(1-M^{(n)}(\sigma)\bigr)
    \right) \\
    &\quad+\sigma_x
    \left(
        \kappa_2M^{(n)}(\sigma)\bigl(1-M^{(n)}(\sigma)\bigr)
        +\frac{\iota_2}{n}M^{(n)}(\sigma)
    \right)\eqstop
\end{aligned}
\end{equation}
Equivalently, writing \(m=M^{(n)}(\sigma)\), a particle of species \(X_0\)
changes into species \(X_1\) at rate
$
    \kappa_1m(1-m)+\frac{\iota_1}{n}(1-m),
$
whereas a particle of species \(X_1\) changes into species \(X_0\) at rate
$
    \kappa_2m(1-m)+\frac{\iota_2}{n}m.
$
The terms with coefficients \(\kappa_1\) and \(\kappa_2\) describe
\emph{autocatalytic reactions}; their common factor \(m(1-m)\) makes these reactions possible only when both species coexist. The terms with coefficients
\(\iota_1\) and \(\iota_2\) describe \emph{weak mutations}: their per-particle rates
are of order \(n^{-1}\), so their total contribution is of order one.  The \emph{generator} of the continuous-time Markov chain $\bigl(\sigma^{(n)}(t);t\ge0\bigr)$
is given by
\begin{equation}
\label{def:TK_generator}
    L_nf(\sigma)\defeq \sum_{x\in\Lambda_n}
    c_x(\sigma)\bigl(f(\sigma^x)-f(\sigma)\bigr)\eqcomma
\end{equation}
for bounded measurable $f:\Omega_n\to\mathbb R$. Mutations prevent the two boundary configurations from being absorbing states. For \(\sigma,\eta\in\Omega_n\), the \emph{jump rate} from a configuration \(\sigma\) to the configuration
\(\eta\) is
    $R_n(\sigma,\eta)\defeq \sum_{x\in\Lambda_n}
    c_x(\sigma)\mathbf1_{\{\eta=\sigma^x\}}$.
Thus, \(R_n(\sigma,\eta)>0\) only when \(\sigma\) and \(\eta\) differ at exactly one site. The \emph{carré-du-champ operator} associated with the generator \(L_n\) is
\[
    \Gamma_n(f,g)\defeq L_n(fg)-fL_ng-gL_nf.
\]
For flip dynamics such as ours,
    $\Gamma_n(f,g)(\sigma)=\sum_{x\in\Lambda_n}c_x(\sigma)
    \bigl(f(\sigma^x)-f(\sigma)\bigr)
    \bigl(g(\sigma^x)-g(\sigma)\bigr)$.
When \(f=g\), we simply write \(\Gamma_nf\defeq \Gamma_n(f,f)\), so that
$\Gamma_nf(\sigma)=\sum_{x\in\Lambda_n}c_x(\sigma)
    \bigl(f(\sigma^x)-f(\sigma)\bigr)^2$.

The dynamics is of mean-field type: its rates depend on \(\sigma\) only through
\(\sigma_x\) and \(M^{(n)}(\sigma)\). If \(\varphi\) is a permutation of
\(\Lambda_n\), define \((\varphi\sigma)_x\defeq \sigma_{\varphi^{-1}(x)}\). Then, note that
\[
    M^{(n)}(\varphi\sigma)=M^{(n)}(\sigma),
    \qquad
    L_n(f\circ\varphi)=(L_nf)\circ\varphi.
\]
Thus, the dynamics is equivariant under arbitrary permutations of particle labels. Because \(\iota_1,\iota_2>0\), the chain is irreducible and has a unique \emph{invariant probability measure}, which we denote by
$
    \mu_{\kappa_1,\kappa_2,\iota_1,\iota_2}^{(n)}.
$
The explicit formula for the invariant measure can be found in Appendix~\ref{appendix:approx_invariant_measure}.
The permutation symmetry and uniqueness of the invariant measure imply
\[
    \mu_{\kappa_1,\kappa_2,\iota_1,\iota_2}^{(n)}(\sigma)
    =
    \mu_{\kappa_1,\kappa_2,\iota_1,\iota_2}^{(n)}(\eta)
    \qquad\text{whenever }
    M^{(n)}(\sigma)=M^{(n)}(\eta).
\]
Thus, the invariant measure is uniform on each global-density level set. In the following, we write
    $M_t^{(n)}\defeq M^{(n)}\left(\sigma^{(n)}(t)\right)$
for the \emph{global density process}. Let \(K_t^{(n)}=nM_t^{(n)}\) be the number of particles of type \(1\)
at time \(t\), and denote its transition rate from \(k\) to
\(\ell\ne k\) by \(q_n(k,\ell)\). The generator
\eqref{def:TK_generator} gives
\[
\begin{aligned}
 q_n(k,k+1)
 &=\frac{(n-k)^2}{n^2}(\kappa_1k+\iota_1),
 &&0\le k<n,\\
 q_n(k,k-1)
 &=\frac{k^2}{n^2}\{\kappa_2(n-k)+\iota_2\},
 &&1\le k\le n.
\end{aligned}
\]
These are the birth and death rates denoted by \(b_n(k)\) and
\(d_n(k)\), respectively, in \eqref{eq:count_rates_worst_case},
with \(b_n(n)=d_n(0)=0\).

\begin{remark}[Relation to the mass-action TK model]
\label{remark:difference2Fupaper}
The process studied here is a TK-type mean-field spin system with nonlinear kinetics. Write \(q_{\rm MA}(k,\ell)\) for the transition rate from count
\(k\) to count \(\ell\) in the comparison mass-action process.
The  mass-action rates in \cite{FuEtAl2026Fragility} (with the slow variable total population size frozen at $n$) are
\[
 q_{\rm MA}(k,k+1)=(n-k)(\kappa_1k+\iota_1),\qquad
 q_{\rm MA}(k,k-1)=k\left(\kappa_2(n-k)+\iota_2\right).
\]
For $k=nM^{(n)}$, the generator \eqref{def:TK_generator} instead gives
\[
 q_n(k,k+1)=\frac{(n-k)^2}{n^2}(\kappa_1k+\iota_1),\qquad
 q_n(k,k-1)=\frac{k^2}{n^2}\left(\kappa_2(n-k)+\iota_2\right).
\]
The two directions are modified by different factors, so these
processes are not related by a common time change. In particular,
the modification changes the drift and invariant law. Moreover, the added nonlinearity allows us to see interesting behaviour. 
\end{remark}

\subsection{Notations}
\label{subsec:notations}

For \(T>0\), let \(\mathcal D([0,T],\Omega_n)\) be the space of
\emph{càdlàg paths}, equipped with the
Skorohod topology and its Borel \(\sigma\)-algebra (see \cite{Billingsley1999Convergence,Ethier_Kurtz1986MPCC}). For probability measures \(\mu,\nu\) on \(\Omega_n\), their \emph{total variation} distance
 is defined as
\[
    \dtv(\mu,\nu)\defeq
    \frac12\sum_{\sigma\in\Omega_n}|\mu(\sigma)-\nu(\sigma)|
    =
    \sup_{A\subseteq\Omega_n}|\mu(A)-\nu(A)|.
\]
For a probability measure
\(\mu_0\) on \(\Omega_n\), let \(\mathbb P_{\mu_0}^{(n)}\) be the law of the chain with generator \(L_n\) and initial distribution
\(\mu_0\), and let \(\mathbb E_{\mu_0}^{(n)}\) be the corresponding
expectation. If \(\mu_0=\delta_{\sigma_0}\), we write
$
    \mathbb P_{\sigma_0}^{(n)}\defeq \mathbb P_{\delta_{\sigma_0}}^{(n)}, $ and $
    \mathbb E_{\sigma_0}^{(n)}\defeq \mathbb E_{\delta_{\sigma_0}}^{(n)}.
$
When the initial distribution is clear, we write simply
\(\mathbb P^{(n)}\) and \(\mathbb E^{(n)}\). For a random element \(X\), we denote its law under \(\mathbb P\) by
$
    \operatorname{Law}(X;\mathbb P)\defeq \mathbb P\circ X^{-1}.
$
When the probability measure is clear from the context, we write \(\operatorname{Law}(X)\). For a probability measure \(\nu\) on \(\Omega_n\) and a function \(f:\Omega_n\to\mathbb R\), we write
    $\mathbb E_\nu[f]\defeq \sum_{\sigma\in\Omega_n}f(\sigma)\nu(\sigma)$.
Let \(\left(S_t^{(n)};t\ge0\right)\) be the \emph{Markov semigroup} generated by \(L_n\). Its action on bounded, measurable functions is
    $S_t^{(n)}f(\sigma)\defeq
    \mathbb E_\sigma^{(n)}
    \left[f\left(\sigma^{(n)}(t)\right)\right]$, for $f:\Omega_n\to\mathbb R$.
For an initial distribution \(\mu_0\) on \(\Omega_n\), we write
\(\mu_0S_t^{(n)}\) for the distribution of the chain at time \(t\).
Thus, for every \(f:\Omega_n\to\mathbb R\),
\[
    \int_{\Omega_n}f(\sigma)\,
        \bigl(\mu_0S_t^{(n)}\bigr)(d\sigma)
    =
    \int_{\Omega_n}S_t^{(n)}f(\sigma)\,\mu_0(d\sigma)
    =
    \mathbb E_{\mu_0}^{(n)}
    \left[f\left(\sigma^{(n)}(t)\right)\right].
\]
For a deterministic initial configuration \(\sigma_0\), we abbreviate
\(\delta_{\sigma_0}S_t^{(n)}\) as \(\sigma_0S_t^{(n)}\).

\subsubsection{The Kantorovich--Rubinstein distance}
\label{subsec:kantorovich}

Let \((\mathcal E,d)\) be a complete separable metric space with its Borel
\(\sigma\)-algebra. For \(f:\mathcal E\to\mathbb R\), define its
\emph{Lipschitz seminorm} by
\[
    \|f\|_{\mathrm{Lip}}\defeq
    \sup_{\substack{x,y\in\mathcal E\\x\ne y}}
    \frac{|f(x)-f(y)|}{d(x,y)}.
\]
Let \(\mathcal P_1(\mathcal E)\) denote the set of probability measures
\(\mu\) on \(\mathcal E\) satisfying
$
    \int_{\mathcal E}d(x,x_0)\,\mu(dx)<\infty
$
for some, equivalently every, \(x_0\in\mathcal E\). For
\(\mu,\nu\in\mathcal P_1(\mathcal E)\), the \emph{Kantorovich--Rubinstein distance} is defined as
\[
    d_K(\mu,\nu)\defeq
    \sup_{\|f\|_{\mathrm{Lip}}\le1}
    \left|
        \int_{\mathcal E}f\,d\mu-\int_{\mathcal E}f\,d\nu
    \right|.
\]
This is also called the \(1\)-Wasserstein distance. A \emph{coupling} of \(\mu\) and \(\nu\) is a probability measure \(\pi\) on
\(\mathcal E\times\mathcal E\) with marginals \(\mu\) and \(\nu\). We denote the set of all couplings of $\mu$ and $\nu$ by \(\Pi(\mu,\nu)\). The Kantorovich--Rubinstein duality theorem states that
\begin{equation}
\label{KR-duality}
    d_K(\mu,\nu)=
    \inf_{\pi\in\Pi(\mu,\nu)}
    \int_{\mathcal E\times\mathcal E}d(x,y)\,\pi(dx,dy).
\end{equation}
See \cite{Villani2009optimaltransport} for more details on optimal transport. For integral probability metrics in general, we refer the readers to \cite{Mueller1997IntegralProbMetrics,Rachev1991ProbabilityMetrics,Zolotarev1983ProbabilityMetrics}. 
For random elements \(X,Y\) with values in \(\mathcal E\), we write
\[
    d_K(X,Y)\defeq d_K\left(\operatorname{Law}(X),\operatorname{Law}(Y)\right),
    \qquad
    d_K(X,\nu)\defeq d_K\left(\operatorname{Law}(X),\nu\right).
\]
In applications to random vectors in \(\mathbb R^Q\), \(d\) is the Euclidean distance.
Vector norms are Euclidean, matrix norms are induced operator norms,
and \(\|D^\ell f\|_\infty\) is the supremum of the induced
\(\ell\)-linear norm. Dependence on a bulk parameter is indicated explicitly.

\subsubsection{Distance to stationarity}
\label{subsec:distance}
For \(\sigma_0\in\Omega_n\), define the \emph{distance to stationarity} by
\[
    \mathscr D_n(t;\sigma_0)\defeq
    \dtv\left(
        \sigma_0S_t^{(n)},
        \mu_{\kappa_1,\kappa_2,\iota_1,\iota_2}^{(n)}
    \right),
\]
and define the \emph{worst-case distance to stationarity} by
\[
    \mathscr D_n(t)\defeq
    \sup_{\sigma_0\in\Omega_n}\mathscr D_n(t;\sigma_0).
\]
For \(\varepsilon_0\in(0,1/2)\) and all sufficiently large \(n\),
so that the admissible set is nonempty, define the
\emph{bulk distance to stationarity} by
\[
    \mathscr D_{n,\varepsilon_0}(t)\defeq
    \sup_{\substack{\sigma_0\in\Omega_n:\\
    \varepsilon_0\le M^{(n)}(\sigma_0)\le1-\varepsilon_0}}
    \mathscr D_n(t;\sigma_0).
\]
For \(\epsilon\in(0,1)\), define the \emph{\(\epsilon\)-mixing time} and the
\emph{bulk \(\epsilon\)-mixing time}, respectively, by
\[
    t_{\mathrm{mix}}^{(n)}(\epsilon)\defeq
    \inf\{t\ge0:\mathscr D_n(t)\le\epsilon\},
    \qquad
    t_{\mathrm{mix},\varepsilon_0}^{(n)}(\epsilon)\defeq
    \inf\{t\ge0:\mathscr D_{n,\varepsilon_0}(t)\le\epsilon\}.
\]
As is the convention, we use the standard abbreviation
$
    t_{\mathrm{mix}}^{(n)}\defeq t_{\mathrm{mix}}^{(n)}(1/4).
$

\subsubsection{Convergence and cutoff profiles}
\label{subsec:sharp_convergence}
Following the profile formulation in \cite{Lacoin_profile,gonccalves2021sharp}, let \(t_n>0\) and \(w_n>0\) satisfy \(w_n/t_n\to0\), and let \((\sigma_0^{(n)})_{n\ge1}\) be a sequence in $\Omega_n$. We say that this sequence has a
\emph{convergence profile} \(\mathcal P:\mathbb R\to[0,1]\) at times \(t_n\)
with windows \(w_n\) if
\[
    \mathscr D_n(t_n+sw_n;\sigma_0^{(n)})
    \longrightarrow\mathcal P(s),
    \qquad s\in\mathbb R.
\]
We say that the sequence of chains has a \emph{cutoff profile}
\(\mathcal P\) at times \(t_n\) with windows \(w_n\) if
\begin{align*}
    \mathscr D_n(t_n+sw_n)\longrightarrow\mathcal P(s),&
    \qquad s\in\mathbb R,\\
    \text{ and }\lim_{s\to-\infty}\mathcal P(s)=1,
    \qquad &
    \lim_{s\to+\infty}\mathcal P(s)=0.
\end{align*}
The bulk version is obtained by replacing \(\mathscr D_n\) with
\(\mathscr D_{n,\varepsilon_0}\). Thus, a convergence profile may depend on the chosen initial sequence, whereas a cutoff profile is obtained after optimising over the relevant class of initial configurations.

Let \(Z\sim\mathcal N(0,1)\), and let \(\Phi\) be its cumulative distribution function. Following \cite{gonccalves2021sharp},  define the \emph{Gaussian profile function}
\(\mathcal G:[0,\infty)\to[0,1)\) by
\[
    \mathcal G(r)\defeq
    \dtv\left(\mathcal N(r,1),\mathcal N(0,1)\right)
    =
    \frac12\mathbb E\left[\left|e^{rZ-r^2/2}-1\right|\right]
    =
    2\Phi\left(\frac r2\right)-1.
\]
The function \(\mathcal G\) is continuous and strictly increasing. Hence, for
every \(r>0\), the map \(s\mapsto\mathcal G(re^{-s})\) decreases continuously
from \(1\) to \(0\) as \(s\) goes from \(-\infty\) to \(+\infty\), and
\(\mathcal G^{-1}(\epsilon)\) is well-defined for \(\epsilon\in(0,1)\).

\section{Main results}
\label{sec:main_results}
Let us define the \emph{profile time scale} by
\begin{equation}
\label{def:profile_time}
    t^{(n)}(s)\defeq t_n+sw_n,\qquad
    t_n\defeq \frac{\kappa_1+\kappa_2}{2\kappa_1\kappa_2}\log n,\qquad
    w_n\defeq \frac{\kappa_1+\kappa_2}{\kappa_1\kappa_2}.
\end{equation}
Throughout the paper, write
\(\lambda_*\defeq\kappa_1\kappa_2/(\kappa_1+\kappa_2)\), so that
\(t_n=(2\lambda_*)^{-1}\log n\) and \(w_n=\lambda_*^{-1}\).

For \(m\in(0,1)\), define the \emph{initial-density profile amplitude} by
\begin{equation}
\label{def:initial_profile_amplitude}
    r(m)\defeq
    \frac{\alpha_*^{\,1-\alpha_*}(1-\alpha_*)^{\alpha_*}}
    {m^{\,1-\alpha_*}(1-m)^{\alpha_*}}
    \left[
        \frac{(1-m)\alpha_*^2+m(1-\alpha_*)^2}
        {\alpha_*(1-\alpha_*)}
    \right]^{1/2}\eqcomma
\end{equation}
with the \emph{equilibrium density} given by
$
    \alpha_*\defeq \frac{\kappa_1}{\kappa_1+\kappa_2}.
$
In the convergence-profile theorem, \(m\) is the limiting initial global
density:
$
    M^{(n)}\left(\sigma_0^{(n)}\right)\longrightarrow m.
$
For \(\varepsilon_0\in(0,1/2)\), define the
\emph{maximal bulk profile amplitude} by
\begin{equation}
\label{def:bulk_profile_amplitude}
    r_{\varepsilon_0}^*\defeq
    \max_{m\in[\varepsilon_0,1-\varepsilon_0]}r(m).
\end{equation}
The maximum exists because \(r\) is continuous on
\([\varepsilon_0,1-\varepsilon_0]\).

\subsection{Convergence and bulk cutoff profiles}
\label{convergence_profile}
The first theorem gives the profile from a deterministic bulk initial
configuration. The second maximizes it over the bulk class. We write down the following bulk initial condition assumption. 
\begin{assumption}[Bulk initial condition]
\label{assumption:bulk_initial_condition_tk}
Fix \(\varepsilon_0\in(0,1/2)\), and assume that
$
    \varepsilon_0
    \le
    M^{(n)}(\sigma_0)
    \le
    1-\varepsilon_0.
$
\end{assumption}

\begin{thm}[Convergence profile with deterministic bulk initial configuration]
\label{thm_convergence_profile}
Fix \(\varepsilon_0\in(0,1/2)\). Let
\((\sigma_0^{(n)})_{n\ge2}\), with
\(\sigma_0^{(n)}\in\Omega_n^*\), satisfy
Assumption~\ref{assumption:bulk_initial_condition_tk} for all
sufficiently large \(n\), and
assume that
$
    M^{(n)}\left(\sigma_0^{(n)}\right)
    \longrightarrow
    m\in[\varepsilon_0,1-\varepsilon_0].
$
Then, for every \(s\in\mathbb R\),
\begin{equation}
\label{eq:individual_convergence_profile}
    \mathscr D_n\left(
        t^{(n)}(s);
        \sigma_0^{(n)}
    \right)
    \longrightarrow
    \mathcal G\left(r(m)e^{-s}\right)\eqcomma
\end{equation}
as $n\to \infty$.  The convergence is uniform in \(s\) in compact subsets of
\(\mathbb R\).
\end{thm}
This theorem is proved in Section~\ref{proofs of cutoff profiles}.
Thus, the profile time and window are the same for all bulk initial sequences, whereas the profile amplitude depends on the limiting initial density \(m\). Taking the supremum over all bulk initial configurations gives the corresponding cutoff profile and the sharp bulk mixing-time asymptotics.

\begin{thm}[Bulk cutoff profile]
\label{cutoff_profile_tk}
Fix \(\varepsilon_0\in(0,1/2)\). Then, for every \(s\in\mathbb R\),
\begin{equation}
\label{eq:bulk_cutoff_profile_main_result}
    \mathscr D_{n,\varepsilon_0}
    \left(t^{(n)}(s)\right)
    \longrightarrow
    \mathcal G\left(r_{\varepsilon_0}^*e^{-s}\right).
\end{equation}
The convergence is uniform in \(s\) in compact subsets of
\(\mathbb R\). Moreover, for every \(\epsilon\in(0,1)\),
\begin{equation}
\label{mixing_time_result}
\begin{aligned}
t_{\mathrm{mix},\varepsilon_0}^{(n)}(\epsilon)
&=
\frac{\kappa_1+\kappa_2}
{2\kappa_1\kappa_2}\log n
+
\frac{\kappa_1+\kappa_2}
{\kappa_1\kappa_2}
\log\left(
    \frac{r_{\varepsilon_0}^*}
    {\mathcal G^{-1}(\epsilon)}
\right)
+o(1).
\end{aligned}
\end{equation}
In particular,
$
    t_{\mathrm{mix},\varepsilon_0}^{(n)}(\epsilon)
    =
    \frac{\kappa_1+\kappa_2}
    {2\kappa_1\kappa_2}\log n+O(1).
$
\end{thm}

The proofs are given in Section~\ref{proofs of cutoff profiles}. 

\subsection{Worst-case cutoff profile}
Boundary effects change both the leading constant in the logarithmic
mixing time and the cutoff profile. The following theorem is proved
in Section~\ref{sec:worst_case}.

\begin{thm}[Worst-case cutoff profile]
\label{thm:worst_case_cutoff}
Let
\[
 T_n\defeq\left(\frac{1}{\min\{\kappa_1,\kappa_2\}}
                    +\frac{1}{2\lambda_*}\right)\log n,
 \qquad w_n=\lambda_*^{-1}.
\]
Put $H_*\defeq\alpha_*^{1-\alpha_*}(1-\alpha_*)^{\alpha_*}$, and let
$W_0\sim\mathrm{Gamma}(\iota_1/\kappa_1,1)$ and
$W_1\sim\mathrm{Gamma}(\iota_2/\kappa_2,1)$. Define
\begin{equation}
\label{eq:boundary_profile_functions}
\begin{aligned}
 \mathcal P_0(s)&\defeq\sup_{z\in\mathbb R}
 \left\{\mathbb E\left[\Phi\left(z+
 \frac{\alpha_*H_*e^{-s}W_0^{-(1-\alpha_*)}}
      {\sqrt{\alpha_*(1-\alpha_*)}}\right)\right]-\Phi(z)\right\},\\
 \mathcal P_1(s)&\defeq\sup_{z\in\mathbb R}
 \left\{\mathbb E\left[\Phi\left(z+
 \frac{(1-\alpha_*)H_*e^{-s}W_1^{-\alpha_*}}
      {\sqrt{\alpha_*(1-\alpha_*)}}\right)\right]-\Phi(z)\right\}.
\end{aligned}
\end{equation}
For every fixed $s\in\mathbb R$,
\begin{equation}
\label{eq:worst_case_profile}
 \mathscr D_n(T_n+sw_n)\longrightarrow\mathcal P(s),
 \qquad
 \mathcal P(s)\defeq
 \begin{cases}
  \mathcal P_0(s),&\kappa_1<\kappa_2,\\
  \mathcal P_1(s),&\kappa_2<\kappa_1,\\
  \max\{\mathcal P_0(s),\mathcal P_1(s)\},&\kappa_1=\kappa_2.
 \end{cases}
\end{equation}
The convergence is uniform on compact intervals. The function
$\mathcal P$ is continuous and strictly decreasing, with limits one
and zero at $-\infty$ and $+\infty$, respectively. Consequently,
for every $\epsilon\in(0,1)$,
\begin{equation}
\label{eq:worst_case_mixing_asymptotic}
 t_{\mathrm{mix}}^{(n)}(\epsilon)
 =T_n+w_n\mathcal P^{-1}(\epsilon)+o(1).
\end{equation}
In particular, the cutoff window has order one, and
\begin{equation}
\label{eq:worst_case_cutoff_location}
 \mathscr D_n((1-\delta)T_n)\longrightarrow1,
 \qquad \mathscr D_n((1+\delta)T_n)\longrightarrow0
 \quad(0<\delta<1).
\end{equation}
If \(\kappa_1<\kappa_2\), the worst-case profile is attained
asymptotically from \(\vec{\mathbf0}\); if
\(\kappa_2<\kappa_1\), it is attained from \(\vec{\mathbf1}\).
When \(\kappa_1=\kappa_2\), the mutation parameters determine
the worst endpoint: it is \(\vec{\mathbf0}\) if
\(\iota_1<\iota_2\), and \(\vec{\mathbf1}\) if
\(\iota_2<\iota_1\).
If both \(\kappa_1=\kappa_2\) and \(\iota_1=\iota_2\),
both boundary configurations attain the same worst-case profile. Moreover, for every fixed $s\in\mathbb R$,
\begin{equation}
\label{eq:worst_case_on_bulk_window}
 \mathscr D_n(t_n+sw_n)\longrightarrow1.
\end{equation}
\end{thm}

\begin{remark}[Bulk and worst-case mixing]
The bulk and worst-case mixing times both have order $\log n$,
but their leading coefficients differ. 
Thus the boundary growth period contributes an additional
$\min\{\kappa_1,\kappa_2\}^{-1}\log n$ to the leading-order
mixing time. This separation has order $\log n$, whereas both
cutoff windows have order one.

By \eqref{eq:worst_case_on_bulk_window}, the worst-case distance
tends to one throughout every fixed bulk profile window.
Since
$t_{\mathrm{mix},\varepsilon_0}^{(n)}(\epsilon)=t_n+O(1)$,
time monotonicity further gives
 $\mathscr D_n\left(
 t_{\mathrm{mix},\varepsilon_0}^{(n)}(\epsilon)
 \right)\longrightarrow1$.
Hence, when all initial configurations with density in
$[\varepsilon_0,1-\varepsilon_0]$ have mixed to level $\epsilon$,
the worst initial configurations remain asymptotically at
total variation distance one from equilibrium.
The boundary effect therefore changes the leading-order cutoff
location, as well as the shape of the cutoff profile. 
\end{remark}

\subsection{The quantitative approximation for the Togashi--Kaneko model
with weak mutations}
\label{subsec:results_TK}

Recall the global density \(M^{(n)}\) and the global density process
\(M_t^{(n)}=M^{(n)}(\sigma^{(n)}(t))\) introduced in
Subsections~\ref{subsec:TK_model_def} and~\ref{subsec:notations}. To retain
the dependence on the initial configuration, fix
\(\sigma_0\in\Omega_n^*\), and define the \emph{initial blocks} and their
proportions by
\[
    A^i=A^i(\sigma_0)\defeq \{x\in\Lambda_n:(\sigma_0)_x=i\},
    \qquad
    a^i=a^i(\sigma_0)\defeq \frac{|A^i|}{n},
    \qquad i\in Q.
\]
Then \(a^0,a^1>0\), \(a^0+a^1=1\), and
$
    a^1=M^{(n)}(\sigma_0),
    a^0=1-M^{(n)}(\sigma_0).
$
Throughout this subsection, all quantitative estimates are stated under Assumption~\ref{assumption:bulk_initial_condition_tk}.

For \(i\in Q\), define the \emph{local density}
\(M^{(i,n)}:\Omega_n\to[0,1]\) by
\[
    M^{(i,n)}(\sigma)\defeq \frac{1}{a^in}\sum_{x\in A^i}\sigma_x,
    \qquad \sigma\in\Omega_n.
\]
We write
$
    M_t^{(i,n)}\defeq M^{(i,n)}(\sigma^{(n)}(t)),
    \vec M^{(n)}(\sigma)\defeq
    \left(M^{(0,n)}(\sigma),M^{(1,n)}(\sigma)\right),
    \vec a\defeq (a^0,a^1).
$
The global and local densities satisfy
    $M^{(n)}(\sigma)
    =
    \sum_{i\in Q}a^iM^{(i,n)}(\sigma)
    =
    \vec a\cdot\vec M^{(n)}(\sigma)$.
For \(\vec m=(m^0,m^1)\in[0,1]^2\), set
$
    m\defeq \vec a\cdot\vec m=a^0m^0+a^1m^1.
$
For \(i\in Q\), define the \emph{local drift} and
\emph{local carré-du-champ coefficient} by
\[
\begin{aligned}
F^{(i,n)}(\vec m)
&\defeq
(1-m^i)
\left[
    \kappa_1m(1-m)+\frac{\iota_1}{n}(1-m)
\right]
-
m^i
\left[
    \kappa_2m(1-m)+\frac{\iota_2}{n}m
\right],\\
G^{(i,n)}(\vec m)
&\defeq
(1-m^i)
\left[
    \kappa_1m(1-m)+\frac{\iota_1}{n}(1-m)
\right]
+
m^i
\left[
    \kappa_2m(1-m)+\frac{\iota_2}{n}m
\right].
\end{aligned}
\]
Let
$
    \vec {F^{(n)}}\defeq
    \left(F^{(0,n)},F^{(1,n)}\right),
    \vec {G^{(n)}}\defeq
    \left(G^{(0,n)},G^{(1,n)}\right).
$
A direct computation gives
\[
    L_nM^{(i,n)}
    =
    F^{(i,n)}(\vec M^{(n)}),
    \qquad
    \Gamma_nM^{(i,n)}
    =
    \frac{1}{a^in}G^{(i,n)}(\vec M^{(n)}).
\]
The functions \(\vec {F^{(n)}}\), \(\vec {G^{(n)}}\), and their derivatives up to order two are bounded uniformly in \(n\) on \([0,1]^2\), with bounds depending only on \(\kappa_1,\kappa_2,\iota_1,\iota_2\). 
We also introduce the local coefficients obtained by removing
the mutation terms proportional to \(n^{-1}\), with the block
weights \(\vec a\) held fixed. For
\(\vec m=(m^0,m^1)\in[0,1]^2\), recall that
\(m=\vec a\cdot\vec m\), and set, for \(i\in Q\),
\[
\begin{aligned}
 F^{(i,\infty)}(\vec m)
 &\defeq
 m(1-m)\bigl[\kappa_1(1-m^i)-\kappa_2m^i\bigr],\\
 G^{(i,\infty)}(\vec m)
 &\defeq
 m(1-m)\bigl[\kappa_1(1-m^i)+\kappa_2m^i\bigr].
\end{aligned}
\]
The corresponding vector-valued functions are
\[
 \vec F^{(\infty)}
 \defeq \bigl(F^{(0,\infty)},F^{(1,\infty)}\bigr),
 \qquad
 \vec G^{(\infty)}
 \defeq \bigl(G^{(0,\infty)},G^{(1,\infty)}\bigr).
\]

Now, define the \emph{deterministic local density}
    $\vec{\rho_t^{(n)}}
    \defeq
\left(\rho_t^{(0,n)},\rho_t^{(1,n)}\right)$, for $t\ge0$,
as the solution of the differential equation
\begin{equation}
\label{mean_field_ode}
\begin{cases}
\partial_t\rho_t^{(i,n)}
=
F^{(i,n)}(\vec{\rho_t^{(n)}}),
& i\in Q,\\
\rho_0^{(i,n)}
=
M^{(i,n)}(\sigma_0)=i,
& i\in Q.
\end{cases}
\end{equation}
The associated \emph{deterministic global density} is
$
    \rho_t^{(n)}
    \defeq
    \vec a\cdot\vec{\rho_t^{(n)}}
    =
    \sum_{i\in Q}a^i\rho_t^{(i,n)}.
$
Recall that
$
    \alpha_*=\frac{\kappa_1}{\kappa_1+\kappa_2}.
$
By the deterministic estimate \eqref{eq:sharp_bulk_deterministic_decay},
proved in Section~\ref{sec:cutoff_profile} using only ODE arguments,
under Assumption~\ref{assumption:bulk_initial_condition_tk} there exist constants \(C_{\varepsilon_0},c_{\varepsilon_0}>0\), depending only on \(\kappa_1,\kappa_2,\iota_1,\iota_2,\varepsilon_0\), such that, for all sufficiently large \(n\),
\[
    \left\|
        \vec{\rho_t^{(n)}}-(\alpha_*,\alpha_*)
    \right\|
    \le
    C_{\varepsilon_0}
    \left(
        e^{-c_{\varepsilon_0}t}+\frac1n
    \right),
    \qquad t\ge0.
\]
Define the \emph{global density fluctuation} and the
\emph{local density fluctuation vector} by
\[
    \xi_t^{(n)}
    \defeq
    \sqrt n\left(M_t^{(n)}-\rho_t^{(n)}\right),
    \qquad
    \xi_t^{(i,n)}
    \defeq
    \sqrt{a^in}
    \left(M_t^{(i,n)}-\rho_t^{(i,n)}\right),
\]
and
$
    \vec{\xi_t^{(n)}}
    \defeq
    \left(\xi_t^{(0,n)},\xi_t^{(1,n)}\right).
$
They satisfy the exact identity
$
    \xi_t^{(n)}
    =
    \sum_{i\in Q}\sqrt{a^i}\,\xi_t^{(i,n)}.
$
We use the same notation for the functions for $\eta \in \Omega_n$,
\[
 \xi_t^{(i,n)}(\eta)=\sqrt{a^in}
        \bigl(M^{(i,n)}(\eta)-\rho_t^{(i,n)}\bigr),\qquad
 \xi_t^{(n)}(\eta)=\sqrt n
        \bigl(M^{(n)}(\eta)-\rho_t^{(n)}\bigr).
\]
Under \(\nu_t^{(n)}\), we denote this local fluctuation
vector by \(\vec\zeta_t^{(n)}\).

The following estimate gives polynomial control of the local fluctuations and is the moment bound used in the quantitative diffusion approximation.

\begin{lemma}[A priori estimate for the local density fluctuations]
\label{lemma: a_priori_estimate_for_the_magnetization_vector}
Under Assumption~\ref{assumption:bulk_initial_condition_tk}, for all
sufficiently large \(n\), there exists a constant
$
    C_{\varepsilon_0}
    =
    C_{\varepsilon_0}
    (\kappa_1,\kappa_2,\iota_1,\iota_2)
$
such that, for every \(t\ge0\) and \(i\in Q\),
    \[\mathbb E_{\sigma_0}^{(n)}
    \left[
        \left(\xi_t^{(i,n)}\right)^2
    \right]
    \le
    C_{\varepsilon_0}t\eqcomma\]
so that
$
    \mathbb E_{\sigma_0}^{(n)}
    \left[
        \left(
            M_t^{(i,n)}-\rho_t^{(i,n)}
        \right)^2
    \right]
    \le
    \frac{C_{\varepsilon_0}t}{a^in}
    \le
    \frac{C_{\varepsilon_0}t}
    {\min_{j\in Q}a^jn}.
$
\end{lemma}
This lemma is proved in Section~\ref{Priori_est_Mag_Vec}.

We can now state the quantitative Gaussian approximation for the local
fluctuation vector.

\begin{prop}[Quantitative CLT for the local density fluctuation vector]
\label{prop:CLT_mag_tk}
Under Assumption~\ref{assumption:bulk_initial_condition_tk}, fix \(C_0>0\).
Let \(\vec Z_t^{(n)}\sim\mathcal N(0,\Gamma_t^{(n)})\), where
\(\Gamma_t^{(n)}\) is defined in \eqref{Gamma_tk_clt}.
For all sufficiently large \(n\) and every \(0\le t\le C_0\log n\),
\[
    d_K\left(
        \operatorname{Law}\left(
            \vec{\xi_t^{(n)}};
            \mathbb P_{\sigma_0}^{(n)}
        \right),
        \operatorname{Law}\left(\vec {Z_t^{(n)}}\right)
    \right)
    \le
    \frac{
        C_{\varepsilon_0,C_0}(t^{1/3}+t^2)
    }{
        \left(\min_{i\in Q}a^in\right)^{1/6}
    },
\]
where the constant
$
    C_{\varepsilon_0,C_0}
    =
    C_{\varepsilon_0,C_0}
    (\kappa_1,\kappa_2,\iota_1,\iota_2)
$
is independent of \(n,t\), and \(\sigma_0\).
\end{prop}
This proposition is proved in Section~\ref{sub:backwards_fokker-planck_equation}.

\begin{remark}[Choice of the logarithmic time window]
\label{remark:choice_C0_log_window_CLT}
The estimates hold on \(0\le t\le C_0\log n\) for every
positive constant \(C_0\) independent of \(n\).
For the subsequent mixing-time arguments, we choose \(C_0\)
sufficiently large. In particular, taking
 $C_0>
 \frac{2(\kappa_1+\kappa_2)}{\kappa_1\kappa_2}$
ensures that this interval contains the time \(4t_n\) used
in the upper bound on the bulk mixing time.
\end{remark}

\subsection{A CLT for the reference measure}
\label{subsec:results_aux_pro}

We now introduce the product reference measure to approximate the local fluctuations of the TK model. Recall the blocks \(A^i\), their relative sizes \(a^i\), and the deterministic local density \(\vec{\rho_t^{(n)}}\) defined in
Subsection~\ref{subsec:results_TK}. For \(t\ge0\), define the
\emph{reference measure}
\begin{align}
\label{reference_measure}
\nu_t^{(n)}
&\defeq
\bigotimes_{i\in Q}
\bigotimes_{x\in A^i}
\mathrm{Ber}\left(\rho_t^{(i,n)}\right),
\\
\intertext{so that, for every \(\eta\in\Omega_n\),}
\nu_t^{(n)}(\eta)
&=
\prod_{i\in Q}\prod_{x\in A^i}
\left(
    \eta_x\rho_t^{(i,n)}
    +(1-\eta_x)(1-\rho_t^{(i,n)})
\right).
\notag
\end{align}
Analogously, we also define the \emph{equilibrium reference measure}
\begin{equation}
\label{reference_measure_limit}
    \nu_*^{(n)}
    \defeq
    \bigotimes_{x\in\Lambda_n}\mathrm{Ber}(\alpha_*),
    \qquad
    \alpha_*=\frac{\kappa_1}{\kappa_1+\kappa_2}.
\end{equation}
Under \(\nu_t^{(n)}\), $a^inM^{(i,n)}\sim\mathrm{Bin}\left(a^in,\rho_t^{(i,n)}\right)$, for $i\in Q$.
Define the \emph{reference-measure local fluctuation vector} by
\[
    \zeta_t^{(i,n)}
    \defeq
    \sqrt{a^in}
    \left(
        M^{(i,n)}-\rho_t^{(i,n)}
    \right),
    \qquad
    \vec{\zeta_t^{(n)}}
    \defeq
    \left(
        \zeta_t^{(0,n)},\zeta_t^{(1,n)}
    \right).
\]
The following identity provides its exact second moment.

Under $\nu_t^{(n)}$, the block counts are independent with laws
$\mathrm{Bin}(a^in,\rho_t^{(i,n)})$. Consequently,
\begin{align}
\label{quadratic_fluctuation_estimates_under_reference_measure}
 \mathbb E_{\nu_t^{(n)}}[\zeta_t^{(i,n)}]=0,
 \qquad
 \mathbb E_{\nu_t^{(n)}}[(\zeta_t^{(i,n)})^2]
 =\rho_t^{(i,n)}(1-\rho_t^{(i,n)})\le\tfrac14.
\end{align}

Define the \emph{reference Gaussian vector} by
\[
    \vec{\bar Z_t^{(n)}}
    \sim
    \mathcal N\left(0,\bar\Gamma_t^{(n)}\right),
    \qquad
    \bar\Gamma_t^{(n)}
    \defeq
    \begin{pmatrix}
    \rho_t^{(0,n)}(1-\rho_t^{(0,n)}) & 0\\
    0 & \rho_t^{(1,n)}(1-\rho_t^{(1,n)})
    \end{pmatrix}.
\]

\begin{prop}[Quantitative CLT for the reference measure]
\label{prop:CLT_mag_GD}
For every $n\ge2$, $\sigma_0\in\Omega_n^*$, and $t\ge0$,
\[
 d_K\left(
  \operatorname{Law}(\vec{\zeta_t^{(n)}};\nu_t^{(n)}),
  \operatorname{Law}(\vec{\bar Z_t^{(n)}})
 \right)
 \le\frac32\sum_{i\in Q}\frac1{\sqrt{a^in}}
 \le\frac3{\sqrt{\min_{i\in Q}a^in}}.
\]
\end{prop}
The proof is given in Section~\ref{sec:GD_approx}.

\subsection{Approximation by the reference measure}
\label{results_approach}

We compare the TK dynamics with the reference measure
\(\nu_t^{(n)}\). For each initial configuration \(\sigma_0\), the blocks \(A^i\), the deterministic density \(\vec{\rho_t^{(n)}}\), the reference measure \(\nu_t^{(n)}\), and the fluctuation vectors
\(\vec{\xi_t^{(n)}}\) and \(\vec{\zeta_t^{(n)}}\) are those constructed from
\(\sigma_0\) in the preceding subsections. Recall that
$
    t_n=\frac{\kappa_1+\kappa_2}{2\kappa_1\kappa_2}\log n.
$
The first result compares the two local fluctuation vectors throughout the logarithmic window relevant for the cutoff analysis.

\begin{thm}[Approximation of the local fluctuation vector]
\label{mag_fluc_estimate_dk}
Fix \(\varepsilon_0\in(0,1/2)\) and
$
    C_0>\frac{\kappa_1+\kappa_2}{2\kappa_1\kappa_2}.
$
Then, as $n\to \infty$,
\[
\begin{aligned}
&\sup_{\substack{\sigma_0\in\Omega_n:\\
\varepsilon_0\le M^{(n)}(\sigma_0)\le1-\varepsilon_0}}
\ \sup_{t_n\le t\le C_0\log n}
d_K\left(
    \operatorname{Law}\left(
        \vec{\xi_t^{(n)}};
        \mathbb P_{\sigma_0}^{(n)}
    \right),
    \operatorname{Law}\left(
        \vec{\zeta_t^{(n)}};
        \nu_t^{(n)}
    \right)
\right)
\longrightarrow0\eqstop
\end{aligned}
\]

\end{thm}
This theorem is proved in Section~\ref{approach_tk_nu_t}.
Using Yau's relative entropy method, we can convert the fluctuation estimate in Kantorovich distance to an estimate in total variation distance.

\begin{thm}[Total variation approximation of the configuration law]
\label{convergence_in_total_variation}
Fix \(\varepsilon_0\in(0,1/2)\) and
 $C_0>\frac{\kappa_1+\kappa_2}{2\kappa_1\kappa_2}$.
For every fixed \(\delta>0\), independent of \(n\), we have
\[
\sup_{\substack{\sigma_0\in\Omega_n:\\
\varepsilon_0\le M^{(n)}(\sigma_0)\le1-\varepsilon_0}}
\ \sup_{t_n+\delta\le t\le C_0\log n}
\dtv\left(
    \sigma_0S_t^{(n)},
    \nu_t^{(n)}
\right)
\longrightarrow0
\qquad\text{as }n\to\infty.
\]
\end{thm}
This theorem is proved in Section~\ref{estimate_in_dtv}.
Together with Lemma~\ref{lemma:invariant_measure_to_product_bernoulli},
Theorem~\ref{convergence_in_total_variation} gives the following
upper bound on the bulk mixing time. The sharp asymptotics are
given in Theorem~\ref{cutoff_profile_tk}.

\begin{lemma}[Upper bound on the bulk mixing time]
\label{upper_bound_mixing_time}
Fix \(\varepsilon_0\in(0,1/2)\). Then,
\[
\begin{aligned}
&\sup_{\substack{\sigma_0\in\Omega_n:\\
\varepsilon_0\le M^{(n)}(\sigma_0)\le1-\varepsilon_0}}
\dtv\left(
    \sigma_0S_{4t_n}^{(n)},
    \mu_{\kappa_1,\kappa_2,\iota_1,\iota_2}^{(n)}
\right)
\longrightarrow0
\end{aligned}
\]
as \(n\to\infty\). Consequently, for every \(\epsilon\in(0,1)\), for all sufficiently large \(n\),
\[
    t_{\mathrm{mix},\varepsilon_0}^{(n)}(\epsilon)
    \le
    4t_n
    =
    \frac{2(\kappa_1+\kappa_2)}{\kappa_1\kappa_2}\log n\eqstop
\]

\end{lemma}
This lemma is proved in Section~\ref{estimate_in_dtv}.

\section{Quantitative approximations for the TK model
with weak mutations}
\label{sec:tk_approx}

In this section, we prove the polynomial moment estimate in
Lemma~\ref{lemma: a_priori_estimate_for_the_magnetization_vector} and the quantitative Gaussian approximation in Proposition~\ref{prop:CLT_mag_tk}. We prove the quantitative CLT using the semigroup approach to
Stein's method developed by Barbour~\cite{Barbour1990Diffusion}.
We first introduce a linear diffusion that approximates the local
fluctuation process and compare their time-dependent generators.
We then establish the global and local moment bounds needed to
control the comparison error. Finally, the backward Kolmogorov
equation and its regularity estimates yield a quantitative
Kantorovich bound. Since the approximating diffusion is Gaussian,
this gives the stated CLT. The full Taylor expansion of the Markov chain generator is provided in Appendix~\ref{append_com_tk}. Throughout this section, whenever a bulk hypothesis is required, we work under Assumption~\ref{assumption:bulk_initial_condition_tk}.

\subsection{Diffusion comparison framework}
\label{sub:diffusion_approximation_for_density_fluctuation_variables_tk}

We first introduce the linear diffusion that approximates
\(\vec{\xi_t^{(n)}}\). Let \((\vec{B_t};t\ge0)\) be a standard Brownian motion
in \(\mathbb R^Q\), and define for
 \(i,j\in Q\)
\[
\resizebox{\linewidth}{!}{$\displaystyle
    \beta_{ij}^{(n)}(t)
    \defeq
    \sqrt{\frac{a^i}{a^j}}\,
    \partial_jF^{(i,n)}
    \left(\vec{\rho_t^{(n)}}\right),
    \qquad
    \Sigma^{(n)}(t)
    \defeq
    \begin{pmatrix}
        \sqrt{G^{(0,n)}
        \left(\vec{\rho_t^{(n)}}\right)}
        & 0\\
        0 &
        \sqrt{G^{(1,n)}
        \left(\vec{\rho_t^{(n)}}\right)}
    \end{pmatrix}.
$}
\]
The \emph{approximating diffusion} is the solution of
\begin{equation}
\label{solution_sde_y}
    \dd\vec {Y_t^{(n)}}
    =
    \beta^{(n)}(t)\vec {Y_t^{(n)}}\,\dd t
    +
    \Sigma^{(n)}(t)\,\dd\vec{B_t},
    \qquad
    \vec Y_0^{(n)}=0.
\end{equation}
Its time-dependent generator is
\begin{equation}
\label{diffusion_operator_bkp}
    \mathscr L_t^{(n)}f(\vec x)
    \defeq
    \sum_{i,j\in Q}
    \beta_{ij}^{(n)}(t)x^j\partial_i f(\vec x)
    +
    \frac12
    \sum_{i\in Q}
    G^{(i,n)}
    \left(\vec{\rho_t^{(n)}}\right)
    \partial_i^2f(\vec x).
\end{equation}
Let \(\Phi^{(n)}(t,s)\) be the fundamental matrix associated with
\(\beta^{(n)}\):
\[
    \partial_t\Phi^{(n)}(t,s)
    =
    \beta^{(n)}(t)\Phi^{(n)}(t,s),
    \qquad
    \Phi^{(n)}(s,s)=I.
\]
For \eqref{solution_sde_y}, we can solve for $\vec {Y_t^{(n)}}$ explicitly 
\begin{equation}
\label{Y_tk_clt_representation}
    \vec {Y_t^{(n)}}
    =
    \int_0^t
    \Phi^{(n)}(t,s)\Sigma^{(n)}(s)\,\dd\vec B_s.
\end{equation}
Consequently, \(\vec {Y_t^{(n)}}\) is a centred Gaussian process with covariance
\begin{equation}
\label{Gamma_tk_clt}
    \Gamma_t^{(n)}
    \defeq
    \int_0^t
    \Phi^{(n)}(t,s)
    \Sigma^{(n)}(s)\Sigma^{(n)}(s)^\intercal
    \Phi^{(n)}(t,s)^\intercal\,\dd s.
\end{equation}
More generally, for \(0\le r\le t\),
\[
    \operatorname{Cov}
    \left(
        \vec {Y_t^{(n)}},\vec {Y_r^{(n)}}
    \right)
    =
    \int_0^r
    \Phi^{(n)}(t,s)
    \Sigma^{(n)}(s)\Sigma^{(n)}(s)^\intercal
    \Phi^{(n)}(r,s)^\intercal\,\dd s.
\]

The following result is the quantitative diffusion
approximation.

\begin{lemma}[Quantitative approximation of the local fluctuation vector by the diffusion]
\label{lemma:quantitative_convergence_xi_yt}
Under Assumption~\ref{assumption:bulk_initial_condition_tk}, fix \(C_0>0\).
Then, for all sufficiently large \(n\) and every \(0\le t\le C_0\log n\),
\[
    d_K\left(
        \operatorname{Law}
        \left(
            \vec{\xi_t^{(n)}};
            \mathbb P_{\sigma_0}^{(n)}
        \right),
        \operatorname{Law}\left(\vec {Y_t^{(n)}}\right)
    \right)
    \le
    \frac{
        C_{\varepsilon_0,C_0}(t^{1/3}+t^2)
    }{
        \left(\min_{i\in Q}a^in\right)^{1/6}
    }.
\]
Here \(C_{\varepsilon_0,C_0}\) depends only on
\(\kappa_1,\kappa_2,\iota_1,\iota_2,\varepsilon_0,C_0\).
\end{lemma}

The comparison begins at the generator level. Lemma~\ref{routine_computation_3} gives, for every
\(g\in C_b^3(\mathbb R^Q;\mathbb R)\),
\begin{equation}
\label{first_estimate_of_magnetization_vector}
\begin{aligned}
&\left|
\mathbb E_{\sigma_0}^{(n)}
\left[
    (\partial_t+L_n)
    g\left(\vec{\xi_t^{(n)}}\right)
\right]
-
\mathbb E_{\sigma_0}^{(n)}
\left[
    \mathscr L_t^{(n)}
    g\left(\vec{\xi_t^{(n)}}\right)
\right]
\right|\\
&\quad\le
\frac{C}
{\left(\min_{i\in Q}a^in\right)^{1/2}}
\left(
    \|Dg\|_\infty
    \mathbb E_{\sigma_0}^{(n)}
    \left[
        \left\|\vec{\xi_t^{(n)}}\right\|^2
    \right]
    +
    \|D^2g\|_\infty
    \mathbb E_{\sigma_0}^{(n)}
    \left[
        \left\|\vec{\xi_t^{(n)}}\right\|
    \right]
    +
    \|D^3g\|_\infty
\right),
\end{aligned}
\end{equation}
where \(C=C(\kappa_1,\kappa_2,\iota_1,\iota_2)\). The right-hand side of \eqref{first_estimate_of_magnetization_vector} is controlled by moments of the local fluctuation vector. We establish those
bounds next, beginning with the scalar global density and then returning to the local coordinates.

\subsection{Polynomial moment estimates}
\label{Priori_est_Mag_Vec}

We now prove the moment estimates needed in
\eqref{first_estimate_of_magnetization_vector}.  
The argument is organized in four steps. First, the deterministic global density enters a fixed contractive interval in a uniformly bounded time. Second, a rough estimate controls the global fluctuation before this entrance time. Third, a stopping-time argument uses the contractivity of the drift to obtain a linear-in-time global moment bound.
Finally, we write the local drift difference in terms of the
local and global density errors, and use the global second-moment
bound to control each local fluctuation. 

Define the \emph{global drift} and
\emph{global carré-du-champ coefficient} by
\begin{align}
\label{def:global_F_G_n}
F^{(n)}(m)
&\defeq
(1-m)
\left[
    \kappa_1m(1-m)+\frac{\iota_1}{n}(1-m)
\right]
-
m
\left[
    \kappa_2m(1-m)+\frac{\iota_2}{n}m
\right],\\
G^{(n)}(m)
&\defeq
(1-m)
\left[
    \kappa_1m(1-m)+\frac{\iota_1}{n}(1-m)
\right]
+
m
\left[
    \kappa_2m(1-m)+\frac{\iota_2}{n}m
\right].
\end{align}
Then, a direct computation yields
    $L_nM^{(n)}
    =
    F^{(n)}\left(M^{(n)}\right)$, and
    $\Gamma_nM^{(n)}
    =
    \frac1nG^{(n)}\left(M^{(n)}\right)$.
Moreover, for every \(\vec m\in[0,1]^2\),
\begin{equation}
\label{eq:global_local_FG_relation}
    F^{(n)}(\vec a\cdot\vec m)
    =
    \sum_{i\in Q}a^iF^{(i,n)}(\vec m),
    \qquad
    G^{(n)}(\vec a\cdot\vec m)
    =
    \sum_{i\in Q}a^iG^{(i,n)}(\vec m).
\end{equation}
Let
$
    F^{(\infty)}(m)
    \defeq
    m(1-m)
    \left[
        \kappa_1-(\kappa_1+\kappa_2)m
    \right].
$
The two critical points of \(F^{(\infty)}\) surrounding \(\alpha_*\) are given by
    $x_\pm
    \defeq
    \frac{
        2\kappa_1+\kappa_2
        \pm
        \sqrt{
            \kappa_1^2+\kappa_1\kappa_2+\kappa_2^2
        }
    }{
        3(\kappa_1+\kappa_2)
    }$
and satisfy
\[
    0<x_-<\alpha_*<x_+<1,\qquad
    (F^{(\infty)})'(m)<0\quad\text{for }m\in(x_-,x_+).
\]
Set
$
    \delta_0
    \defeq
    \frac12
    \min\{\alpha_*-x_-,x_+-\alpha_*\},
$
and define
\[
    \mathcal I\defeq [\alpha_*-\delta_0,\alpha_*+\delta_0],\qquad
    \mathcal J\defeq\left[\alpha_*-\frac{\delta_0}{2},
                  \alpha_*+\frac{\delta_0}{2}\right].
\]
Then,
$
    \mathcal J\subset \mathcal I\subset(x_-,x_+).
$
The deterministic global density is equivalently characterised as the
solution of
\begin{equation}
\label{mean_field_ode_global}
\begin{cases}
\partial_t\rho_t^{(n)}
=
F^{(n)}\left(\rho_t^{(n)}\right),\\
\rho_0^{(n)}
=
M^{(n)}(\sigma_0).
\end{cases}
\end{equation}
Indeed, \eqref{eq:global_local_FG_relation} and uniqueness of solutions give
$
    \rho_t^{(n)}
    =
    \sum_{i\in Q}a^i\rho_t^{(i,n)}.
$
Define the \emph{deterministic entrance time}
$
    t_{\mathrm{ent}}^{(n)}
    \defeq
    \inf\{t\ge0:\rho_t^{(n)}\in \mathcal J\}.
$
We have the following lemma.

\begin{lemma}[Uniform bound on the deterministic entrance time]
\label{lemma: deterministic_entrance_time_bulk}
Under Assumption~\ref{assumption:bulk_initial_condition_tk},
there exist \(T_{\varepsilon_0}<\infty\) and
\(n_{\varepsilon_0}\in\mathbb N\), depending only on
\(\kappa_1,\kappa_2,\iota_1,\iota_2,\varepsilon_0\), such that,
for every \(n\ge n_{\varepsilon_0}\),
\[
 t_{\mathrm{ent}}^{(n)}\le T_{\varepsilon_0},
 \qquad
 \rho_t^{(n)}\in \mathcal J
 \quad\text{for all }t\ge t_{\mathrm{ent}}^{(n)},
\]
uniformly over bulk initial configurations.
More generally, for such \(n\), every solution of
\(x'(t)=F^{(n)}(x(t))\) satisfies
\(x(s)\in \mathcal J\Rightarrow x(t)\in \mathcal J\) whenever \(t\ge s\ge0\).
\end{lemma}

\begin{proof}[Proof of \Cref{lemma: deterministic_entrance_time_bulk}]
Write \(\ell=\alpha_*-\delta_0/2\) and
\(r=\alpha_*+\delta_0/2\), so that \(\mathcal J=[\ell,r]\).
Since \(F^{(\infty)}(\ell)>0\), \(F^{(\infty)}(r)<0\), and
\(F^{(n)}\to F^{(\infty)}\) uniformly on \([0,1]\), we have
\(F^{(n)}(\ell)>0\) and \(F^{(n)}(r)<0\) for all sufficiently
large \(n\). These signs prevent a solution from crossing either
endpoint outward, proving that
\(x(s)\in \mathcal J\Rightarrow x(t)\in \mathcal J\) for \(t\ge s\).

If \(\varepsilon_0<\ell\), let
\(\gamma_-=\frac12\min_{m\in[\varepsilon_0,\ell]}
F^{(\infty)}(m)>0\).
Uniform convergence gives \(F^{(n)}\ge\gamma_-\) on this interval
for all sufficiently large \(n\).
For a bulk initial density \(\rho_0^{(n)}<\ell\), integration
of the ODE therefore gives
\(\rho_t^{(n)}\ge\rho_0^{(n)}+\gamma_-t\) before entrance, and hence
\[
 t_{\mathrm{ent}}^{(n)}
 \le\frac{\ell-\rho_0^{(n)}}{\gamma_-}
 \le\frac{\ell-\varepsilon_0}{\gamma_-}
 \defeq T_-.
\]
Similarly, if \(r<1-\varepsilon_0\), applying the same argument
to \(-F^{(n)}\) on \([r,1-\varepsilon_0]\), with
\(\gamma_+=\frac12\min_{m\in[r,1-\varepsilon_0]}
(-F^{(\infty)}(m))>0\), gives
\(t_{\mathrm{ent}}^{(n)}\le
T_+\defeq(1-\varepsilon_0-r)/\gamma_+\)
whenever \(\rho_0^{(n)}>r\).

Set \(T_-=0\) if \(\varepsilon_0\ge\ell\), and \(T_+=0\) if
\(r\ge1-\varepsilon_0\), since the corresponding starting case
cannot occur. If \(\rho_0^{(n)}\in \mathcal J\), the entrance time is zero.
Thus \(T_{\varepsilon_0}=\max\{T_-,T_+\}\) gives the required
uniform bound, after choosing \(n_{\varepsilon_0}\) large enough
for all preceding inequalities to hold.
Since the trajectory is continuous and $\mathcal J$ is closed,
$\rho_{t_{\mathrm{ent}}^{(n)}}^{(n)}\in \mathcal J$.
\end{proof}

We next establish a fluctuation bound valid from time zero
without the bulk assumption. This bound will control the process
before the deterministic entrance time.

\begin{lemma}
\label{lemma: priori_estimate_for_the_global_density}
For every \(\sigma_0\in\Omega_n\), \(n\in\mathbb N\), and \(t\ge0\),
    $\mathbb E_{\sigma_0}^{(n)}
    \left[
        \left(\xi_t^{(n)}\right)^2
    \right]
    \le
    C(e^{Ct}-1)$,
where \(C=C(\kappa_1,\kappa_2,\iota_1,\iota_2)\).
\end{lemma}

\begin{proof}[Proof of \Cref{lemma: priori_estimate_for_the_global_density}]
Regard
$
    \xi_t^{(n)}(\sigma)
    \defeq
    \sqrt n
    \left(
        M^{(n)}(\sigma)-\rho_t^{(n)}
    \right)
$
as a time-dependent function of \(\sigma\). Since
$
    \partial_t\rho_t^{(n)}
    =
    F^{(n)}\left(\rho_t^{(n)}\right),
$
the carré-du-champ identity gives
\[
\begin{aligned}
(\partial_t+L_n)
\left[
    \left(\xi_t^{(n)}\right)^2
\right](\sigma)
&=
2n
\left(
    M^{(n)}(\sigma)-\rho_t^{(n)}
\right)
\left(
    F^{(n)}\left(M^{(n)}(\sigma)\right)
    -
    F^{(n)}\left(\rho_t^{(n)}\right)
\right)\\
&\quad+
G^{(n)}\left(M^{(n)}(\sigma)\right).
\end{aligned}
\]
The mean-value theorem and the uniform bounds on
\((F^{(n)})'\) and \(G^{(n)}\) imply
    $(\partial_t+L_n)
    \left[
        \left(\xi_t^{(n)}\right)^2
    \right]
    \le
    C\left(\xi_t^{(n)}\right)^2+C$.
Since \(\xi_0^{(n)}=0\), Dynkin's formula and Gr\"onwall's inequality yield the required estimate.

\end{proof}

We can now combine the bounded entrance time with the contractive drift on \(\mathcal I\). The next stopping-time argument improves the preceding exponential bound to the linear estimate needed on logarithmic time scales.

\begin{lemma}[Linear bound for global density fluctuation under bulk initial density]
\label{lemma: priori_estimate_for_the_global_density_II}
Under Assumption~\ref{assumption:bulk_initial_condition_tk},
for all sufficiently large \(n\),
 $\mathbb E_{\sigma_0}^{(n)}
 \left[\left(\xi_t^{(n)}\right)^2\right]
 \le C_{\varepsilon_0}t$, with $t\ge0$,
uniformly over bulk initial configurations.
\end{lemma}

\begin{proof}[Proof of
\Cref{lemma: priori_estimate_for_the_global_density_II}]
Write \(T=t_{\mathrm{ent}}^{(n)}\).
By Lemma~\ref{lemma: deterministic_entrance_time_bulk},
\(T\le T_{\varepsilon_0}\).
On this bounded time interval,
Lemma~\ref{lemma: priori_estimate_for_the_global_density}
gives \(C(e^{Ct}-1)\le C_{\varepsilon_0}t\), proving the
claim for \(0\le t\le T\) and, in particular,
$\mathbb E_{\sigma_0}^{(n)}
 \left[\left(\xi_T^{(n)}\right)^2\right]
 \le C_{\varepsilon_0}T$.
 
Since \(\mathcal I\subset(x_-,x_+)\) and
\(F^{(n)}\to F^{(\infty)}\) in \(C^1([0,1])\), there exists
\(c_0>0\) such that \((F^{(n)})'\le-c_0\) on \(\mathcal I\) for all
sufficiently large \(n\).
Set \(\tau=\inf\{u\ge T:M_u^{(n)}\notin \mathcal I\}\).
For \(T\le u<\tau\), both \(M_u^{(n)}\) and
\(\rho_u^{(n)}\) belong to \(\mathcal I\), so the mean value theorem
and boundedness of \(G^{(n)}\) give
$(\partial_u+L_n)\left(\xi_u^{(n)}\right)^2
 \le -2c_0\left(\xi_u^{(n)}\right)^2+C$.
 
Applying Dynkin's formula up to \(t\wedge\tau\) and dropping
the nonpositive term yields
\begin{equation}
\label{eq:bulk_stopped_xi_bound}
\begin{aligned}
 \mathbb E_{\sigma_0}^{(n)}
 \left[\left(\xi_{t\wedge\tau}^{(n)}\right)^2\right]
 &\le
 \mathbb E_{\sigma_0}^{(n)}
 \left[\left(\xi_T^{(n)}\right)^2\right]+C(t-T)\le C_{\varepsilon_0}t,
 \qquad t\ge T.
\end{aligned}
\end{equation}

Put \(d=\operatorname{dist}(\mathcal J,\mathcal I^c)>0\).
On \(\{\tau\le t\}\), we have \(M_\tau^{(n)}\notin \mathcal I\)
and \(\rho_\tau^{(n)}\in \mathcal J\), hence
\(\left(\xi_{t\wedge\tau}^{(n)}\right)^2\ge nd^2\).
Therefore \eqref{eq:bulk_stopped_xi_bound} implies
 $\mathbb P_{\sigma_0}^{(n)}(\tau\le t)
 \le \frac{1}{nd^2}
 \mathbb E_{\sigma_0}^{(n)}
 \left[\left(\xi_{t\wedge\tau}^{(n)}\right)^2\right]
 \le \frac{C_{\varepsilon_0}t}{n}$.
Finally, \(M_t^{(n)},\rho_t^{(n)}\in[0,1]\) implies
\(\left(\xi_t^{(n)}\right)^2\le n\), so
\[
\begin{aligned}
 \mathbb E_{\sigma_0}^{(n)}
 \left[\left(\xi_t^{(n)}\right)^2\right]
 &\le
 \mathbb E_{\sigma_0}^{(n)}
 \left[\left(\xi_{t\wedge\tau}^{(n)}\right)^2\right]
 +n\mathbb P_{\sigma_0}^{(n)}(\tau\le t)\le C_{\varepsilon_0}t,
 \qquad t\ge T.
\end{aligned}
\]
Together with the bound for \(t\le T\), this proves the lemma.
\end{proof}
Now we are ready to prove a similar linear bound for local density fluctuations. 

\begin{proof}[Proof of
\Cref{lemma: a_priori_estimate_for_the_magnetization_vector}]
Put
$\delta_{\varepsilon_0}
\defeq \frac12\min\{\varepsilon_0,\alpha_*,1-\alpha_*\}$.
For all sufficiently large \(n\), the drift \(F^{(n)}\) is
positive at \(\delta_{\varepsilon_0}\) and negative at
\(1-\delta_{\varepsilon_0}\). Since the initial density lies
between these endpoints, it follows that
\begin{equation}
\label{eq:rho_global_bulk_bound}
 \delta_{\varepsilon_0}\le\rho_t^{(n)}
 \le1-\delta_{\varepsilon_0},
 \qquad t\ge0.
\end{equation}

The local drift difference admits the exact decomposition
\begin{equation}
\label{eq:local_F_difference_decomposition}
\begin{aligned}
 F^{(i,n)}\left(\vec M_t^{(n)}\right)
 -F^{(i,n)}\left(\vec{\rho_t^{(n)}}\right)
 &=-A_t^{(n)}
 \left(M_t^{(i,n)}-\rho_t^{(i,n)}\right)+B_t^{(i,n)}
 \left(M_t^{(n)}-\rho_t^{(n)}\right),
\end{aligned}
\end{equation}
where
\begin{align}
\label{eq:A_t_def}
 A_t^{(n)}
 &\defeq
 (\kappa_1+\kappa_2)\rho_t^{(n)}
 \left(1-\rho_t^{(n)}\right)
 +\frac{\iota_1(1-\rho_t^{(n)})
          +\iota_2\rho_t^{(n)}}{n},\\
\label{eq:B_t_def}
 B_t^{(i,n)}
 &\defeq
 \left[\kappa_1-(\kappa_1+\kappa_2)M_t^{(i,n)}\right]
 \left[1-M_t^{(n)}-\rho_t^{(n)}\right]-
 \frac{\iota_1(1-M_t^{(i,n)})
       +\iota_2M_t^{(i,n)}}{n}.
\end{align}
By \eqref{eq:rho_global_bulk_bound},
\(A_t^{(n)}\ge c_{\varepsilon_0}>0\), while
\(\lvert B_t^{(i,n)}\rvert\) and
\(G^{(i,n)}(\vec M_t^{(n)})\) are uniformly bounded.

Applying the time-dependent generator to
\((\xi_t^{(i,n)})^2\), using
\eqref{eq:local_F_difference_decomposition}, and then applying
Young's inequality
\(2|xy|\le\varepsilon x^2+\varepsilon^{-1}y^2\)
with a suitable \(\varepsilon>0\), we obtain
\begin{align}
\label{eq:local_xi_generator_bound_pre_young}
(\partial_t+L_n)\left[(\xi_t^{(i,n)})^2\right]
 &=-2A_t^{(n)}(\xi_t^{(i,n)})^2
   +2\sqrt{a^i}\,B_t^{(i,n)}      \xi_t^{(i,n)}\xi_t^{(n)}+G^{(i,n)}\left(\vec M_t^{(n)}\right)\\
\label{eq:local_xi_generator_bound}
 &\le-c_{\varepsilon_0}(\xi_t^{(i,n)})^2
      +C_{\varepsilon_0}(\xi_t^{(n)})^2
      +C_{\varepsilon_0},
\end{align}
where the state functions are evaluated at
\(\sigma^{(n)}(t)\).

Set
\(y_i(t)\defeq
\mathbb E_{\sigma_0}^{(n)}[(\xi_t^{(i,n)})^2]\).
Taking expectations and applying
Lemma~\ref{lemma: priori_estimate_for_the_global_density_II},
we obtain
\begin{equation}
\label{eq:local_xi_moment_differential_ineq}
 y_i'(t)\le-c_{\varepsilon_0}y_i(t)
              +C_{\varepsilon_0}(1+t),
 \qquad y_i(0)=0.
\end{equation}
Since $\xi_0^{(i,n)}=0$, multiplication by $e^{c_{\varepsilon_0}t}$ and integration yield
\[
\begin{aligned}
 y_i(t)
 \le C_{\varepsilon_0}
 \int_0^t e^{-c_{\varepsilon_0}(t-s)}(1+s)\,\dd s\le C_{\varepsilon_0}
 \left(t+t\int_0^t
 e^{-c_{\varepsilon_0}(t-s)}\,\dd s\right)
 \le C_{\varepsilon_0}(1+c_{\varepsilon_0}^{-1})t.
\end{aligned}
\]
Absorbing the final factor into the constant proves the bound
for all \(t\ge0\). The corresponding estimates for
\(M_t^{(i,n)}-\rho_t^{(i,n)}\) follow by dividing by \(a^in\).
\end{proof}

\subsection{Backward Kolmogorov equation and quantitative Gaussian approximation}
\label{sub:backwards_fokker-planck_equation}

With the moment bounds established, we return to the comparison problem. The backward Kolmogorov equation converts the difference of expectations into the time integral of the generator error, while the regularity estimate transfers derivatives of the backward solution to derivatives of its terminal condition.

For \(f\in C_b^3(\mathbb R^Q;\mathbb R)\), let
\((g_{s,t};0\le s\le t)\) solve the \emph{backward Kolmogorov equation}
\begin{equation*}
\label{def:sol_fokker_planck_eq}
\begin{cases}
(\partial_s+\mathscr L_s^{(n)})g_{s,t}=0,
&0\le s<t,\\
g_{t,t}=f.
\end{cases}
\tag{bKE}
\end{equation*}
$\mathscr L_s^{(n)}$ is defined in \eqref{diffusion_operator_bkp}.
The probabilistic representation gives
$
    g_{s,t}(\vec x)
    =
    \mathbb E
    \left[
        f\left(
            \vec {Y_t^{(n),s,\vec x}}
        \right)
    \right],
$
where \(\vec {Y_t^{(n),s,\vec x}}\) is the diffusion started from \(\vec x\) at
time \(s\). In particular,
$ \mathbb E
    \left[
        f\left(\vec {Y_t^{(n)}}\right)
    \right]
    =
    g_{0,t}(0).
$
For fixed \(t\) and \(\sigma\), we have
\[
\begin{aligned}
 \partial_s\left[g_{s,t}
 \left(\vec\xi_s^{(n)}(\sigma)\right)\right]
 =(\partial_s g_{s,t})
 \left(\vec\xi_s^{(n)}(\sigma)\right)-\sum_{i\in Q}\sqrt{a^in}\,
 \dot\rho_s^{(i,n)}\,\partial_i g_{s,t}
 \left(\vec\xi_s^{(n)}(\sigma)\right).
\end{aligned}
\]
Thus the explicit
time derivatives of \(g_{s,t}\) cancel.
Dynkin's formula and \eqref{def:sol_fokker_planck_eq} give
\begin{equation}
\label{bound_by_f}
\begin{aligned}
&\left|
\mathbb E_{\sigma_0}^{(n)}
\left[
    f\left(\vec{\xi_t^{(n)}}\right)
\right]
-
\mathbb E
\left[
    f\left(\vec {Y_t^{(n)}}\right)
\right]
\right|\\
&\qquad=
\left|
\mathbb E_{\sigma_0}^{(n)}
\left[
    \int_0^t
    \left[
        (\partial_s+L_n)
        g_{s,t}\left(\vec{\xi_s^{(n)}}\right)
        -
        \bigl((\partial_s+\mathscr L_s^{(n)})g_{s,t}\bigr)
        \left(\vec{\xi_s^{(n)}}\right)
    \right]\dd s
\right]
\right|.
\end{aligned}
\end{equation}

\subsubsection{Regularity of the backward Kolmogorov equation}
\label{sub:regularity_backwards-fokker-planck_equation}

To use \eqref{bound_by_f}, we need derivative bounds that are uniform over the
whole time interval. The uniform stability of the deterministic flow implies that, under the bulk
condition, for all
sufficiently large \(n\), the fundamental matrix satisfies
\begin{equation}
\label{fundamental_matrix_bound_bfp}
    \left\|
        \Phi^{(n)}(t,s)
    \right\|
    \le
    C_{\varepsilon_0}
    e^{-c_{\varepsilon_0}(t-s)},
    \qquad
    0\le s\le t.
\end{equation}
This is verified in Appendix~\ref{appendix_reg_bFP}.

\begin{lemma}[Regularity estimate for
\eqref{def:sol_fokker_planck_eq}]
\label{lemma:reg_bFP}
Let
\(k\in\mathbb N_0\).
Under Assumption~\ref{assumption:bulk_initial_condition_tk}, for all
sufficiently large \(n\) and for every
\(f\in C_b^k(\mathbb R^Q;\mathbb R)\), the solution of
\eqref{def:sol_fokker_planck_eq} satisfies
\[
    \left\|D^\ell g_{s,t}\right\|_\infty
    \le
    C_{\varepsilon_0,k}
    \left\|D^\ell f\right\|_\infty,
    \qquad
    0\le s\le t,\quad \ell=0,\ldots,k.
\]
\end{lemma}

The proof is provided in Appendix~\ref{appendix_reg_bFP}.

\subsubsection{Proof of the quantitative diffusion approximation}
\label{sub:quantitive_approxiamation_to_diffusion_process}

With the generator comparison, the moment estimates, and the backward
regularity bound in place, we can now prove the quantitative approximation.
\begin{proof}[Proof of \Cref{lemma:quantitative_convergence_xi_yt}]
The case \(t=0\) is immediate, since both fluctuation vectors vanish.
Fix \(0<t\le C_0\log n\), let
\(f\in C_b^\infty(\mathbb R^Q;\mathbb R)\), and let \(g_{s,t}\) solve
\eqref{def:sol_fokker_planck_eq}. By \eqref{bound_by_f},
Lemma~\ref{routine_computation_3}, and Lemma~\ref{lemma:reg_bFP},
\[
\begin{aligned}
&\left|
\mathbb E_{\sigma_0}^{(n)}
\left[f\left(\vec{\xi_t^{(n)}}\right)\right]
-\mathbb E\left[f\left(\vec{Y_t^{(n)}}\right)\right]
\right|\\
&\qquad\le
\frac{C_{\varepsilon_0}}
     {\left(\min_{i\in Q}a^in\right)^{1/2}}
\int_0^t
\left[\begin{aligned}
    &\|Df\|_\infty
    \mathbb E_{\sigma_0}^{(n)}
    \left[\left\|\vec{\xi_s^{(n)}}\right\|^2\right]
    +\|D^2f\|_\infty
    \mathbb E_{\sigma_0}^{(n)}
    \left[\left\|\vec{\xi_s^{(n)}}\right\|\right]\\
    &\qquad+\|D^3f\|_\infty
\end{aligned}\right]\dd s.
\end{aligned}
\]
Lemma~\ref{lemma: a_priori_estimate_for_the_magnetization_vector}
and Cauchy--Schwarz give
\[
    \mathbb E_{\sigma_0}^{(n)}
    \left[\left\|\vec{\xi_s^{(n)}}\right\|^2\right]
    \le C_{\varepsilon_0}s,
    \qquad
    \mathbb E_{\sigma_0}^{(n)}
    \left[\left\|\vec{\xi_s^{(n)}}\right\|\right]
    \le C_{\varepsilon_0}s^{1/2}.
\]
Consequently,
\begin{equation}
\label{smooth_test_bound_final}
\begin{aligned}
&\left|
\mathbb E_{\sigma_0}^{(n)}
\left[f\left(\vec{\xi_t^{(n)}}\right)\right]
-\mathbb E\left[f\left(\vec{Y_t^{(n)}}\right)\right]
\right|\\
&\qquad\le
\frac{C_{\varepsilon_0}}
     {\left(\min_{i\in Q}a^in\right)^{1/2}}
\left[
    \|Df\|_\infty t^2
    +\|D^2f\|_\infty t^{3/2}
    +\|D^3f\|_\infty t
\right].
\end{aligned}
\end{equation}

Applying Lemma~\ref{Milton_approx_B2} to
\eqref{smooth_test_bound_final}, we obtain
\begin{equation}
\label{before_final_simplification_dK}
\begin{aligned}
&d_K\left(
    \operatorname{Law}\left(
        \vec{\xi_t^{(n)}};\mathbb P_{\sigma_0}^{(n)}
    \right),
    \operatorname{Law}\left(\vec{Y_t^{(n)}}\right)
\right)\\
&\qquad\le C_{\varepsilon_0}
\left[
    \frac{t^2}{\left(\min_{i\in Q}a^in\right)^{1/2}}
    +\frac{t^{3/4}}{\left(\min_{i\in Q}a^in\right)^{1/4}}
    +\frac{t^{1/3}}{\left(\min_{i\in Q}a^in\right)^{1/6}}
\right].
\end{aligned}
\end{equation}
Since \(\min_{i\in Q}a^in\ge1\) and
\(t^{3/4}\le t^{1/3}+t^2\) for every \(t\ge0\),
\eqref{before_final_simplification_dK} implies
\[
    d_K\left(
        \operatorname{Law}\left(
            \vec{\xi_t^{(n)}};\mathbb P_{\sigma_0}^{(n)}
        \right),
        \operatorname{Law}\left(\vec{Y_t^{(n)}}\right)
    \right)
    \le
    \frac{C_{\varepsilon_0,C_0}(t^{1/3}+t^2)}
         {\left(\min_{i\in Q}a^in\right)^{1/6}},
\]
which proves the lemma.
\end{proof}

\subsubsection{Gaussian identification and proof of the CLT}
\label{sub:CLT_for_mag_fluc_tk}

The preceding lemma compares the local fluctuation vector with the linear diffusion. It remains only to identify the law of that diffusion.

\begin{proof}[Proof of \Cref{prop:CLT_mag_tk}]
By \eqref{Y_tk_clt_representation},
\(\vec {Y_t^{(n)}}\) is a stochastic integral with a deterministic
integrand. It is therefore centred Gaussian, and the It\^o
isometry gives its covariance as \(\Gamma_t^{(n)}\).
Consequently,
 $\operatorname{Law}\left(\vec {Y_t^{(n)}}\right)
 =
 \mathcal N\left(0,\Gamma_t^{(n)}\right)
 =
 \operatorname{Law}\left(\vec {Z_t^{(n)}}\right)$.
The claimed estimate follows from
Lemma~\ref{lemma:quantitative_convergence_xi_yt}.
\end{proof}

\section{Quantitative approximation for the reference measure}
\label{sec:GD_approx}
\label{Priori_est_Mag_Vec_GD}
\label{subsec:stein_regularity_GD}
For the quantitative approximation, we use the 
Gaussian Stein method employed in \cite[Section~4.7]{JaraXiang},
applied here to sums of independent Bernoulli variables.

\begin{proof}[Proof of \Cref{prop:CLT_mag_GD}]
Let $B_1,\ldots,B_N$ be independent $\mathrm{Ber}(p)$ variables,
where $N\ge1$ and $p\in(0,1)$. Put $q=1-p$,
$X_j=(B_j-p)/\sqrt N$, $S=\sum_{j=1}^N X_j$, and $S_j=S-X_j$.
For a bounded smooth $1$-Lipschitz function $h$, the Stein equation
\[
 pqf'(x)-xf(x)=h(x)-\mathbb E h(\sqrt{pq}\,Z),
 \qquad Z\sim\mathcal N(0,1),
\]
has a bounded solution satisfying $\|f''\|_\infty\le2/(pq)$;
see \cite[Proposition~4.12]{JaraXiang}, or
\cite[Lemma~2.4, equation~(2.13)]{ChenGoldsteinShao2011Stein}
after rescaling.

Since $S_j$ and $X_j$ are independent and $\mathbb E X_j=0$,
Taylor's formula gives
\[
 \left|
 \mathbb E[X_jf(S)]
 -\mathbb E X_j^2\,\mathbb E f'(S_j)
 \right|
 \le\tfrac12\|f''\|_\infty\,\mathbb E|X_j|^3.
\]
Also, $\sum_{j=1}^N\mathbb E X_j^2=pq$ and
$|f'(S)-f'(S_j)|\le\|f''\|_\infty|X_j|$. It follows that
\[
\begin{aligned}
 \left|\mathbb E\{pqf'(S)-Sf(S)\}\right|
 &\le \sum_{j=1}^N\mathbb E X_j^2\,
       \left|\mathbb E[f'(S)-f'(S_j)]\right|+\sum_{j=1}^N
 \left|\mathbb E[X_jf(S)]
       -\mathbb E X_j^2\,\mathbb E f'(S_j)\right|\\
 &\le\|f''\|_\infty\sum_{j=1}^N
 \left\{\mathbb E X_j^2\,\mathbb E|X_j|
       +\tfrac12\mathbb E|X_j|^3\right\}.
\end{aligned}
\]
The variable $X_j$ equals $q/\sqrt N$ with probability $p$
and $-p/\sqrt N$ with probability $q$, so
\[
 \mathbb E X_j^2=\frac{pq}{N},
 \qquad
 \mathbb E|X_j|=\frac{2pq}{\sqrt N},
 \qquad
 \mathbb E|X_j|^3=\frac{pq(p^2+q^2)}{N^{3/2}}.
\]
Substituting these identities and the bound on $\|f''\|_\infty$
into the preceding estimate gives
\[
\begin{aligned}
 \left|\mathbb E h(S)-\mathbb E h(\sqrt{pq}\,Z)\right|
 \le\frac{4pq+p^2+q^2}{\sqrt N}=\frac{1+2pq}{\sqrt N}\le\frac3{2\sqrt N}.
\end{aligned}
\]

Truncation followed by Gaussian convolution extends the estimate to
every $1$-Lipschitz function.
Taking the supremum over $1$-Lipschitz functions yields
 $d_K\left(
 \frac{\mathrm{Bin}(N,p)-Np}{\sqrt N},
 \mathcal N(0,pq)
 \right)\le\frac3{2\sqrt N}$.
For $p\in\{0,1\}$, both laws are $\delta_0$, so the estimate
holds for every $p\in[0,1]$.

Apply this bound with $N=a^in$ and $p=\rho_t^{(i,n)}$ for
each $i\in Q$, and couple each block fluctuation with its Gaussian
approximation so that the expected absolute difference satisfies
this bound, choosing the pairs independently across blocks.
The resulting fluctuation coordinates have the required independent
block laws, while the Gaussian vector has covariance
$\bar\Gamma_t^{(n)}$. Since
$\|(x^0,x^1)\|\le|x^0|+|x^1|$, the coupling representation of
$d_K$ gives
\[
\begin{aligned}
 d_K\left(
 \operatorname{Law}\left(\vec{\zeta_t^{(n)}}\right),
 \mathcal N(0,\bar\Gamma_t^{(n)})
 \right)
 \le\frac32\sum_{i\in Q}\frac1{\sqrt{a^in}}\le\frac3{\sqrt{\min_{i\in Q}(a^in)}}.
\end{aligned}
\]
This proves the proposition.
\end{proof}

\section{Approximation via the reference measure}
\label{approach_tk_nu_t}

In this section, we prove Theorem~\ref{mag_fluc_estimate_dk}. The two
fluctuation vectors are already approximated by centred Gaussian vectors in Propositions~\ref{prop:CLT_mag_tk} and~\ref{prop:CLT_mag_GD}. It therefore remains to compare the corresponding covariance matrices on the logarithmic time window.

\begin{lemma}[Matching of the Gaussian covariances]
\label{lemma:gaussian_covariance_matching_log_window}
Fix \(\varepsilon_0\in(0,1/2)\) and
$
    C_0>\frac{\kappa_1+\kappa_2}{2\kappa_1\kappa_2}.
$
Then, as $n\to \infty$,
\[
\begin{aligned}
&\sup_{\substack{\sigma_0\in\Omega_n:\\
\varepsilon_0\le M^{(n)}(\sigma_0)\le1-\varepsilon_0}}
\ \sup_{t_n\le t\le C_0\log n}
d_K\left(
    \vec{Z_t^{(n)}},
    \vec{\bar Z_t^{(n)}}
\right)
\longrightarrow0\eqcomma
\end{aligned}
\]

where
$
    \vec{Z_t^{(n)}}\sim\mathcal N(0,\Gamma_t^{(n)}),
    \qquad
    \vec{\bar Z_t^{(n)}}\sim\mathcal N(0,\bar\Gamma_t^{(n)}).
$
\end{lemma}

\begin{proof}[Proof of \Cref{lemma:gaussian_covariance_matching_log_window}]
Throughout this proof, $I$ denotes the $2\times2$ identity matrix.
Recall that \(\Gamma_t^{(n)}=\operatorname{Cov}(\vec{Y_t^{(n)}})\). Differentiating \eqref{Gamma_tk_clt} gives
\begin{equation}
\label{eq:Gamma_matrix_ode}
    \frac{\dd}{\dd t}\Gamma_t^{(n)}
    =
    \beta^{(n)}(t)\Gamma_t^{(n)}
    +
    \Gamma_t^{(n)}\beta^{(n)}(t)^T
    +
    \Sigma^{(n)}(t)\Sigma^{(n)}(t)^T,
    \qquad
    \Gamma_0^{(n)}=0.
\end{equation}
The covariance of the Gaussian vector is
$\bar\Gamma_t^{(n)}
    =
    \begin{pmatrix}
    \rho_t^{(0,n)}(1-\rho_t^{(0,n)}) & 0\\
    0 & \rho_t^{(1,n)}(1-\rho_t^{(1,n)})
    \end{pmatrix}$.
At \((\alpha_*,\alpha_*)\), direct differentiation of the drift and evaluation of the carré-du-champ coefficients give
$
    D\vec F^{(\infty)}(\alpha_*,\alpha_*)
    =
    -\frac{\kappa_1\kappa_2}{\kappa_1+\kappa_2}I,
$
and
$
    G^{(i,\infty)}(\alpha_*,\alpha_*)
    =
    \frac{2\kappa_1\kappa_2}{\kappa_1+\kappa_2}
    \alpha_*(1-\alpha_*),
    \ i\in Q.
$
By \eqref{eq:sharp_bulk_deterministic_decay}, proved below,
uniformly over the bulk initial conditions,
$\left\|
        \vec{\rho_t^{(n)}}-(\alpha_*,\alpha_*)
    \right\|
    \le
    C_{\varepsilon_0}
    \left(e^{-c_{\varepsilon_0}t}+\frac1n\right)$, for $t\ge0$.
    
The first and second derivatives of \(\vec {F^{(n)}}\) and the first derivatives
of \(\vec {G^{(n)}}\) are uniformly bounded on \([0,1]^2\). Both \(\vec {F^{(n)}}\) and  \(\vec {G^{(n)}}\) differ from their
respective limits by $O(n^{-1})$ in $C^1([0,1]^2)$.  Since \(\min_{i\in Q}a^i\ge\varepsilon_0\), it follows that
\begin{equation}
\label{eq:beta_sigma_close_to_limit}
\begin{aligned}
&\left\|
    \beta^{(n)}(t)
    +
    \frac{\kappa_1\kappa_2}{\kappa_1+\kappa_2}I
\right\|
+
\left\|
    \Sigma^{(n)}(t)\Sigma^{(n)}(t)^T
    -
    \frac{2\kappa_1\kappa_2}{\kappa_1+\kappa_2}
    \alpha_*(1-\alpha_*)I
\right\|\\
&\qquad+
\left\|
    \bar\Gamma_t^{(n)}
    -
    \alpha_*(1-\alpha_*)I
\right\|
\le
C_{\varepsilon_0}
\left(e^{-c_{\varepsilon_0}t}+\frac1n\right).
\end{aligned}
\end{equation}

Using \eqref{eq:Gamma_matrix_ode} and
$\partial_t[\Phi^{(n)}(t,0)^{-1}]
=-\Phi^{(n)}(t,0)^{-1}\beta^{(n)}(t)$,
we differentiate
$\Phi^{(n)}(t,0)^{-1}
[\Gamma_t^{(n)}-\alpha_*(1-\alpha_*)I]
\Phi^{(n)}(t,0)^{-T}$
with respect to $t$.
We then integrate the resulting identity from $0$ to $t$
and multiply on the left by $\Phi^{(n)}(t,0)$ and on the
right by $\Phi^{(n)}(t,0)^T$.
Since $\Gamma_0^{(n)}=0$ and
$\Phi^{(n)}(t,0)\Phi^{(n)}(s,0)^{-1}=\Phi^{(n)}(t,s)$,
we obtain
\[
\begin{aligned}
 \Gamma_t^{(n)}-\alpha_*(1-\alpha_*)I
 &=
 -\alpha_*(1-\alpha_*)
 \Phi^{(n)}(t,0)\Phi^{(n)}(t,0)^T\\
 &\quad+\int_0^t\Phi^{(n)}(t,s)
 \Bigl\{
 \alpha_*(1-\alpha_*)
 \bigl[\beta^{(n)}(s)+\beta^{(n)}(s)^T
       +2\lambda_*I\bigr]\\
 &\qquad\qquad+
 \Sigma^{(n)}(s)\Sigma^{(n)}(s)^T
 -2\lambda_*\alpha_*(1-\alpha_*)I
 \Bigr\}\Phi^{(n)}(t,s)^T\,\dd s.
\end{aligned}
\]
Here we added and subtracted
$2\lambda_*\alpha_*(1-\alpha_*)I$
to express the integrand in terms of deviations from the
limiting coefficients. In particular,
 $\left\|
 \beta^{(n)}(s)+\beta^{(n)}(s)^T+2\lambda_*I
 \right\|
 \le 2\left\|\beta^{(n)}(s)+\lambda_*I\right\|$,
so \eqref{eq:beta_sigma_close_to_limit} bounds the expression
in braces by
$C_{\varepsilon_0}(e^{-c_{\varepsilon_0}s}+n^{-1})$.

Choose $c_{\varepsilon_0}>0$ small enough that both
\eqref{fundamental_matrix_bound_bfp} and
\eqref{eq:beta_sigma_close_to_limit} hold with this exponent. It follows that
\[
\begin{aligned}
 \left\|\Gamma_t^{(n)}-\alpha_*(1-\alpha_*)I\right\|
 &\le C_{\varepsilon_0}e^{-2c_{\varepsilon_0}t}+C_{\varepsilon_0}\int_0^t
 e^{-2c_{\varepsilon_0}(t-s)}
 \left(e^{-c_{\varepsilon_0}s}+\frac1n\right)\dd s\\
 &\le C_{\varepsilon_0}
 \left(e^{-c_{\varepsilon_0}t}+\frac1n\right).
\end{aligned}
\]
The triangle inequality, together with the bound on
$\bar\Gamma_t^{(n)}-\alpha_*(1-\alpha_*)I$ in
\eqref{eq:beta_sigma_close_to_limit}, therefore yields
\begin{equation}
\label{eq:Gamma_barGamma_close}
 \left\|\Gamma_t^{(n)}-\bar\Gamma_t^{(n)}\right\|
 \le C_{\varepsilon_0}
 \left(e^{-c_{\varepsilon_0}t}+\frac1n\right),
 \qquad t\ge0.
\end{equation}
Since \(t\ge t_n\) on the window under consideration,
$
    e^{-c_{\varepsilon_0}t}
    \le
    e^{-c_{\varepsilon_0}t_n}
    \longrightarrow0.
$
Hence, \(\Gamma_t^{(n)}\) and \(\bar\Gamma_t^{(n)}\) converge uniformly on this
window to the positive definite matrix
\(\alpha_*(1-\alpha_*)I\).
Lemma~\ref{dk_two_gaussians} therefore gives
$
    d_K\left(
        \vec{Z_t^{(n)}},
        \vec{\bar Z_t^{(n)}}
    \right)
    \le
    C
    \left\|
        \Gamma_t^{(n)}-\bar\Gamma_t^{(n)}
    \right\|.
$
Taking the supremum over the bulk initial conditions and
\(t_n\le t\le C_0\log n\), and then using
\eqref{eq:Gamma_barGamma_close}, proves the lemma.
\end{proof}

\begin{proof}[Proof of \Cref{mag_fluc_estimate_dk}]
The bulk assumption gives $\min_{i\in Q}a^in\ge\varepsilon_0n$.
Thus, Propositions~\ref{prop:CLT_mag_tk}
and~\ref{prop:CLT_mag_GD}, together with the triangle inequality,
yield, for all sufficiently large $n$,
\[
\begin{aligned}
 &d_K\left(
 \operatorname{Law}\left(
 \vec{\xi_t^{(n)}};\mathbb P_{\sigma_0}^{(n)}
 \right),
 \operatorname{Law}\left(
 \vec{\zeta_t^{(n)}};\nu_t^{(n)}
 \right)
 \right)\\
 &\quad\le
 \frac{C_{\varepsilon_0,C_0}(\log n)^2}{n^{1/6}}
 +d_K\left(
 \operatorname{Law}\left(\vec{Z_t^{(n)}}\right),
 \operatorname{Law}\left(\vec{\bar Z_t^{(n)}}\right)
 \right),
\end{aligned}
\]
uniformly over the bulk initial configurations and
$t_n\le t\le C_0\log n$.
The first term converges to zero, and the Gaussian term converges to zero
uniformly by
Lemma~\ref{lemma:gaussian_covariance_matching_log_window}.
Taking the supremum over these configurations and times
proves the theorem.
\end{proof}

\section{Approximation of the invariant measure}
\label{approx_invariant_measure}

This section establishes the invariant-measure approximation used in the
mixing-time and cutoff arguments.
By the permutation invariance established in
Subsection~\ref{subsec:TK_model_def}, both
\(\mu_{\kappa_1,\kappa_2,\iota_1,\iota_2}^{(n)}\) and \(\nu_*^{(n)}\) are
uniform conditional on the global count \(nM^{(n)}\). The corresponding
count distribution under \(\nu_*^{(n)}\) is
\(\mathrm{Bin}(n,\alpha_*)\).

\begin{lemma}[Approximation of the invariant measure]
\label{lemma:invariant_measure_to_product_bernoulli}
For fixed $\kappa_1,\kappa_2,\iota_1,\iota_2>0$, there is a finite
constant $C$, depending only on these parameters, such that
\[
 \dtv\left(\mu_{\kappa_1,\kappa_2,\iota_1,\iota_2}^{(n)},
            \nu_*^{(n)}\right)\le Cn^{-1/2},\qquad n\ge1.
\]
\end{lemma}

The explicit invariant count law, its partition-function asymptotics, and
the proof of the lemma are given in
Appendix~\ref{appendix:approx_invariant_measure}.

\section{From Kantorovich--Rubinstein distance to total variation distance}
\label{sec:from_dk_to_dtv}

In this section, we prove Theorem~\ref{convergence_in_total_variation} and Lemma~\ref{upper_bound_mixing_time}. We first establish short-time and long-time relative entropy estimates with respect to \(\nu_t^{(n)}\). 

\subsection{Yau's relative entropy method}
\label{Yau's relative entropy method}

To compare the law of the Togashi--Kaneko model with the time-dependent
reference measure \(\nu_t^{(n)}\), we use Yau's relative entropy method
\cite{Yau-relative-entropy}. This approach is a standard tool for quantitative comparisons between interacting particle systems and time-dependent non-equilibrium reference measures; see also
\cite{jara2018non, jara_amorim_xiang,gonccalves2021sharp} for applications to non-equilibrium fluctuations.

Let $\mu_t^{(n)}$ denote the law of the Markov chain generated by
$L_n$ at time $t$. Whenever $\nu_t^{(n)}$ has full support, define
the relative density and relative entropy by
\begin{equation}
\label{def:relative_density}
 f_t^{(n)}\defeq\frac{\dd\mu_t^{(n)}}{\dd\nu_t^{(n)}},
 \qquad
 H_n(t)\defeq H(\mu_t^{(n)}\mid\nu_t^{(n)})
 \defeq\int_{\Omega_n}f_t^{(n)}\log f_t^{(n)}\,\dd\nu_t^{(n)}.
\end{equation}
For each fixed $t$, let $L_n^*$ be the adjoint of $L_n$ in
$L^2(\nu_t^{(n)})$, so that
\[
 \int_{\Omega_n}fL_ng\,\dd\nu_t^{(n)}
 =\int_{\Omega_n}gL_n^*f\,\dd\nu_t^{(n)},
 \qquad f,g:\Omega_n\to\mathbb R.
\]
Following Yau's method, set
 $\mathbb F_t^{(n)}(\eta)
 \defeq L_n^*\mathbf1(\eta)
 -\partial_t\log\nu_t^{(n)}(\eta)$.
Here the fluctuations are the state functions defined in
Section~\ref{sec:main_results}, satisfying
$\xi_t^{(n)}=\sum_i\sqrt{a^i}\xi_t^{(i,n)}$ pointwise.

The following inequality controls the evolution of the relative entropy.
We refer to \cite{Yau-relative-entropy,KipnisLandim1999} for background
on Yau's relative entropy method and to
\cite[Lemma~A.1]{jara2018non} for a proof.

\begin{prop}[Yau's relative entropy inequality]
\label{yau_ineq}
Assume that \(t\mapsto\nu_t^{(n)}(\eta)\) is continuously differentiable for every \(\eta\), and that \(\nu_t^{(n)}\) has full support for
\(t\in[s_0,s_1]\), and let
$
    \mu_t^{(n)}
    \defeq 
    \mu_{s_0}^{(n)}S_{t-s_0}^{(n)}.
$
With \(f_t^{(n)}\) defined by \eqref{def:relative_density}, for almost every
\(t\in(s_0,s_1)\),
\begin{equation}
\label{eq:Yau_entropy_upper_bound_F_n}
\begin{aligned}
\frac{\dd}{\dd t}
H\left(
    \mu_t^{(n)}
    \,\middle|\,
    \nu_t^{(n)}
\right)
\leq
-\int_{\Omega_n}
\Gamma_n
\left(
    \sqrt{f_t^{(n)}}
\right)
\,\dd\nu_t^{(n)}
+
\mathbb E_{\mu_t^{(n)}}
\left[
    L_n^*\mathbf1
    -
    \partial_t\log\nu_t^{(n)}
\right].
\end{aligned}
\end{equation}
In particular,
\begin{equation}
\label{eq:Yau_entropy_upper_bound_F_n_inequality}
\frac{\dd}{\dd t}
H\left(
    \mu_t^{(n)}
    \,\middle|\,
    \nu_t^{(n)}
\right)
\le
\mathbb E_{\mu_t^{(n)}}
\left[
    \mathbb F_t^{(n)}
\right].
\end{equation}
\end{prop}

The entropy inequality reduces the comparison with the reference measure to estimating \(\mathbb F_t^{(n)}\). The explicit computations of \(L_n^*\mathbf1\), \(\partial_t\log\nu_t^{(n)}\), and
\(\mathbb F_t^{(n)}\) are given in
Appendix~\ref{appendix_Yau_computations}. These estimates are the main
ingredients in the short-time and long-time entropy estimates.

\subsection{Short-time estimate}
\label{sec:Short-time estimate}

\begin{lemma}[Short-time entropy estimate]
\label{short-time-estimation}
Under Assumption~\ref{assumption:bulk_initial_condition_tk},
there exist $T_0\in(0,1]$ and $C_{\varepsilon_0}<\infty$,
depending only on $\kappa_1,\kappa_2,\iota_1,\iota_2$ and $\varepsilon_0$,
such that, for all sufficiently large $n$,
\[
 H\left(\eta_0S_t^{(n)}\,\middle|\,\nu_t^{(n)}\right)
 \le C_{\varepsilon_0}t,
 \qquad 0<t\le T_0,
\]
uniformly over the admissible initial configurations $\eta_0$.

\end{lemma}

\begin{proof}[Proof of \Cref{short-time-estimation}]
The initial derivatives of \(\rho_t^{(0,n)}\) and
\(1-\rho_t^{(1,n)}\) in \eqref{mean_field_ode} are uniformly
positive. Decreasing \(T_0\) if necessary gives
\[
\begin{gathered}
c_{\varepsilon_0}t\le\rho_t^{(0,n)}\le C_{\varepsilon_0}t,\qquad
c_{\varepsilon_0}t\le1-\rho_t^{(1,n)}\le C_{\varepsilon_0}t,\\
1-\rho_t^{(0,n)}\ge\tfrac12,\qquad
\rho_t^{(1,n)}\ge\tfrac12,\qquad 0<t\le T_0.
\end{gathered}
\]
Lemmas~\ref{lemma: a_priori_estimate_for_the_magnetization_vector}
and~\ref{lemma: priori_estimate_for_the_global_density_II} give
 $\mathbb E_{\eta_0}^{(n)}
 \left[(\xi_t^{(n)})^2+\sum_{i\in Q}(\xi_t^{(i,n)})^2\right]
 \le C_{\varepsilon_0}t$.
Since $\rho_t^{(i,n)}(1-\rho_t^{(i,n)})\ge c_{\varepsilon_0}t$,
Cauchy--Schwarz yields
\[
 \frac{\mathbb E_{\eta_0}^{(n)}
       [|\xi_t^{(i,n)}\xi_t^{(n)}|]}
      {\rho_t^{(i,n)}(1-\rho_t^{(i,n)})}
 \le
 \frac{
 \sqrt{\mathbb E_{\eta_0}^{(n)}[(\xi_t^{(i,n)})^2]}
 \sqrt{\mathbb E_{\eta_0}^{(n)}[(\xi_t^{(n)})^2]}
 }{c_{\varepsilon_0}t}
 \le C_{\varepsilon_0}.
\]
The coefficients multiplying these products in
Lemma~\ref{computation_short_time_estimate_Ft} are uniformly bounded,
so each leading term has bounded expected absolute value.

For the correction terms,
Lemma~\ref{lemma:defect_estimate_initial_blocks} gives
$\mathbb E_{\eta_0}^{(n)}
[M_t^{(0,n)}+1-M_t^{(1,n)}]\le C_{\varepsilon_0}t$.
Together with
$\rho_t^{(0,n)},\,1-\rho_t^{(1,n)}\ge c_{\varepsilon_0}t$,
this implies
 $\mathbb E_{\eta_0}^{(n)}
 \left[
 \frac{M_t^{(0,n)}}{\rho_t^{(0,n)}}
 +\frac{1-M_t^{(1,n)}}{1-\rho_t^{(1,n)}}
 \right]\le C_{\varepsilon_0}$.
The other two correction factors are bounded by
$2\rho_t^{(0,n)}$ and $2(1-\rho_t^{(1,n)})$, respectively,
and hence by $C_{\varepsilon_0}t$.
Since all their remaining coefficients are uniformly bounded,
summing these estimates gives
\[
 \mathbb E_{\eta_0S_t^{(n)}}[\mathbb F_t^{(n)}]
 \le
 \mathbb E_{\eta_0}^{(n)}
 \left[|\mathbb F_t^{(n)}(\sigma^{(n)}(t))|\right]
 \le C_{\varepsilon_0},
 \qquad 0<t\le T_0.
\]
Integrating Proposition~\ref{yau_ineq}  yields
 $H_n(t)\le H_n(s)+C_{\varepsilon_0}(t-s)$, for $0<s<t\le T_0$.
For fixed $n$, Lemma~\ref{lemma:initial_relative_entropy_limit}
gives $H_n(s)\to0$ as $s\downarrow0$, proving
$H_n(t)\le C_{\varepsilon_0}t$.
\end{proof}

\begin{prop}[Short-time total variation estimate]
\label{short_time_estimate}
Under the assumptions of Lemma~\ref{short-time-estimation},
\[
\dtv\left(\eta_0S_t^{(n)},\nu_t^{(n)}\right)
\le C_{\varepsilon_0}\sqrt t,\qquad 0<t\le T_0.
\]
\end{prop}

\begin{proof}[Proof of \Cref{short_time_estimate}]
Apply Pinsker's inequality,
\(\dtv(\mu,\nu)\le\sqrt{H(\mu\mid\nu)/2}\).
\end{proof}

\subsection{Long-time estimate}
\label{sec:long_time_estimate}

\begin{lemma}[Long-time entropy estimate]
\label{long_time_estimate_entropy}
Under Assumption~\ref{assumption:bulk_initial_condition_tk},
there exist $C_{\varepsilon_0},c_{\varepsilon_0}>0$ such that,
for all sufficiently large $n$, $r\ge t_n$, and $s\ge0$,
\[
 H\left(\nu_r^{(n)}S_s^{(n)}\,\middle|\,\nu_{r+s}^{(n)}\right)
 \le
 \exp\left\{
 \frac{C_{\varepsilon_0}}{\sqrt n}
 \left(s+\frac{1-e^{-c_{\varepsilon_0}s}}{c_{\varepsilon_0}}\right)
 \right\}-1.
\]
Consequently, for every fixed $C_0>0$,
\[
 \sup_{r\ge t_n}\sup_{0\le s\le C_0\log n}
 \dtv\left(\nu_r^{(n)}S_s^{(n)},\nu_{r+s}^{(n)}\right)
 \longrightarrow0.
\]
Both assertions hold uniformly over the admissible initial configurations.
\end{lemma}

\begin{proof}[Proof of \Cref{long_time_estimate_entropy}]
Suppose first that $\rho_t^{(i,n)}\in[\delta,1-\delta]$ for some
fixed $\delta\in(0,1/2)$, all $i\in Q$, and $t\in[r,r+s]$, and set
$\omega_t^{(n)}
\defeq\max_{i\in Q}|\rho_t^{(i,n)}-\alpha_*|$.
Since
$\kappa_1(1-\rho_t^{(i,n)})-\kappa_2\rho_t^{(i,n)}
=-(\kappa_1+\kappa_2)(\rho_t^{(i,n)}-\alpha_*)$,
the bounded denominators in
Lemma~\ref{computation_short_time_estimate_Ft} imply that each
coefficient of $\xi_t^{(i,n)}\xi_t^{(n)}$ is bounded in absolute
value by $C(\omega_t^{(n)}+n^{-1})$.

For the correction terms,
\[
 \frac{1-\rho_t^{(i,n)}}{\rho_t^{(i,n)}}
 =\frac{\kappa_2}{\kappa_1}+O(\omega_t^{(n)}),
 \qquad
 \frac{\rho_t^{(i,n)}}{1-\rho_t^{(i,n)}}
 =\frac{\kappa_1}{\kappa_2}+O(\omega_t^{(n)}),
\]
uniformly in $i$ and $t\in[r,r+s]$.
Using $\sum_i a^iM^{(i,n)}=M^{(n)}$, we find that their leading parts sum to
\[
\begin{aligned}
 \kappa_2M^{(n)}(2M^{(n)}-1)
 +\kappa_1(1-M^{(n)})(1-2M^{(n)})=(\kappa_1+\kappa_2)
 (2M^{(n)}-1)(M^{(n)}-\alpha_*).
\end{aligned}
\]
The combined ratio error satisfies
\[
\begin{aligned}
 \Bigg|&
 \kappa_1(2M^{(n)}-1)\sum_i a^iM^{(i,n)}
 \left[\frac{1-\rho_t^{(i,n)}}{\rho_t^{(i,n)}}
       -\frac{\kappa_2}{\kappa_1}\right]\\
 &+\kappa_2(1-2M^{(n)})\sum_i a^i(1-M^{(i,n)})
 \left[\frac{\rho_t^{(i,n)}}{1-\rho_t^{(i,n)}}
       -\frac{\kappa_1}{\kappa_2}\right]\Bigg|
 \le C\omega_t^{(n)},
\end{aligned}
\]
since $\sum_i a^i[M^{(i,n)}+(1-M^{(i,n)})]=1$.
Here we used $|2M^{(n)}-1|\le1$ and $0\le M^{(i,n)}\le1$.
Both ratios are uniformly bounded, so the explicit $n^{-1}$
terms contribute at most $C/n$.
Thus the combined correction sums have absolute value at most
$C|M^{(n)}-\alpha_*|+C(\omega_t^{(n)}+n^{-1})$.
Moreover,
 $|M^{(n)}-\alpha_*|
 \le n^{-1/2}|\xi_t^{(n)}|
 +\sum_i a^i|\rho_t^{(i,n)}-\alpha_*|
 \le n^{-1/2}|\xi_t^{(n)}|+\omega_t^{(n)}$.
Applying $2|xy|\le x^2+y^2$ and $|x|\le1+x^2$ therefore gives
\begin{equation}
\label{eq:long_time_F_pathwise_bound}
 |\mathbb F_t^{(n)}|
 \le C(\omega_t^{(n)}+n^{-1/2})
 \left[1+(\xi_t^{(n)})^2+\sum_{i\in Q}(\xi_t^{(i,n)})^2\right].
\end{equation}

For the fixed $\gamma>0$ in
Lemma~\ref{lemma:Yau_quadratic_exponential_moment}, the variational
entropy inequality gives
\[
 \mathbb E_\mu
 \left[1+(\xi_t^{(n)})^2+\sum_i(\xi_t^{(i,n)})^2\right]
 \le\frac{H(\mu\mid\nu_t^{(n)})+\log C}{\gamma}
 \le C[H(\mu\mid\nu_t^{(n)})+1].
\]
Applying Proposition~\ref{yau_ineq} to the evolution from
$\nu_r^{(n)}$ yields
\begin{equation}
\label{long_time_yau}
\begin{aligned}
 \frac{\dd}{\dd v}
H\left(\nu_r^{(n)}S_v^{(n)}\,\middle|\,\nu_{r+v}^{(n)}\right)\le C(\omega_{r+v}^{(n)}+n^{-1/2})
 \left[
 H\left(\nu_r^{(n)}S_v^{(n)}\,\middle|\,\nu_{r+v}^{(n)}\right)+1
 \right].
\end{aligned}
\end{equation}
The initial entropy is zero. Dividing by the entropy plus one and integrating its logarithmic derivative, we obtain
\begin{equation}
\label{eq:shifted_time_entropy_bound}
 H\left(\nu_r^{(n)}S_s^{(n)}\,\middle|\,\nu_{r+s}^{(n)}\right)
 \le
 \exp\left\{
 C\int_0^s[\omega_{r+v}^{(n)}+n^{-1/2}]\,\dd v
 \right\}-1.
\end{equation}

For $r\ge t_n$, \eqref{eq:sharp_bulk_deterministic_decay} and
$e^{-\lambda_*t_n}=n^{-1/2}$ imply
\[
 \omega_{r+v}^{(n)}
 \le C_{\varepsilon_0}
 \left(e^{-\lambda_*(r+v)}+n^{-1}\right)
 \le C_{\varepsilon_0}
 \left(n^{-1/2}e^{-\lambda_*v}+n^{-1}\right).
\]
 Integrating and using $n^{-1}\le n^{-1/2}$ gives
\[
 C\int_0^s[\omega_{r+v}^{(n)}+n^{-1/2}]\,\dd v
 \le\frac{C_{\varepsilon_0}}{\sqrt n}
 \left(s+\frac{1-e^{-\lambda_*s}}{\lambda_*}\right).
\]
Substitution proves the first assertion with
$c_{\varepsilon_0}=\lambda_*$.
For $0\le s\le C_0\log n$, the exponent is uniformly
$O(\log n/\sqrt n)$; hence the entropy tends uniformly to zero,
and Pinsker's inequality proves the second assertion.
\end{proof}
\subsection{An estimate in total variation distance}
\label{estimate_in_dtv}

We now convert the Kantorovich--Rubinstein estimate into a total variation
estimate by evolving a coupling for a short additional time.
All constants and \(o_n(1)\)-terms below are uniform over bulk initial
configurations.

\begin{proof}[Proof of \Cref{convergence_in_total_variation}]
Write \(T=t_n\), and, for the original blocks \(A^0,A^1\), put
\[
X^j(\bar\sigma)=\sum_{x\in A^j}\bar\sigma_x
=a^jnM^{(j,n)}(\bar\sigma),\qquad j\in Q.
\]
We have that $\sigma_0S_T^{(n)}$ is invariant
under permutations within those blocks; $\nu_T^{(n)}$ has the
same invariance by its product form.
Consequently, both laws are uniform conditional on the block
counts $(X^0,X^1)$. For any two such measures,
\begin{equation}
\label{eq:block_exchangeable_TV_identity}
\dtv(\mu,\nu)=
\dtv\left(\operatorname{Law}((X^0,X^1);\mu),
          \operatorname{Law}((X^0,X^1);\nu)\right).
\end{equation}

Fix $0<\varepsilon<1$.
Theorem~\ref{mag_fluc_estimate_dk} and
Kantorovich--Rubinstein duality give, for all sufficiently large
$n$, uniformly over bulk initial configurations, a coupling satisfying
\begin{equation}
\label{eq:coupled_fluctuation_L1}
\mathbb E\left[\|\vec{\xi_T}^{(n)}-\vec{\zeta_T}^{(n)}\|\right]
\le\varepsilon^2.
\end{equation}
Given the coupled block counts, sample each configuration uniformly
among those with its prescribed counts. Since both laws are uniform
conditional on the block counts, this gives a coupling of
$\sigma_0S_T^{(n)}$ and $\nu_T^{(n)}$ that preserves the original
joint distribution of the counts.

Let $\mathcal A_\varepsilon$ be the event that
$\|\vec\xi_T^{(n)}-\vec\zeta_T^{(n)}\|\le\varepsilon$
and all four local densities belong to
$[\alpha_*/2,(1+\alpha_*)/2]$.
By \eqref{eq:coupled_fluctuation_L1} and Markov's inequality,
$\mathbb P(\|\vec\xi_T^{(n)}-\vec\zeta_T^{(n)}\|>\varepsilon)
\le\varepsilon$.
The second-moment estimates,
\eqref{eq:sharp_bulk_deterministic_decay}, and $a^j\ge\varepsilon_0$
also give
\[
\begin{aligned}
 \mathbb E\left[
 \bigl(M^{(j,n)}(\sigma^{(n)}(T))-\alpha_*\bigr)^2
 +\bigl(M^{(j,n)}(\eta^{(n)}(T))-\alpha_*\bigr)^2
 \right]\le C_{\varepsilon_0}\frac{1+\log n}{n},
 \ j\in Q.
\end{aligned}
\]
Thus Markov's inequality and a union bound yield
\begin{equation}
\label{eq:A_epsilon_probability}
 \mathbb P(\mathcal A_\varepsilon^c)
 \le\varepsilon+C_{\varepsilon_0}\frac{1+\log n}{n}.
\end{equation}
For fixed $0<\varepsilon<1$, this ensures
$\mathbb P(\mathcal A_\varepsilon)>0$ for all sufficiently large $n$,
so the conditioning below is well defined.
So on
$\mathcal A_\varepsilon$,
\begin{equation}
\label{eq:local_count_difference_A}
 |X^j(\sigma^{(n)}(T))-X^j(\eta^{(n)}(T))|
 \le\varepsilon\sqrt{a^jn},
 \qquad j\in Q.
\end{equation}
Finally, $M^{(n)}=n^{-1}\sum_jX^j$ and
$\sum_j\sqrt{a^j}\le\sqrt2$, by Cauchy--Schwarz and
$a^0+a^1=1$, imply
\begin{equation}
\label{eq:global_density_difference_A}
 |M^{(n)}(\sigma^{(n)}(T))-M^{(n)}(\eta^{(n)}(T))|
 \le\frac{\varepsilon}{\sqrt n}\sum_j\sqrt{a^j}
 \le\frac{\sqrt2\,\varepsilon}{\sqrt n}.
\end{equation}
Denote the conditional laws of $\sigma^{(n)}(T)$ and
$\eta^{(n)}(T)$ given $\mathcal A_\varepsilon$ by
$\sigma_\varepsilon^{(n)}(T)$ and $\eta_\varepsilon^{(n)}(T)$,
respectively. Both laws remain uniform conditional on their
block counts.

Let $(\bar\sigma,\bar\eta)$ be any pair in the support of the
conditional coupling given $\mathcal A_\varepsilon$.
We restart at time $T=t_n$, writing $u$ for the elapsed time.
Let $\nu_u^{(n)}(\bar\sigma;\cdot)$ be the product reference
measure restarted from $\bar\sigma$, with blocks
$A^i(\bar\sigma)=\{x:\bar\sigma_x=i\}$ and probabilities
$\rho_u^{(i,n)}(\bar\sigma)$ solving \eqref{mean_field_ode}
with initial vector $(0,1)$. Define the corresponding quantities
for $\bar\eta$ in the same way. Since both restarted systems
have the same initial vector, subtracting their integral equations
gives
\[
\begin{aligned}
 \max_i|\rho_u^{(i,n)}(\bar\sigma)-\rho_u^{(i,n)}(\bar\eta)|
 &\le C\int_0^u
 \max_i|\rho_v^{(i,n)}(\bar\sigma)-\rho_v^{(i,n)}(\bar\eta)|\,\dd v\\
 &\quad+Cu|M^{(n)}(\bar\sigma)-M^{(n)}(\bar\eta)|.
\end{aligned}
\]
Consequently, using \eqref{eq:global_density_difference_A},
\begin{equation}
\label{eq:restarted_ODE_stability}
\begin{aligned}
 \max_{i\in Q}|\rho_u^{(i,n)}(\bar\sigma)
                  -\rho_u^{(i,n)}(\bar\eta)|
 \le Cu e^{Cu}|M^{(n)}(\bar\sigma)-M^{(n)}(\bar\eta)|\le C_{\varepsilon_0}u\varepsilon/\sqrt n,
 \quad 0\le u\le T_0,
\end{aligned}
\end{equation}
for every pair in $\mathcal A_\varepsilon$.

On $\mathcal A_\varepsilon$, the global densities
$M^{(n)}(\bar\sigma)=\sum_j a^jM^{(j,n)}(\bar\sigma)$ and
$M^{(n)}(\bar\eta)$ belong to $[\alpha_*/2,(1+\alpha_*)/2]$,
since their original block densities lie in this interval.
Thus both configurations satisfy the bulk condition with parameter
$\min\{\varepsilon_0,\alpha_*/2,(1-\alpha_*)/2\}$.
At the restart, $u=0$, the new block densities equal $(0,1)$.
After reducing $T_0$ if necessary, for $0\le u\le T_0$,
\begin{equation}
\label{eq:restarted_ODE_boundary_bounds}
\begin{gathered}
 c_{\varepsilon_0}u\le\rho_u^{(0,n)}(\bar\sigma)
 \le C_{\varepsilon_0}u,
 \qquad
 c_{\varepsilon_0}u\le1-\rho_u^{(1,n)}(\bar\sigma)
 \le C_{\varepsilon_0}u,\\
 1-\rho_u^{(0,n)}(\bar\sigma)\ge\tfrac12,
 \qquad
 \rho_u^{(1,n)}(\bar\sigma)\ge\tfrac12,
\end{gathered}
\end{equation}
and the same bounds hold for $\bar\eta$.

Fix $0<u\le T_0$ and a pair
$(\bar\sigma,\bar\eta)\in\mathcal A_\varepsilon$.
Under $\nu_u^{(n)}(\bar\sigma;\cdot)$, the count in the original
block $A^j$ satisfies
\begin{equation}
\label{eq:block_count_restarted_reference}
 X^j\overset d=
 \mathrm{Bin}(a^jn-X^j(\bar\sigma),\rho_u^{(0,n)}(\bar\sigma))
 +
 \mathrm{Bin}(X^j(\bar\sigma),\rho_u^{(1,n)}(\bar\sigma)),
\end{equation}
with independent summands, and likewise for $\bar\eta$.
By the definition of $\mathcal A_\varepsilon$, all four trial
numbers are at least $c_{\varepsilon_0}a^jn$; corresponding
numbers differ by at most $\varepsilon\sqrt{a^jn}$ by
\eqref{eq:local_count_difference_A}.

Here $i$ is the index in $\rho_u^{(i,n)}$, while $j$ labels
the original block. For each summand, change the trial number first and the parameter
second. For $i=1$, the triangle inequality gives
\[
\begin{aligned}
 &\dtv\left(
 \mathrm{Bin}(X^j(\bar\sigma),\rho_u^{(1,n)}(\bar\sigma)),
 \mathrm{Bin}(X^j(\bar\eta),\rho_u^{(1,n)}(\bar\eta))
 \right)\\
 &\quad\le\dtv\left(
 \mathrm{Bin}(X^j(\bar\sigma),\rho_u^{(1,n)}(\bar\sigma)),
 \mathrm{Bin}(X^j(\bar\eta),\rho_u^{(1,n)}(\bar\sigma))
 \right)\\
 &\qquad+\dtv\left(
 \mathrm{Bin}(X^j(\bar\eta),\rho_u^{(1,n)}(\bar\sigma)),
 \mathrm{Bin}(X^j(\bar\eta),\rho_u^{(1,n)}(\bar\eta))
 \right).
\end{aligned}
\]
The same decomposition applies to the $i=0$ summands.
For either $i$, Lemma~\ref{dtv_two_binomials_II} bounds the
trial-number change by
\[
 \frac{|X^j(\bar\sigma)-X^j(\bar\eta)|}
 {2\sqrt{c_{\varepsilon_0}a^jn}}
 \sqrt{\frac{\rho_u^{(i,n)}(\bar\sigma)}
 {1-\rho_u^{(i,n)}(\bar\sigma)}}
 \le C_{\varepsilon_0}\varepsilon
 \begin{cases}
 \sqrt u,&i=0,\\
 u^{-1/2},&i=1,
 \end{cases}
\]
where the last bounds follow from
\eqref{eq:restarted_ODE_boundary_bounds}.

For the parameter change, every $p$ between the corresponding
parameters satisfies $p(1-p)\ge c_{\varepsilon_0}u$.
Since the trial number is at most $a^jn$,
Lemma~\ref{dtv_two_binomials_I},
$|\frac{\dd}{\dd p}\arcsin(1-2p)|=[p(1-p)]^{-1/2}$,
and \eqref{eq:restarted_ODE_stability} bound this change by
\[
\begin{aligned}
 \frac{\sqrt{a^jn}}2
 \left|\arcsin(1-2\rho_u^{(i,n)}(\bar\sigma))
       -\arcsin(1-2\rho_u^{(i,n)}(\bar\eta))\right|\le C_{\varepsilon_0}\sqrt{a^jn}\,
 \frac{u\varepsilon}{\sqrt{nu}}
 \le C_{\varepsilon_0}\varepsilon\sqrt u.
\end{aligned}
\]
The counts in distinct original blocks are independent under each
product measure, so Lemma~\ref{formula_total_variation} gives
\begin{equation}
\label{eq:restarted_reference_TV_comparison}
\begin{aligned}
 \dtv\left(
 \operatorname{Law}((X^0,X^1);\nu_u^{(n)}(\bar\sigma;\cdot)),
 \operatorname{Law}((X^0,X^1);\nu_u^{(n)}(\bar\eta;\cdot))
 \right)\le
 C_{\varepsilon_0}\varepsilon(\sqrt u+u^{-1/2}).
\end{aligned}
\end{equation}

 We recall that $\sigma_\varepsilon^{(n)}(T)$ denotes the conditional
probability distribution of $\sigma^{(n)}(T)$ given
$\mathcal A_\varepsilon$. The integral
$\int\nu_u^{(n)}(\bar\sigma;\cdot)\,
\sigma_\varepsilon^{(n)}(T)(\dd\bar\sigma)$
denotes a probability measure on $\Omega_n$: first choose
$\bar\sigma$ according to $\sigma_\varepsilon^{(n)}(T)$,
then sample a configuration from
$\nu_u^{(n)}(\bar\sigma;\cdot)$.
Explicitly, its probability of an event $B\subseteq\Omega_n$ is
\[
\begin{aligned}
 &\left[
 \int\nu_u^{(n)}(\bar\sigma;\cdot)\,
       \sigma_\varepsilon^{(n)}(T)(\dd\bar\sigma)
 \right](B)=
 \sum_{\bar\sigma\in\Omega_n}
 \nu_u^{(n)}(\bar\sigma;B)\,
 \mathbb P\bigl(\sigma^{(n)}(T)=\bar\sigma
 \mid\mathcal A_\varepsilon\bigr).
\end{aligned}
\]
This weighted sum is a probability measure because each restarted
reference measure is a probability measure and the conditional
weights are nonnegative and sum to one.
The same construction applies to $\eta_\varepsilon^{(n)}(T)$.
Both mixtures can therefore be represented using the joint law
of $(\sigma^{(n)}(T),\eta^{(n)}(T))$ conditional on
$\mathcal A_\varepsilon$.

Permutations within the original blocks preserve the conditional
starting laws and relabel the restarted blocks without changing
the ODE parameters. Hence both mixtures are invariant under these
permutations and are uniform conditional on $(X^0,X^1)$.
Applying \eqref{eq:block_exchangeable_TV_identity} to the mixtures,
then \eqref{eq:restarted_reference_TV_comparison}, gives
\[
\begin{aligned}
 \dtv\Bigl(
 \int\nu_u^{(n)}(\bar\sigma;\cdot)\,
       \sigma_\varepsilon^{(n)}(T)(\dd\bar\sigma),
 \int\nu_u^{(n)}(\bar\eta;\cdot)\,
       \eta_\varepsilon^{(n)}(T)(\dd\bar\eta)
 \Bigr)
 \le C_{\varepsilon_0}\varepsilon(\sqrt u+u^{-1/2}).
\end{aligned}
\]

Every starting configuration in these conditional laws has global
density in $[\alpha_*/2,(1+\alpha_*)/2]$.
Proposition~\ref{short_time_estimate} therefore gives
$\dtv(\bar\sigma S_u^{(n)},\nu_u^{(n)}(\bar\sigma;\cdot))
\le C_{\varepsilon_0}\sqrt u$ uniformly, and likewise for $\bar\eta$. In the underlying entropy estimate, sending time to zero is
justified by Lemma~\ref{lemma:initial_relative_entropy_limit}
for each fixed $n$. The short-time bound holds uniformly for every possible starting configuration, so it also holds after averaging over a random starting configuration.
It follows that
\begin{equation}
\label{eq:conditional_short_time_bounds}
\begin{aligned}
 &\dtv\left(
 \sigma_\varepsilon^{(n)}(T)S_u^{(n)},
 \int\nu_u^{(n)}(\bar\sigma;\cdot)\,
       \sigma_\varepsilon^{(n)}(T)(\dd\bar\sigma)
 \right)\\
 &\quad+
 \dtv\left(
 \eta_\varepsilon^{(n)}(T)S_u^{(n)},
 \int\nu_u^{(n)}(\bar\eta;\cdot)\,
       \eta_\varepsilon^{(n)}(T)(\dd\bar\eta)
 \right)
 \le 2C_{\varepsilon_0}\sqrt u.
\end{aligned}
\end{equation}
Inserting the two restarted mixtures between the evolved laws
and applying the triangle inequality now gives
\begin{equation}
\label{eq:conditional_evolved_TV_bound}
 \dtv\left(
 \sigma_\varepsilon^{(n)}(T)S_u^{(n)},
 \eta_\varepsilon^{(n)}(T)S_u^{(n)}
 \right)
 \le C_{\varepsilon_0}
 \left[\sqrt u+\varepsilon(\sqrt u+u^{-1/2})\right].
\end{equation}
Removing conditioning adds at most
\(\mathbb P(\mathcal A_\varepsilon^c)\). Hence
\begin{equation}
\label{eq:unconditioned_evolved_TV_bound}
\dtv(\sigma_0S_{T+u}^{(n)},\nu_T^{(n)}S_u^{(n)})
\le\varepsilon+o_n(1)+
C_{\varepsilon_0}\left[\sqrt u+
\varepsilon(\sqrt u+u^{-1/2})\right].
\end{equation}
Fix $0<u<\min\{\delta,T_0\}$.
Lemma~\ref{long_time_estimate_entropy} gives
$\dtv(\nu_T^{(n)}S_u^{(n)},\nu_{T+u}^{(n)})=o_n(1)$
uniformly over bulk initial configurations.
The preceding bound and the triangle inequality, followed by
$n\to\infty$ and then $\varepsilon\downarrow0$, imply
 $\limsup_{n\to\infty}\sup_{\sigma_0\ \mathrm{bulk}}
 \dtv(\sigma_0S_{T+u}^{(n)},\nu_{T+u}^{(n)})
 \le C_{\varepsilon_0}\sqrt u$.

For $T+\delta\le t\le C_0\log n$, we have $t-T-u\ge0$ and,
by the semigroup property,
$\sigma_0S_t^{(n)}
=(\sigma_0S_{T+u}^{(n)})S_{t-T-u}^{(n)}$.
Inserting $\nu_{T+u}^{(n)}S_{t-T-u}^{(n)}$ in the triangle
inequality and then using total variation contraction gives
\begin{equation}
\label{eq:Markov_contraction_after_smoothing}
\begin{aligned}
 \dtv(\sigma_0S_t^{(n)},\nu_t^{(n)})
 &\le
 \dtv(\sigma_0S_t^{(n)},
      \nu_{T+u}^{(n)}S_{t-T-u}^{(n)})+
 \dtv(\nu_{T+u}^{(n)}S_{t-T-u}^{(n)},\nu_t^{(n)})\\
 &\le
\dtv(\sigma_0S_{T+u}^{(n)},\nu_{T+u}^{(n)})+
 \dtv(\nu_{T+u}^{(n)}S_{t-T-u}^{(n)},\nu_t^{(n)}).
\end{aligned}
\end{equation}
The second term is uniformly $o_n(1)$ by
Lemma~\ref{long_time_estimate_entropy}.
Hence the upper limit of the required supremum is at most
$C_{\varepsilon_0}\sqrt u$.
Letting $u\downarrow0$ proves the theorem. 
\end{proof}

\begin{proof}[Proof of \Cref{upper_bound_mixing_time}]
Choose \(C_0>2/\lambda_*\). Then
\(t_n+1\le4t_n\le C_0\log n\) for all sufficiently large \(n\).
Theorem~\ref{convergence_in_total_variation} gives
$\sup_{\sigma_0\ \mathrm{bulk}}
\dtv(\sigma_0S_{4t_n}^{(n)},\nu_{4t_n}^{(n)})\to 0$.
By \eqref{eq:sharp_bulk_deterministic_decay},
\(\max_i|\rho_{4t_n}^{(i,n)}-\alpha_*|
\le C_{\varepsilon_0}/n\). Therefore
\eqref{eq:block_exchangeable_TV_identity} and
Lemmas~\ref{dtv_two_binomials_I} and~\ref{formula_total_variation} give
\[
\dtv(\nu_{4t_n}^{(n)},\nu_*^{(n)})
\le C_{\varepsilon_0}\sum_{i\in Q}
\sqrt{a^in}\,|\rho_{4t_n}^{(i,n)}-\alpha_*|
\le C_{\varepsilon_0}/\sqrt n.
\]
Combining these bounds with
Lemma~\ref{lemma:invariant_measure_to_product_bernoulli} proves
that the bulk distance at \(4t_n\) tends to zero, and hence
\(t_{\mathrm{mix},\varepsilon_0}^{(n)}(\epsilon)\le4t_n\)
for every fixed \(\epsilon\in(0,1)\) and all sufficiently large \(n\).
\end{proof}

\section{The bulk cutoff profile}
\label{sec:cutoff_profile}

In this section, we prove
Theorems~\ref{thm_convergence_profile} and~\ref{cutoff_profile_tk}. The proof has three parts. We first obtain the
precise profile-scale asymptotics of the deterministic local densities. We
then compute the total variation profile of the product reference measure.
Finally, we transfer this explicit profile to the original chain and maximise
over the bulk initial densities.

\subsection{Deterministic asymptotics on the profile scale}

The first lemma identifies the leading $n^{-1/2}$-order deviation
of each deterministic local density from equilibrium, uniformly
on compact intervals of the profile parameter.
\begin{lemma}
\label{lemma:deterministic_profile_asymptotics}
Under the assumptions of Theorem~\ref{thm_convergence_profile}, for every
\(R>0\) and \(i\in Q\),
\begin{equation}
\label{eq:deterministic_profile_asymptotics}
\sup_{|s|\le R}
\left|
\sqrt n\left(\rho_{t^{(n)}(s)}^{(i,n)}-\alpha_*\right)
-(i-\alpha_*)
\frac{\alpha_*^{\,1-\alpha_*}(1-\alpha_*)^{\alpha_*}}
     {m^{\,1-\alpha_*}(1-m)^{\alpha_*}}e^{-s}
\right|\longrightarrow0.
\end{equation}
\end{lemma}

\begin{proof}[Proof of \Cref{lemma:deterministic_profile_asymptotics}]
Write \(m_n\defeq M^{(n)}(\sigma_0^{(n)})\to m\). For
\(a\in[\varepsilon_0,1-\varepsilon_0]\), let
\(\vec{\widehat\rho_t}(a)
=(\widehat\rho_t^{(0)}(a),\widehat\rho_t^{(1)}(a))\)
solve the limiting system obtained from \eqref{mean_field_ode} by deleting
the terms of order \(n^{-1}\) and setting \(a^0=1-a\), \(a^1=a\), with
\(\widehat\rho_0^{(i)}(a)=i\). Its global density is
\(\widehat\rho_t(a)
\defeq(1-a)\widehat\rho_t^{(0)}(a)+a\widehat\rho_t^{(1)}(a)\).
The limiting equations are
\[
\partial_t\widehat\rho_t^{(i)}(a)
=(\kappa_1+\kappa_2)\widehat\rho_t(a)
\bigl(1-\widehat\rho_t(a)\bigr)
\bigl(\alpha_*-\widehat\rho_t^{(i)}(a)\bigr),
\qquad i\in Q,
\]
and hence
\begin{equation}
\label{eq:limiting_scalar_ode_profile}
\partial_t\widehat\rho_t(a)
=(\kappa_1+\kappa_2)\widehat\rho_t(a)
\bigl(1-\widehat\rho_t(a)\bigr)
\bigl(\alpha_*-\widehat\rho_t(a)\bigr),
\qquad\widehat\rho_0(a)=a.
\end{equation}
Direct differentiation using these equations and the initial conditions
gives, including when \(a=\alpha_*\),
\begin{equation}
\label{eq:exact_local_deterministic_profile}
e^{\lambda_*t}\bigl(\widehat\rho_t^{(i)}(a)-\alpha_*\bigr)
=(i-\alpha_*)
\frac{\widehat\rho_t(a)^{\,1-\alpha_*}
      (1-\widehat\rho_t(a))^{\alpha_*}}
     {a^{\,1-\alpha_*}(1-a)^{\alpha_*}}.
\end{equation}
The global density remains between \(a\) and \(\alpha_*\), and hence
in a fixed compact subset of \((0,1)\). Summing
\eqref{eq:exact_local_deterministic_profile} with weights \(1-a,a\) gives
\(|\widehat\rho_t(a)-\alpha_*|\le C_{\varepsilon_0}e^{-\lambda_*t}\),
uniformly in \(a\). Since \(x\mapsto x^{1-\alpha_*}(1-x)^{\alpha_*}\)
is Lipschitz on that compact set, the same identity yields
\begin{equation}
\label{eq:uniform_limiting_local_profile}
\sup_{a\in[\varepsilon_0,1-\varepsilon_0]}
\left|
e^{\lambda_*t}\bigl(\widehat\rho_t^{(i)}(a)-\alpha_*\bigr)
-(i-\alpha_*)
\frac{\alpha_*^{\,1-\alpha_*}(1-\alpha_*)^{\alpha_*}}
     {a^{\,1-\alpha_*}(1-a)^{\alpha_*}}
\right|
\le C_{\varepsilon_0}e^{-\lambda_*t}.
\end{equation}

We next compare the finite-\(n\) and limiting systems. Write
\(\vec F_a^{(n)}\) and \(\vec F_a^{(\infty)}\) for their vector fields
with \(a^0=1-a\), \(a^1=a\), and put \(\vec p=(\alpha_*,\alpha_*)\).
Solutions of both systems starting in $[0,1]^2$ remain there
for all times. On \([0,1]^2\), both vector fields have a
common Lipschitz constant independent of $n$ and $a$.
For \(a\) and \(\vec x\in[0,1]^2\),
\[
\|\vec F_a^{(n)}(\vec x)-\vec F_a^{(\infty)}(\vec x)\|_2\le C/n,
\qquad D\vec F_a^{(\infty)}(\vec p)=-\lambda_*I.
\]
By continuity of the Jacobian, uniformly in \(a\), choose \(\delta>0\)
such that \(\overline B(\vec p,2\delta)\subset(0,1)^2\) and
\begin{align}
\label{contraction_for_limiting_ode}
\langle\vec x-\vec y,
\vec F_a^{(\infty)}(\vec x)-\vec F_a^{(\infty)}(\vec y)\rangle
\le-\frac{\lambda_*}{2}\|\vec x-\vec y\|_2^2,
\qquad\vec x,\vec y\in\overline B(\vec p,2\delta).
\end{align}
Let \(\vec{\rho_t^{(n)}}(a)\) be the finite-\(n\) solution from
\((0,1)\), so that \(\vec{\rho_t^{(n)}}=\vec{\rho_t^{(n)}}(m_n)\).
By \eqref{eq:exact_local_deterministic_profile}, there is a fixed
\(T=T_{\varepsilon_0}<\infty\) such that
\(\sup_a\sup_{t\ge T}\|\vec{\widehat\rho_t}(a)-\vec p\|_2\le\delta/2\).
Since both solutions start from $(0,1)$, subtracting their integral
equations and using the uniform Lipschitz bound and
$\|\vec F_a^{(n)}-\vec F_a^{(\infty)}\|_\infty\le C/n$ gives
 $\|\vec{\rho_t^{(n)}}(a)-\vec{\widehat\rho_t}(a)\|_2
 \le \frac{Ct}{n}
 +C\int_0^t
 \|\vec{\rho_s^{(n)}}(a)-\vec{\widehat\rho_s}(a)\|_2\,\dd s$.
 
Gr\"onwall's inequality therefore yields
 $\sup_a\sup_{0\le t\le T}
 \|\vec{\rho_t^{(n)}}(a)-\vec{\widehat\rho_t}(a)\|_2
 \le \frac{CTe^{CT}}n=\frac{C_T}{n}$.
For sufficiently large \(n\), the finite-\(n\) density therefore lies
in \(B(\vec p,\delta)\) at time \(T\). Fix \(a\), set
\(\vec z_t=\vec{\rho_t^{(n)}}(a)-\vec{\widehat\rho_t}(a)\), and let
\(\tau\) be the first exit time of \(\vec{\rho_t^{(n)}}(a)\) from
\(B(\vec p,2\delta)\) after \(T\), with \(\tau=\infty\) if there is
no exit. For $T\le t<\tau$, \eqref{contraction_for_limiting_ode} and the
uniform bound give
\[
 \frac12\frac{\mathrm d}{\mathrm dt}\|\vec z_t\|_2^2
 \le-\frac{\lambda_*}{2}\|\vec z_t\|_2^2
 +\frac Cn\|\vec z_t\|_2
 \le-\frac{\lambda_*}{4}\|\vec z_t\|_2^2+\frac C{n^2},
\]
where the second inequality uses Young's inequality in the form
$\frac Cn\|\vec z_t\|_2
\le\frac{\lambda_*}{4}\|\vec z_t\|_2^2
+\frac{C^2}{\lambda_*n^2}$.
Multiplying by $2e^{\lambda_*(t-T)/2}$ and integrating from $T$
to $t$ yields
\[
 \|\vec z_t\|_2^2
 \le e^{-\lambda_*(t-T)/2}\|\vec z_T\|_2^2+\frac C{n^2}
 \le\frac{C_{\varepsilon_0}}{n^2},
\]
since $\|\vec z_T\|_2\le C_T/n$.
For sufficiently large \(n\), a finite exit time would therefore imply,
by continuity, \(\|\vec{\rho_\tau}^{(n)}(a)-\vec p\|_2\le\delta<2\delta\),
a contradiction. Thus \(\tau=\infty\). Combining this with the
bounded-time comparison, uniformly in \(a\), proves
\begin{equation}
\label{eq:finite_n_limiting_ODE_uniform_profile}
\sup_{a\in[\varepsilon_0,1-\varepsilon_0]}\sup_{t\ge0}
\|\vec{\rho_t^{(n)}}(a)-\vec{\widehat\rho_t}(a)\|_2
\le C_{\varepsilon_0}/n.
\end{equation}
Together with \eqref{eq:exact_local_deterministic_profile} and \eqref{eq:finite_n_limiting_ODE_uniform_profile}, this gives
\begin{equation}
\label{eq:sharp_bulk_deterministic_decay}
\sup_{a\in[\varepsilon_0,1-\varepsilon_0]}
\max_{i\in Q}|\rho_t^{(i,n)}(a)-\alpha_*|
\le C_{\varepsilon_0}\bigl(e^{-\lambda_*t}+n^{-1}\bigr),
\qquad t\ge0.
\end{equation}

Finally, \(t^{(n)}(s)\to\infty\) uniformly for \(|s|\le R\), and
\(\sqrt n\,e^{-\lambda_*t^{(n)}(s)}=e^{-s}\).
Using \eqref{eq:uniform_limiting_local_profile},
\eqref{eq:finite_n_limiting_ODE_uniform_profile}, and the Lipschitz
continuity of \(a\mapsto[a^{1-\alpha_*}(1-a)^{\alpha_*}]^{-1}\)
on \([\varepsilon_0,1-\varepsilon_0]\), we obtain
\[
\begin{aligned}
&\sup_{|s|\le R}
\left|
\sqrt n\left(\rho_{t^{(n)}(s)}^{(i,n)}-\alpha_*\right)
-(i-\alpha_*)
\frac{\alpha_*^{\,1-\alpha_*}(1-\alpha_*)^{\alpha_*}}
     {m^{\,1-\alpha_*}(1-m)^{\alpha_*}}e^{-s}
\right|\\
&\qquad\le C_{\varepsilon_0,R}
\left(n^{-1/2}+|m_n-m|\right)\longrightarrow0.
\end{aligned}
\]
This proves \eqref{eq:deterministic_profile_asymptotics}.
\end{proof}

\subsection{Reference-measure approximation and product profile}

The deterministic asymptotics determine the parameters of the product
reference measure on the cutoff window. We first show that this measure
approximates the law of the original chain in total variation, uniformly
for the profile parameter $s$ in compact intervals.

\begin{lemma}
\label{lemma:profile_window_reference_approximation}
Under the assumptions of Theorem~\ref{thm_convergence_profile}, for every
\(R>0\),
\[
    \sup_{|s|\le R}
    \dtv\left(
        \sigma_0^{(n)}S_{t^{(n)}(s)}^{(n)},
        \nu_{t^{(n)}(s)}^{(n)}
    \right)
    \longrightarrow0.
\]
\end{lemma}

The proof is given in
Appendix~\ref{appendix_profile_window_reference}. It is a uniform version of
the short-time comparison argument used in
Subsection~\ref{estimate_in_dtv}. It remains to compute the profile of the explicit product measures.

\begin{lemma}
\label{lemma:reference_measure_profile}
Under the assumptions of Theorem~\ref{thm_convergence_profile}, for every
\(R>0\),
\[
\begin{aligned}
\sup_{|s|\le R}
\left|
    \dtv\left(
        \nu_{t^{(n)}(s)}^{(n)},
        \nu_*^{(n)}
    \right)
    -
    \mathcal G\left(r(m)e^{-s}\right)
\right|
\longrightarrow0.
\end{aligned}
\]
\end{lemma}

\begin{proof}[Proof of \Cref{lemma:reference_measure_profile}]
Put \(\delta_{i,n}(s)=\rho_{t^{(n)}(s)}^{(i,n)}-\alpha_*\),
\(K_i=\sum_{x\in A^i}\eta_x\), and
\(\Psi_{t^{(n)}(s)}^{(n)}
=\dd\nu_{t^{(n)}(s)}^{(n)}/\dd\nu_*^{(n)}\).
Under \(\nu_*^{(n)}\), the counts in the disjoint blocks are independent, with
\(K_i\sim\mathrm{Bin}(a^in,\alpha_*)\), and
\begin{equation}
\label{eq:reference_profile_TV_likelihood}
\dtv(\nu_{t^{(n)}(s)}^{(n)},\nu_*^{(n)})
=\tfrac12\mathbb E_{\nu_*^{(n)}}|\Psi_{t^{(n)}(s)}^{(n)}-1|.
\end{equation}
Fix \(R>0\). Lemma~\ref{lemma:deterministic_profile_asymptotics}
gives \(\sup_{|s|\le R}|\delta_{i,n}(s)|\le C_R/\sqrt n\).
Expanding the logarithms in the product likelihood yields
\begin{equation}
\label{eq:likelihood_ratio_LAN}
\begin{split}
\log\Psi_{t^{(n)}(s)}^{(n)}
=\sum_{i\in Q}\frac{\delta_{i,n}(s)(K_i-a^in\alpha_*)}
                         {\alpha_*(1-\alpha_*)}-\frac12\sum_{i\in Q}
\frac{a^in\delta_{i,n}(s)^2}{\alpha_*(1-\alpha_*)}
+\mathcal R_n(s),
\end{split}
\end{equation}
where
\[
\sup_{|s|\le R}\mathbb E_{\nu_*^{(n)}}|\mathcal R_n(s)|
\le C\sup_{|s|\le R}\sum_i
\left[\delta_{i,n}(s)^2\sqrt{a^in}
+a^in|\delta_{i,n}(s)|^3\right]
\le C_Rn^{-1/2}.
\]
To prove uniform convergence, take any sequence \(s_n\in[-R,R]\)
and a subsequence on which \(s_n\to s\). The standardised counts
converge jointly to two independent standard normals, since
\(a^1\to m\in(0,1)\). The deterministic asymptotics give
\[
\sum_i\frac{a^in\delta_{i,n}(s_n)^2}{\alpha_*(1-\alpha_*)}
\longrightarrow
\frac{\alpha_*^{2(1-\alpha_*)}(1-\alpha_*)^{2\alpha_*}}
     {m^{2(1-\alpha_*)}(1-m)^{2\alpha_*}}
\frac{(1-m)\alpha_*^2+m(1-\alpha_*)^2}{\alpha_*(1-\alpha_*)}
e^{-2s}
=r(m)^2e^{-2s}.
\]
Under $\nu_*^{(n)}$, the counts are independent, with
$\mathbb E K_i=a^in\alpha_*$ and
$\operatorname{Var}(K_i)=a^in\alpha_*(1-\alpha_*)$.
Thus the linear term in \eqref{eq:likelihood_ratio_LAN} is
centred and has variance
$\sum_i a^in\delta_{i,n}(s_n)^2/
[\alpha_*(1-\alpha_*)]\to r(m)^2e^{-2s}$.
The joint CLT for the standardised counts and convergence of
their coefficients give a centred Gaussian limit for this term.
Together with the quadratic-term limit and
$\mathcal R_n(s_n)\to0$ in probability, this yields
\begin{equation}
\label{eq:reference_profile_log_likelihood_limit}
 \log\Psi_{t^{(n)}(s_n)}^{(n)}
 \Longrightarrow r(m)e^{-s}Z-\tfrac12r(m)^2e^{-2s},
 \qquad Z\sim\mathcal N(0,1).
\end{equation}

For each site in $A^i$, the expected squared likelihood factor is
\[
 \alpha_*\left(1+\frac{\delta_{i,n}(s)}{\alpha_*}\right)^2
 +(1-\alpha_*)\left(1-\frac{\delta_{i,n}(s)}{1-\alpha_*}\right)^2
 =1+\frac{\delta_{i,n}(s)^2}{\alpha_*(1-\alpha_*)}.
\]
Independence across sites and $1+x\le e^x$ therefore give
\begin{equation}
\label{eq:reference_profile_second_moment}
\begin{aligned}
 \mathbb E_{\nu_*^{(n)}}[(\Psi_{t^{(n)}(s)}^{(n)})^2]
 =\prod_i\left(1+\frac{\delta_{i,n}(s)^2}
 {\alpha_*(1-\alpha_*)}\right)^{a^in}\le\exp\left\{\sum_i
 \frac{a^in\delta_{i,n}(s)^2}{\alpha_*(1-\alpha_*)}\right\}
 \le e^{C_R},
\end{aligned}
\end{equation}
uniformly for $|s|\le R$, since
$\sup_{|s|\le R}|\delta_{i,n}(s)|\le C_R/\sqrt n$
and $\sum_i a^i=1$.
Consequently, the likelihood ratios are uniformly integrable for
large \(n\) and \(|s|\le R\). Passing to expectations in
\eqref{eq:reference_profile_TV_likelihood}, we obtain
\[
\dtv(\nu_{t^{(n)}(s_n)}^{(n)},\nu_*^{(n)})
\longrightarrow
\tfrac12\mathbb E\left|e^{r(m)e^{-s}Z-r(m)^2e^{-2s}/2}-1\right|
=\mathcal G(r(m)e^{-s}).
\]
By continuity the comparison
profile at \(s_n\) has the same limit. Every sequence has such a
subsequence, proving uniform convergence on \([-R,R]\).
\end{proof}

\subsection{Proofs of the convergence and bulk cutoff profiles}
\label{proofs of cutoff profiles}

We now transfer the explicit reference-measure profile to the original chain
and its invariant measure.

\begin{proof}[Proof of \Cref{thm_convergence_profile}]
Fix $R>0$. For $|s|\le R$ and sufficiently large $n$, the
triangle inequality gives
\[
\begin{aligned}
 &\left|
 \mathscr D_n(t^{(n)}(s);\sigma_0^{(n)})
 -\dtv(\nu_{t^{(n)}(s)}^{(n)},\nu_*^{(n)})
 \right|\\
 &\qquad\le
 \dtv(\sigma_0^{(n)}S_{t^{(n)}(s)}^{(n)},
      \nu_{t^{(n)}(s)}^{(n)})
 +\dtv(\mu_{\kappa_1,\kappa_2,\iota_1,\iota_2}^{(n)},
       \nu_*^{(n)}).
\end{aligned}
\]
The right-hand side tends to zero uniformly for $|s|\le R$ by
Lemmas~\ref{lemma:profile_window_reference_approximation}
and~\ref{lemma:invariant_measure_to_product_bernoulli}.
Combining this with Lemma~\ref{lemma:reference_measure_profile}
proves the theorem, since $R>0$ is arbitrary.
\end{proof}

The bulk profile requires uniformity over all admissible initial
configurations. This follows from compactness of the bulk-density interval.

\begin{proof}[Proof of \Cref{cutoff_profile_tk}]
For every \(R>0\),
\begin{equation}
\label{eq:individual_bulk_profile_uniform}
\begin{aligned}
&\sup_{\substack{\sigma_0\in\Omega_n:\\
\varepsilon_0\le M^{(n)}(\sigma_0)\le1-\varepsilon_0}}
\ \sup_{|s|\le R}
\Bigg|
\mathscr D_n\left(
    t^{(n)}(s);
    \sigma_0
\right)
-
\mathcal G\left(
    r\left(M^{(n)}(\sigma_0)\right)e^{-s}
\right)
\Bigg|
\longrightarrow0.
\end{aligned}
\end{equation}
Indeed, if this failed, one could choose a subsequence of initial
configurations and profile parameters for which the displayed difference stays
bounded away from zero. By compactness, after passing to a further subsequence indexed
by $n_j$, we have
\[
 M^{(n_j)}(\sigma_0^{(n_j)})\longrightarrow
 m\in[\varepsilon_0,1-\varepsilon_0],
 \qquad
 s_{n_j}\longrightarrow s_0\in[-R,R].
\]
The preceding two lemmas, together with
Lemma~\ref{lemma:invariant_measure_to_product_bernoulli}, would then force
the difference to converge to zero, a contradiction.

Taking the supremum over \(\sigma_0\) in
\eqref{eq:individual_bulk_profile_uniform} gives
\[
\begin{aligned}
\mathscr D_{n,\varepsilon_0}
\left(t^{(n)}(s)\right)
&=
\max_{\substack{0\le k\le n:\\
\varepsilon_0\le k/n\le1-\varepsilon_0}}
\mathcal G\left(r(k/n)e^{-s}\right)+o(1).
\end{aligned}
\]
The admissible values \(k/n\) become dense in
\([\varepsilon_0,1-\varepsilon_0]\). Since \(r\) and \(\mathcal G\) are
continuous and \(\mathcal G\) is increasing,
\begin{equation}
\label{eq:bulk_profile_limit}
    \mathscr D_{n,\varepsilon_0}
    \left(t^{(n)}(s)\right)
    \longrightarrow
    \mathcal G\left(r_{\varepsilon_0}^*e^{-s}\right),
\end{equation}
uniformly in \(s\).

Fix \(\epsilon\in(0,1)\) and \(\delta>0\), and set
    $s_\epsilon\defeq \log\left(\frac{r_{\varepsilon_0}^*}{\mathcal G^{-1}(\epsilon)}\right)$, and
    $r_{\varepsilon_0}^*e^{-s_\epsilon}=\mathcal G^{-1}(\epsilon)$.

Since \(\mathcal G\) is strictly increasing,
\[
    \mathcal G\left(r_{\varepsilon_0}^*e^{-(s_\epsilon-\delta)}\right)
    =\mathcal G\left(e^\delta\mathcal G^{-1}(\epsilon)\right)>\epsilon,
    \qquad
    \mathcal G\left(r_{\varepsilon_0}^*e^{-(s_\epsilon+\delta)}\right)
    =\mathcal G\left(e^{-\delta}\mathcal G^{-1}(\epsilon)\right)<\epsilon.
\]
Hence, by \eqref{eq:bulk_profile_limit},
$
    \mathscr D_{n,\varepsilon_0}\left(t^{(n)}(s_\epsilon-\delta)\right)>\epsilon,
    \
    \mathscr D_{n,\varepsilon_0}\left(t^{(n)}(s_\epsilon+\delta)\right)<\epsilon,
$
for all sufficiently large \(n\).

The map \(t\mapsto\mathscr D_{n,\varepsilon_0}(t)\) is nonincreasing. Indeed, for every \(u\ge0\), invariance of \(\mu_{\kappa_1,\kappa_2,\iota_1,\iota_2}^{(n)}\) and contraction of total variation under \(S_u^{(n)}\) give
\begin{align*}
\mathscr D_n(t+u;\sigma_0)
&=\dtv\left((\sigma_0S_t^{(n)})S_u^{(n)},\mu_{\kappa_1,\kappa_2,\iota_1,\iota_2}^{(n)}S_u^{(n)}\right)\\
&\le\dtv\left(\sigma_0S_t^{(n)},\mu_{\kappa_1,\kappa_2,\iota_1,\iota_2}^{(n)}\right)
=\mathscr D_n(t;\sigma_0).
\end{align*}
Taking the supremum over the bulk initial configurations preserves this monotonicity. Therefore, by the definition of \(t_{\mathrm{mix},\varepsilon_0}^{(n)}(\epsilon)\),
    $t^{(n)}(s_\epsilon-\delta)
    \le t_{\mathrm{mix},\varepsilon_0}^{(n)}(\epsilon)
    \le t^{(n)}(s_\epsilon+\delta)$
for all sufficiently large \(n\). Since
    $t^{(n)}(s)=\frac{\kappa_1+\kappa_2}{2\kappa_1\kappa_2}\log n
    +\frac{\kappa_1+\kappa_2}{\kappa_1\kappa_2}s$,
the preceding bounds imply
\[
\begin{aligned}
\left|
    t_{\mathrm{mix},\varepsilon_0}^{(n)}(\epsilon)
    -\frac{\kappa_1+\kappa_2}{2\kappa_1\kappa_2}\log n
    -\frac{\kappa_1+\kappa_2}{\kappa_1\kappa_2}s_\epsilon
\right|
&\le\frac{\kappa_1+\kappa_2}{\kappa_1\kappa_2}\delta .
\end{aligned}
\]
Taking the upper limit as \(n\to\infty\) and then letting \(\delta\downarrow0\), we obtain
\[
    t_{\mathrm{mix},\varepsilon_0}^{(n)}(\epsilon)
    =
    \frac{\kappa_1+\kappa_2}{2\kappa_1\kappa_2}\log n
    +\frac{\kappa_1+\kappa_2}{\kappa_1\kappa_2}
    \log\left(\frac{r_{\varepsilon_0}^*}{\mathcal G^{-1}(\epsilon)}\right)
    +o(1),
\]
which is \eqref{mixing_time_result}.
\end{proof}

\section{Worst-case and boundary mixing}
\label{sec:worst_case}
Starting near a boundary configuration introduces a second time scale.
The minority population first grows from a finite number of particles,
and its random growth amplitude produces an order-one fluctuation in
the entrance time to the bulk. The subsequent relaxation takes the
bulk mixing time. Thus, the boundary changes both the leading mixing
time and the cutoff profile. We identify the random entrance time,
transfer it to the bulk Gaussian limit, and then optimise over all
initial configurations. The auxiliary estimates are collected in
Appendix~\ref{app:worst_case_tools}.

Recall that $K_t=nM_t^{(n)}$. Its birth and death rates are
\begin{equation}
\label{eq:count_rates_worst_case}
 b_n(k)=(1-k/n)^2(\kappa_1k+\iota_1),\qquad
 d_n(k)=(k/n)^2\{\kappa_2(n-k)+\iota_2\}.
\end{equation}
For a fixed integer $k\ge0$, a sequence $\sigma_0^{(n)}\in\Omega_n$
has \emph{$k$ boundary seeds} at the lower (respectively, upper)
boundary if $nM^{(n)}(\sigma_0^{(n)})=k$
(respectively, $n(1-M^{(n)}(\sigma_0^{(n)}))=k$)
for all sufficiently large $n$.
The seeds are the minority particles, whose locations may vary
with $n$; $k=0$ gives the corresponding boundary configuration.
\subsection{The random boundary delay}

\begin{lemma}[Sharp entrance from a fixed minority count]
\label{lemma:sharp_boundary_entrance}
Fix an integer $\ell\ge0$ and $q\in(0,\alpha_*)$. Start the count
process at $\ell$, and let
$\tau_{\lfloor nq\rfloor}=\inf\{t\ge0:K_t=\lfloor nq\rfloor\}$.
Then
\begin{equation}
\label{eq:sharp_boundary_entrance_limit}
 \tau_{\lfloor nq\rfloor}-\frac{\log n}{\kappa_1}
 \ \Longrightarrow\
 \psi(q)-\frac{\log W_{0,\ell}}{\kappa_1},
 \qquad W_{0,\ell}\sim\operatorname{Gamma}
                      (\ell+\iota_1/\kappa_1,1),
\end{equation}
where the gamma distribution has scale one and
\begin{equation}
\label{eq:boundary_clock_definition}
 \psi(x)=\frac1{\lambda_*}
       \log\frac{\alpha_*x^{1-\alpha_*}(1-x)^{\alpha_*}}
                     {\alpha_*-x},\qquad 0<x<\alpha_*.
\end{equation}
The reflected assertion holds from $K_0=n-\ell$, with
$(\kappa_1,\iota_1,\alpha_*)$ replaced by
$(\kappa_2,\iota_2,1-\alpha_*)$ and
$W_{1,\ell}\sim\operatorname{Gamma}(\ell+\iota_2/\kappa_2,1)$.
\end{lemma}

The proof has two stages. Up to a level $j_n=\lfloor n^\gamma\rfloor$,
$0<\gamma<1/2$, the count process agrees, with probability $1-o(1)$,
with the linear birth process with immigration of rate
$\kappa_1 k+\iota_1$. Its classical gamma limit
\cite[Theorem~5 and Corollary~1]{Tavare1987BirthImmigration} gives the
random delay $-\kappa_1^{-1}\log W_{0,\ell}$. Between $j_n$ and
$\lfloor nq\rfloor$, the clock $\psi(K_t/n)$ has drift $1+O(K_t^{-1})$
and quadratic variation rate $O(K_t^{-1})$. Its accumulated error vanishes, so the passage from $j_n$ to
$\lfloor nq\rfloor$ takes
$\psi(q)-\psi(j_n/n)+o_{\mathbb P}(1)$ time.
The complete proof is in Appendix~\ref{app:sharp_boundary_entrance}.

\subsection{From the entrance time to the boundary profile}

Put $\sigma_*^2=\alpha_*(1-\alpha_*)$ and
$T_{i,n}=\left(\frac1{\kappa_i}+\frac1{2\lambda_*}\right)\log n$, for $i\in\{1,2\}$.
For a fixed integer $k\ge0$, let
$W_{0,k}\sim\operatorname{Gamma}(k+\iota_1/\kappa_1,1)$ and
$W_{1,k}\sim\operatorname{Gamma}(k+\iota_2/\kappa_2,1)$.
Define $\mathcal P_{0,k}$ and $\mathcal P_{1,k}$ by
\eqref{eq:boundary_profile_functions}, with $W_0,W_1$ replaced by
$W_{0,k},W_{1,k}$, respectively. In particular,
$\mathcal P_{0,0}=\mathcal P_0$ and $\mathcal P_{1,0}=\mathcal P_1$.

\begin{lemma}[Profiles from fixed boundary seeds]
\label{lemma:fixed_seed_count_profile}
Let $\pi_n$ be the invariant count law. For every fixed integer
$k\ge0$ and $s\in\mathbb R$, the count chain started at $k$ satisfies
\begin{equation}
\label{eq:fixed_seed_count_weak_limit}
 \frac{K^{(n)}_{T_{1,n}+s/\lambda_*}-n\alpha_*}{\sqrt n}
 \Rightarrow
 \sigma_*Z-\alpha_*H_*e^{-s}W_{0,k}^{-(1-\alpha_*)},
\end{equation}
where $Z\sim\mathcal N(0,1)$ is independent of $W_{0,k}$, and
\begin{equation}
\label{eq:fixed_seed_count_profile}
 \dtv\left(
 \operatorname{Law}_k(K^{(n)}_{T_{1,n}+s/\lambda_*}),\pi_n
 \right)\longrightarrow\mathcal P_{0,k}(s).
\end{equation}
The same profile limit holds for the full configuration chain
from every sequence of initial configurations with exactly $k$
ones. The reflected count chain started at $n-k$, and every
sequence of configurations with exactly $k$ zeros, have profile
$\mathcal P_{1,k}(s)$ at time $T_{2,n}+s/\lambda_*$.
All these total variation profile limits are uniform for $s$
in compact intervals.
\end{lemma}

\begin{proof}[Proof of \Cref{lemma:fixed_seed_count_profile}]
Fix $q\in(0,\alpha_*)$, write $j_n=\lfloor nq\rfloor$, and put
$B(q)=(\alpha_*-q)H_*/[q^{1-\alpha_*}(1-q)^{\alpha_*}]$.
For an initial configuration with $j_n$ ones,
$a^1=j_n/n\to q$ and $a^0\to1-q$.
Under $\nu_{t_n+u/\lambda_*}^{(n)}$, the count is the sum of
independent binomials with trial numbers $a^in$ and parameters
$\rho_{t_n+u/\lambda_*}^{(i,n)}$.
Lemma~\ref{lemma:deterministic_profile_asymptotics} gives
\[
\begin{aligned}
 \sqrt n\sum_{i\in Q}a^i
 \bigl(\rho_{t_n+u/\lambda_*}^{(i,n)}-\alpha_*\bigr)
 &\longrightarrow
 \frac{(q-\alpha_*)H_*e^{-u}}
      {q^{1-\alpha_*}(1-q)^{\alpha_*}}
 =-B(q)e^{-u},\\
 \frac1n\operatorname{Var}_{\nu_{t_n+u/\lambda_*}^{(n)}}
 (nM^{(n)})
 &=\sum_{i\in Q}a^i\rho_{t_n+u/\lambda_*}^{(i,n)}
 \bigl(1-\rho_{t_n+u/\lambda_*}^{(i,n)}\bigr)
 \longrightarrow\sigma_*^2,
\end{aligned}
\]
uniformly for $u$ in compact intervals.
Each centred Bernoulli summand divided by $\sqrt n$ has absolute
value at most $n^{-1/2}$, so the Lindeberg condition holds.
The CLT gives the Gaussian limit under the reference measure;
Lemma~\ref{lemma:profile_window_reference_approximation} transfers it to
the chain started at $j_n$:
\begin{equation}
\label{eq:bulk_count_kernel_weak}
 \frac{K^{(n)}_{t_n+u/\lambda_*}-n\alpha_*}{\sqrt n}
 \Rightarrow\mathcal N(-B(q)e^{-u},\sigma_*^2).
\end{equation}
The argument also applies along every sequence $u_n\to u$.
Compactness and continuity of the Gaussian limit therefore give
uniform convergence of expectations of each bounded Lipschitz
test function on compact intervals of $u$.

Now start the count chain at $k$, let $\tau_n$ be its first visit
to $j_n$, and put $U_n=\tau_n-\kappa_1^{-1}\log n$.
Lemma~\ref{lemma:sharp_boundary_entrance} gives
 $U_n\Rightarrow U
 \defeq\psi(q)-\kappa_1^{-1}\log W_{0,k}$.
Thus $(U_n)$ is tight and, since $t_n\to\infty$,
$\mathbb P_k(\tau_n>T_{1,n}+s/\lambda_*)
 =\mathbb P_k(U_n>t_n+s/\lambda_*)\longrightarrow0$.
On $\{\tau_n\le T_{1,n}+s/\lambda_*\}$, the remaining time is
 $T_{1,n}+\frac{s}{\lambda_*}-\tau_n
 =t_n+\frac{s-\lambda_*U_n}{\lambda_*}$.
Since $K_{\tau_n}^{(n)}=j_n$, the strong Markov property gives
the transition law from $j_n$ over this duration.

Fix a bounded Lipschitz function $f:\mathbb R\to\mathbb R$ and define
\[
\begin{aligned}
 \mathcal K_n(v)
 \defeq
 \mathbb E_{j_n}\left[
 f\left(
 \frac{K^{(n)}_{t_n+s/\lambda_*-v}-n\alpha_*}{\sqrt n}
 \right)\right] \ \text{and} \ \
 \mathcal K(v)
 \defeq
 \mathbb E\left[
 f\left(\sigma_*Z-B(q)e^{-s}e^{\lambda_*v}\right)\right],
\end{aligned}
\]
where $Z\sim\mathcal N(0,1)$ is independent of $W_{0,k}$.
Set $\mathcal K_n(v)=0$ when $t_n+s/\lambda_*-v<0$.
For fixed $R>0$ and $|v|\le R$,
$t_n+s/\lambda_*-v\ge t_n+s/\lambda_*-R>0$ for all sufficiently
large $n$. Since $s-\lambda_*v$ ranges over a compact interval,
\eqref{eq:bulk_count_kernel_weak} gives
 $\sup_{|v|\le R}|\mathcal K_n(v)-\mathcal K(v)|
 \longrightarrow0$.
Both functions are bounded in absolute value by $\|f\|_\infty$, so
\[
 \left|
 \mathbb E_k[\mathcal K_n(U_n)]-\mathbb E_k[\mathcal K(U_n)]
 \right|
 \le
 \sup_{|v|\le R}|\mathcal K_n(v)-\mathcal K(v)|
 +2\|f\|_\infty\mathbb P_k(|U_n|>R).
\]
Taking $n\to\infty$ and then $R\to\infty$ shows that the left-hand
side tends to zero. Moreover, $\mathcal K$ is bounded and continuous
by bounded convergence, so $U_n\Rightarrow U$ gives
$\mathbb E_k[\mathcal K(U_n)]\to\mathbb E[\mathcal K(U)]$.
The strong Markov property therefore yields
\[
\begin{aligned}
 \mathbb E_k\left[
 f\left(
 \frac{K^{(n)}_{T_{1,n}+s/\lambda_*}-n\alpha_*}{\sqrt n}
 \right)\right]
 =\mathbb E_k[\mathcal K_n(U_n)]+o(1)\to
 \mathbb E\left[
 f\left(\sigma_*Z-B(q)e^{-s}e^{\lambda_*U}\right)\right].
\end{aligned}
\]
The error is bounded by
$\|f\|_\infty\mathbb P_k(\tau_n>T_{1,n}+s/\lambda_*)$.
Since $B(q)e^{\lambda_*\psi(q)}=\alpha_*H_*$ and
$\lambda_*/\kappa_1=1-\alpha_*$,
 $B(q)e^{-s}e^{\lambda_*U}
 =\alpha_*H_*e^{-s}W_{0,k}^{-(1-\alpha_*)}$.
As $f$ is arbitrary, this proves
\eqref{eq:fixed_seed_count_weak_limit}.

After centring at $n\alpha_*$ and dividing by $\sigma_*\sqrt n$,
the count law thus converges to the law of $Z-A_s$, where
 $A_s\defeq\frac{\alpha_*H_*}{\sigma_*}
 e^{-s}W_{0,k}^{-(1-\alpha_*)}>0 \ \text{almost surely}$.
Applying
Lemma~\ref{invariant_measure_to_bin} gives
 $\dtv\bigl(\pi_n,\mathrm{Bin}(n,\alpha_*)\bigr)
 \le Cn^{-1/2}$.
The binomial law has mean $n\alpha_*$ and variance $n\sigma_*^2$,
so its central limit theorem gives the standard normal limit
for the stationary count under the same rescaling.

The difference of the two limiting distribution functions is
 $h(z)\defeq\mathbb E[\Phi(z+A_s)]-\Phi(z)$.
Both distribution functions are continuous. Since $A_s>0$
almost surely, $h(z)>0$ for every finite $z$, while bounded
convergence gives $h(z)\to0$ as $z\to\pm\infty$.
Thus $\inf_z h(z)=0$.
Writing $\phi$ for the standard normal density, Tonelli's theorem
identifies the density of $Z-A_s$ as
\[
 g(z)=\mathbb E[\phi(z+A_s)],
 \qquad
 \frac{g(z)}{\phi(z)}
 =\mathbb E[e^{-A_sz-A_s^2/2}].
\]
The expectation is finite because
$-A_sz-A_s^2/2=z^2/2-(z+A_s)^2/2\le z^2/2$.
This bound is uniform on compact intervals of $z$, so dominated
convergence proves continuity of the ratio.
The ratio is strictly decreasing because $A_s>0$ almost surely.
Since $g$ and $\phi$ both integrate to one, it takes values
both above and below one and therefore crosses one at a unique
$c\in\mathbb R$. Hence $g(z)>\phi(z)$ for $z<c$ and
$g(z)<\phi(z)$ for $z>c$.

Boundedness of $\phi$ justifies differentiation under the
expectation, giving $h'(z)=g(z)-\phi(z)$. Thus $h$ attains its
maximum at $c$, and
\[
\begin{aligned}
 &\dtv\bigl(\operatorname{Law}(Z-A_s),\mathcal N(0,1)\bigr)
 =\frac12\int_{\mathbb R}|g(z)-\phi(z)|\,\mathrm dz
 =\int_{-\infty}^{c}\{g(z)-\phi(z)\}\,\mathrm dz\\
 &=\mathbb E[\Phi(c+A_s)]-\Phi(c)
 =h(c)=\sup_z h(z)-\inf_z h(z)=\mathcal P_{0,k}(s).
\end{aligned}
\]
Multiplication by $\sigma_*>0$ preserves total variation, so
\[
 \mathcal P_{0,k}(s)
 =
 \dtv\left(
 \operatorname{Law}\left(
 \sigma_*Z-\alpha_*H_*e^{-s}W_{0,k}^{-(1-\alpha_*)}
 \right),
 \mathcal N(0,\sigma_*^2)
 \right).
\]

We now apply Lemma~\ref{lemma:birth_death_tv_oscillation}
with deterministic initial state $k$, observation time
$T_{1,n}+s/\lambda_*$, and the strictly increasing affine map
$x\mapsto(x-n\alpha_*)/(\sigma_*\sqrt n)$.
The two finite-$n$ distribution functions and their respective
continuous limits are explicitly
\[
\begin{aligned}
 &\mathbb P_k\left(
 \frac{K^{(n)}_{T_{1,n}+s/\lambda_*}-n\alpha_*}
      {\sigma_*\sqrt n}\le z
 \right)
 \longrightarrow\mathbb E[\Phi(z+A_s)],\\
 &\pi_n\left(
 \left\{j\in\{0,\ldots,n\}:
 \frac{j-n\alpha_*}{\sigma_*\sqrt n}\le z\right\}
 \right)
 \longrightarrow\Phi(z),
 \qquad z\in\mathbb R.
\end{aligned}
\]
Consequently,
\[
 \dtv\left(
 \operatorname{Law}_k(K^{(n)}_{T_{1,n}+s/\lambda_*}),\pi_n
 \right)
 \longrightarrow
 \sup_z h(z)-\inf_z h(z)
 =\mathcal P_{0,k}(s),
\]
proving \eqref{eq:fixed_seed_count_profile}.
By Lemma~\ref{lem:finite_minority_label_erasure}, the full
configuration distance differs from this count distance by
$o(1)$ for every sequence of initial configurations with exactly
$k$ ones. It therefore has the same profile.
Applying the argument to $n-K^{(n)}$, with
$(\kappa_1,\iota_1)$ and $(\kappa_2,\iota_2)$ interchanged,
gives the corresponding assertions for $k$ zeros at time
$T_{2,n}+s/\lambda_*$.

Finally, using the same $W_{0,k}$ for different values of $s$
gives $A_{s'}=e^{s-s'}A_s\to A_s$ almost surely as $s'\to s$.
The triangle inequality gives
\[
 \left|\mathcal P_{0,k}(s')-\mathcal P_{0,k}(s)\right|
 \le
 \mathbb E\left[
 \dtv\bigl(\mathcal N(-A_{s'},1),\mathcal N(-A_s,1)\bigr)
 \right].
\]
The integrand equals
$2\Phi(|A_{s'}-A_s|/2)-1$, tends to zero almost surely,
and is bounded by one. Bounded convergence proves continuity
of $\mathcal P_{0,k}$; the reflected profile is continuous
by the same argument.
For each $n$, the count and configuration distances are
nonincreasing in $s$ by Markov contraction and stationarity.
Their pointwise convergence to continuous profiles is therefore
uniform on compact intervals.
\end{proof}

\subsection{Optimisation over the initial configuration}

The preceding lemma identifies the profiles from fixed boundary seeds.
Two auxiliary lemmas complete the passage to worst-case mixing.
Lemma~\ref{lem:uniform_nonzero_minority_entrance} gives the uniform
entrance estimate
 $\sup_{0\le k<\lfloor an\rfloor}
 \mathbb E_k\left|\tau_{\lfloor an\rfloor}
       -\frac1{\kappa_1}\log\frac n{k+1}\right|\le C_a$
for sufficiently small fixed $a>0$.
Consequently, starting with $k_n\to\infty$ minority particles
reduces the entrance-time centring by
$\kappa_i^{-1}\log(k_n+1)\to\infty$. At time $T_n+s/\lambda_*$, the total variation distance tends to
zero for growing minority counts, including bulk initial conditions,
and for fixed seeds at the faster endpoint;
see Lemma~\ref{lem:growing_minority_already_mixed}.
Thus the worst-case limit is determined by comparing fixed-seed
profiles at the slower endpoint, or at both endpoints when
$\kappa_1=\kappa_2$. The \emph{lower endpoint} is the all-zero configuration
$\vec{\mathbf0}$, with count $0$ and density $0$.
The \emph{upper endpoint} is the all-one configuration
$\vec{\mathbf1}$, with count $n$ and density $1$.

\begin{proof}[Proof of \Cref{thm:worst_case_cutoff}]
For fixed $k\ge1$, let
$V_k\sim\operatorname{Gamma}(k,1)$ be independent of $W_{0,0}$.
By gamma additivity,
$W_{0,0}+V_k\sim\operatorname{Gamma}\left(k+\frac{\iota_1}{\kappa_1},1\right)$,
which is the law of $W_{0,k}$. We therefore construct it as
 $W_{0,k}=W_{0,0}+V_k\ge W_{0,0}
 \ \text{almost surely}$.
For $k=0$, take $V_0=0$. Since $1-\alpha_*>0$, this construction gives
 $W_{0,k}^{-(1-\alpha_*)}
 \le W_{0,0}^{-(1-\alpha_*)}$, almost surely.
Thus, for every fixed $s,z\in\mathbb R$, monotonicity of $\Phi$
implies
\[
\begin{aligned}
 \mathbb E\left[
 \Phi\left(z+\frac{\alpha_*H_*}{\sigma_*}e^{-s}
                  W_{0,k}^{-(1-\alpha_*)}\right)\right]\le
 \mathbb E\left[
 \Phi\left(z+\frac{\alpha_*H_*}{\sigma_*}e^{-s}
                  W_{0,0}^{-(1-\alpha_*)}\right)\right].
\end{aligned}
\]
Subtracting $\Phi(z)$ and taking the supremum over $z$ yields
 $\mathcal P_{0,k}(s)\le\mathcal P_{0,0}(s)=\mathcal P_0(s)$.
At the other endpoint, gamma additivity similarly gives a
coupling with $W_{1,k}\ge W_{1,0}$. Since $\alpha_*>0$, the
map $w\mapsto w^{-\alpha_*}$ is decreasing, and the same
comparison proves
 $\mathcal P_{1,k}(s)\le\mathcal P_{1,0}(s)=\mathcal P_1(s)$.

Fix $s\in\mathbb R$, choose $a$ as in
Lemma~\ref{lem:uniform_nonzero_minority_entrance}, and put
$j_n=\lfloor an\rfloor$. Since the configuration space is finite,
for each $n$ we may choose $\sigma_0^{(n)}$ attaining
$\mathscr D_n(T_n+s/\lambda_*)$. Write
$k_n=nM^{(n)}(\sigma_0^{(n)})$, and pass to a subsequence along
which these distances converge to
 $L=\limsup_{n\to\infty}\mathscr D_n(T_n+s/\lambda_*)$.
Every further subsequence retains the same limit $L$.

By passing to a further subsequence, we may assume that
$k_n$ lies either in $[j_n,n-j_n]$, below $j_n$,
or above $n-j_n$ throughout. If $j_n\le k_n\le n-j_n$,
the initial configurations are in the bulk, and
Lemma~\ref{lem:growing_minority_already_mixed} gives $L=0$.

Suppose instead that $k_n<j_n$. Since $k_n$ is a nonnegative
integer, we may pass to a further subsequence on which either
$k_n\to\infty$ or $k_n=k$ for a fixed integer $k$.
In the first case, the same lemma gives $L=0$.
In the second case, it also gives $L=0$ if
$\kappa_1>\kappa_2$, since the lower endpoint is then the
faster endpoint.

If $k_n>n-j_n$, apply the same argument to the number of zeros,
$n-k_n$. The limit is again $L=0$ if this number tends to
infinity, or if it is fixed and $\kappa_2>\kappa_1$.
Thus only fixed minority counts at an endpoint with the
smaller growth rate can contribute; when the rates agree,
both endpoints remain to be considered.

For a fixed minority count near zero with
$\kappa_1\le\kappa_2$, we have $T_n=T_{1,n}$.
Lemma~\ref{lemma:fixed_seed_count_profile} and the gamma
comparison therefore give
 $L=\mathcal P_{0,k}(s)\le\mathcal P_0(s)\le\mathcal P(s)$.
Similarly, near $n$ with $\kappa_2\le\kappa_1$, we have
$T_n=T_{2,n}$ and
 $L=\mathcal P_{1,k}(s)\le\mathcal P_1(s)\le\mathcal P(s)$.
These cases cover all possibilities, and hence
 $\limsup_{n\to\infty}
 \mathscr D_n(T_n+s/\lambda_*)\le\mathcal P(s)$.
For the lower bound, suppose first that $\kappa_1<\kappa_2$.
Then $T_n=T_{1,n}$, and the worst-case distance is at least
the distance from $\vec{\mathbf0}$. Applying
Lemma~\ref{lemma:fixed_seed_count_profile} with $k=0$ gives
\[
\begin{aligned}
 \liminf_{n\to\infty}\mathscr D_n(T_n+s/\lambda_*)
 \ge
 \lim_{n\to\infty}
 \mathscr D_n(T_n+s/\lambda_*;\vec{\mathbf0})=\mathcal P_0(s)=\mathcal P(s).
\end{aligned}
\]
If $\kappa_2<\kappa_1$, the same argument starting from
$\vec{\mathbf1}$ gives the lower bound
$\mathcal P_1(s)=\mathcal P(s)$.

When $\kappa_1=\kappa_2$, both endpoint centres equal $T_n$.
The distances from $\vec{\mathbf0}$ and $\vec{\mathbf1}$
therefore converge to $\mathcal P_0(s)$ and $\mathcal P_1(s)$,
respectively. Since the worst-case distance dominates both
endpoint distances,
 $\liminf_{n\to\infty}\mathscr D_n(T_n+s/\lambda_*)
 \ge\max\{\mathcal P_0(s),\mathcal P_1(s)\}
 =\mathcal P(s)$.
Together with the preceding limsup bound, this proves
\eqref{eq:worst_case_profile}.

We next establish the properties of the limiting profile.
Replacing the number of ones by the number of zeros at the upper boundary, we can write
both profiles as
\[
 \mathcal P_i(s)=\sup_{z\in\mathbb R}h_s(z),
 \qquad
 h_s(z)=\mathbb E\left[\Phi(z+e^{-s}V)\right]-\Phi(z),
\]
where $\sigma_*=\sqrt{\alpha_*(1-\alpha_*)}$ and
 $V=\frac{H_*}{\sigma_*}
 \begin{cases}
 \alpha_*W_0^{-(1-\alpha_*)}, & i=0,\\
 (1-\alpha_*)W_1^{-\alpha_*}, & i=1.
 \end{cases}$
The standardised displacement amplitude $V$ is finite and
strictly positive almost surely.

Since $\Phi$ is Lipschitz with constant $(2\pi)^{-1/2}$,
\[
\begin{aligned}
 |\mathcal P_i(s')-\mathcal P_i(s)|
 \le \sup_z|h_{s'}(z)-h_s(z)|\le
 \mathbb E\left[
 \min\left\{1,
 \frac{|e^{-s'}-e^{-s}|V}{\sqrt{2\pi}}\right\}
 \right].
\end{aligned}
\]
Bounded convergence proves continuity as $s'\to s$.

For fixed $s$, the function $h_s$ is continuous, strictly
positive at every finite $z$, and tends to zero as
$z\to\pm\infty$. Its positive supremum is therefore attained
at a finite point. If $s_1<s_2$, let $z_2$ be a maximiser
for $h_{s_2}$. Since $V>0$ almost surely, strict monotonicity
of $\Phi$ gives
\[
 \mathcal P_i(s_1)
 \ge h_{s_1}(z_2)
 >h_{s_2}(z_2)
 =\mathcal P_i(s_2).
\]
Thus both component profiles are strictly decreasing.
Their maximum is also strictly decreasing: choose a component
attaining the maximum at $s_2$ and apply the same inequality.
Continuity of their maximum follows from continuity of the
two components.

As $s\to+\infty$, the bound
 $0\le\mathcal P_i(s)
 \le
 \mathbb E\left[
 \min\left\{1,\frac{e^{-s}V}{\sqrt{2\pi}}\right\}
 \right]$
and bounded convergence give $\mathcal P_i(s)\to0$.
For the other limit, fix $R>0$. Evaluating the difference of
the distribution functions at $z=-R$ gives
 $\mathcal P_i(s)
 \ge
 \mathbb E\left[\Phi(-R+e^{-s}V)\right]-\Phi(-R)$.
As $s\to-\infty$, the right-hand side tends to
$1-\Phi(-R)=\Phi(R)$. Letting $R\to\infty$, and using
$\mathcal P_i(s)\le1$, proves $\mathcal P_i(s)\to1$.
The same limits hold for $\mathcal P$.

For each $n$, the function
$s\mapsto\mathscr D_n(T_n+s/\lambda_*)$ is nonincreasing.
Its pointwise limit $\mathcal P$ is continuous, so convergence
is uniform on compact intervals.

Fix $\epsilon\in(0,1)$ and write
$s_\epsilon=\mathcal P^{-1}(\epsilon)$.
For every $\eta>0$, strict monotonicity gives
\[
 \mathcal P(s_\epsilon-\eta)>\epsilon,
 \qquad
 \mathcal P(s_\epsilon+\eta)<\epsilon.
\]
By profile convergence, these inequalities also hold for the
corresponding finite-$n$ distances for all sufficiently large
$n$. Hence
 $T_n+\frac{s_\epsilon-\eta}{\lambda_*}
 \le t_{\mathrm{mix}}^{(n)}(\epsilon)
 \le T_n+\frac{s_\epsilon+\eta}{\lambda_*}$.
Since $\eta>0$ is arbitrary, this proves
\eqref{eq:worst_case_mixing_asymptotic}.

For fixed $\delta\in(0,1)$ and $R>0$, for all sufficiently
large $n$,
\[
 (1-\delta)T_n\le T_n-R/\lambda_*,
 \qquad
 (1+\delta)T_n\ge T_n+R/\lambda_*.
\]
Time monotonicity and profile convergence therefore give
\[
\begin{aligned}
 \liminf_{n\to\infty}\mathscr D_n((1-\delta)T_n)
 \ge\mathcal P(-R)\quad \text{and} \quad
 \limsup_{n\to\infty}\mathscr D_n((1+\delta)T_n)
 \le\mathcal P(R).
\end{aligned}
\]
Letting $R\to\infty$ proves
\eqref{eq:worst_case_cutoff_location}.
Similarly, since $T_n-t_n\to\infty$, for fixed $s$ and $R$,
 $t_n+s/\lambda_*\le T_n-R/\lambda_*$
for all sufficiently large $n$. Thus
 $\liminf_{n\to\infty}
 \mathscr D_n(t_n+s/\lambda_*)\ge\mathcal P(-R)$.
Letting $R\to\infty$ proves
\eqref{eq:worst_case_on_bulk_window}.

Finally, if $\kappa_1=\kappa_2=\kappa$, then
$\alpha_*=H_*=\sigma_*=1/2$.
Both profiles use the shift $\tfrac12e^{-s}W_i^{-1/2}$, where
$W_0\sim\operatorname{Gamma}(\iota_1/\kappa,1)$ and
$W_1\sim\operatorname{Gamma}(\iota_2/\kappa,1)$.
If $\iota_1\le\iota_2$, gamma additivity gives a coupling
with $W_0\le W_1$, and hence
$\mathcal P_0(s)\ge\mathcal P_1(s)$ for every $s$.
The reverse ordering holds if $\iota_2\le\iota_1$.
Thus, if $\iota_1<\iota_2$, the worst-case profile is attained
asymptotically from $\vec{\mathbf0}$.
If $\iota_2<\iota_1$, it is attained from $\vec{\mathbf1}$. In each case, the distance from the indicated
boundary configuration converges to $\mathcal P(s)$ at time
$T_n+s/\lambda_*$ for every fixed $s$.
If $\iota_1=\iota_2$, the two endpoint profiles coincide,
and both $\vec{\mathbf0}$ and $\vec{\mathbf1}$ attain the
worst-case profile asymptotically.
\end{proof}

The window cannot be replaced by an $o(1)$ window. Indeed, for
$0<\epsilon_1<\epsilon_2<1$,
\[
 t_{\mathrm{mix}}^{(n)}(\epsilon_1)
 -t_{\mathrm{mix}}^{(n)}(\epsilon_2)
 \longrightarrow\lambda_*^{-1}
 \bigl(\mathcal P^{-1}(\epsilon_1)
              -\mathcal P^{-1}(\epsilon_2)\bigr)>0.
\]
The comparison with bulk mixing is explicit:
\begin{equation}
\label{eq:worst_bulk_ratio}
 \frac{t_{\mathrm{mix}}^{(n)}(\epsilon)}
      {t_{\mathrm{mix},\varepsilon_0}^{(n)}(\epsilon)}
 \longrightarrow
 1+\frac{2\lambda_*}{\min\{\kappa_1,\kappa_2\}}
 =1+\frac{2\max\{\kappa_1,\kappa_2\}}{\kappa_1+\kappa_2},
\end{equation}
for every fixed $\epsilon\in(0,1)$ and
$\varepsilon_0\in(0,1/2)$. The ratio is $2$ when the growth rates
agree and lies strictly between $2$ and $3$ otherwise. Both windows have order one, but their profiles have different
forms. The bulk profile $\mathcal G(r(m)e^{-s})$ corresponds to a
deterministic Gaussian translation and is independent of the
mutation parameters $\iota_1,\iota_2$. These parameters enter
the worst-case profile through the gamma amplitudes generated
during escape from the boundary.

\begin{appendix}

\section{Auxiliary estimates for probability metrics}
\label{appendix_metric_estimates}

\subsection{Some useful bounds in Kantorovich--Rubinstein distance}
\label{appendix_KR_estimates}

We state the regularisation lemma from \cite{JaraXiang} and the Gaussian
comparison used in the covariance estimate. Derivative norms below are the
suprema of the induced multilinear operator norms.
\begin{lemma}[Regularisation lemma]
\label{Milton_approx_B2}
Let $\mu,\nu\in\mathcal P_1(\mathbb R^d)$ and let $\ell\ge2$ be an
integer. Suppose that $a_1,\ldots,a_\ell\ge0$ satisfy
\begin{equation}
\label{eq:regularization_input}
 \left|\int f\,d\mu-\int f\,d\nu\right|
 \le\sum_{j=1}^\ell a_j\|D^jf\|_\infty,
 \qquad f\in C_b^\infty(\mathbb R^d).
\end{equation}
Then
\begin{equation}
\label{eq:regularization_conclusion}
 d_K(\mu,\nu)
 \le C(d,\ell)\max_{1\le j\le\ell}a_j^{1/j}.
\end{equation}
\end{lemma}

\begin{proof}[Proof of \Cref{Milton_approx_B2}]
Fix a $1$-Lipschitz function $f$ with $f(0)=0$ and, for $R>0$,
set $f^R=\max\{-R,\min\{f,R\}\}$. Then $f^R$ is bounded and
$1$-Lipschitz, $|f^R(x)|\le\|x\|$, and $f^R\to f$ pointwise.
Let $\gamma$ be the standard Gaussian density on $\mathbb R^d$,
put $\gamma_\theta(x)=\theta^{-d}\gamma(x/\theta)$, and define
$f^R_\theta=f^R*\gamma_\theta$ for $\theta>0$.
Since every Gaussian derivative is integrable,
$f^R_\theta\in C_b^\infty(\mathbb R^d)$.

For \(j\ge1\), differentiating
\(f^R_\theta(x)=\int_{\mathbb R^d}f^R(z)\gamma_\theta(x-z)\,dz\)
with respect to \(x\), changing variables \(y=x-z\), and using
\(\int_{\mathbb R^d}D^j\gamma_\theta(y)\,dy=0\), we obtain
\[
 D^jf^R_\theta(x)
 =
 \int_{\mathbb R^d}
 \bigl(f^R(x-y)-f^R(x)\bigr)D^j\gamma_\theta(y)\,dy.
\]
The Lipschitz bound and a change of variables therefore yield
\begin{align}
\label{eq:trunc_smoothing_error}
 \|f^R_\theta-f^R\|_\infty
 &\le
 \int_{\mathbb R^d}\|y\|\gamma_\theta(y)\,dy
 =C_d\theta,\\
\label{eq:derivative_mollified_bound}
 \|D^jf^R_\theta\|_\infty
 &\le
 \int_{\mathbb R^d}\|y\|\|D^j\gamma_\theta(y)\|\,dy
 =C_{d,j}\theta^{1-j},
 \qquad 1\le j\le\ell.
\end{align}
These constants are finite and independent of $R$ and $f$.

Applying \eqref{eq:regularization_input} to $f^R_\theta$ and
using the smoothing error under each probability measure gives
\begin{equation}
\label{eq:regularization_pre_optimization}
 \left|\int f^R\,d\mu-\int f^R\,d\nu\right|
 \le C(d,\ell)
 \left(\theta+\sum_{j=1}^\ell a_j\theta^{1-j}\right).
\end{equation}
Because $|f^R(x)|\le\|x\|$ and $\mu,\nu\in\mathcal P_1$,
dominated convergence allows $R\to\infty$. Taking the supremum
over all such $f$ then yields the same bound for $d_K(\mu,\nu)$.

Put $A=\max_{1\le j\le\ell}a_j^{1/j}$ and choose
$\theta=A+\varepsilon$ for any $\varepsilon>0$.
Since $a_j\le A^j\le\theta^j$, each term
$a_j\theta^{1-j}$ is at most $\theta$. Hence
$d_K(\mu,\nu)\le C(d,\ell)(A+\varepsilon)$.
Letting $\varepsilon\downarrow0$ proves the claim.
\end{proof}

\begin{lemma}[Comparison of centred Gaussian measures]
\label{dk_two_gaussians}
For symmetric positive semidefinite $k\times k$ matrices
$\Sigma_1,\Sigma_2$,
\[
 d_K(\mathcal N(0,\Sigma_1),\mathcal N(0,\Sigma_2))
 \le \mathbb E|\vec Z|\,
 \|\Sigma_1^{1/2}-\Sigma_2^{1/2}\|_{\mathrm{op}},
 \qquad \vec Z\sim\mathcal N(0,I_k).
\]
If $\Sigma_1,\Sigma_2\ge mI_k$ for some $m>0$, then
\[
 d_K(\mathcal N(0,\Sigma_1),\mathcal N(0,\Sigma_2))
 \le\frac{\mathbb E|\vec Z|}{2\sqrt m}
       \|\Sigma_1-\Sigma_2\|_{\mathrm{op}}
 \le\frac{\sqrt k}{2\sqrt m}
       \|\Sigma_1-\Sigma_2\|_{\mathrm{op}}.
\]
\end{lemma}

\begin{proof}[Proof of \Cref{dk_two_gaussians}]
The coupling $\Sigma_1^{1/2}\vec Z,\Sigma_2^{1/2}\vec Z$
gives the first bound.
For the second, differentiate
$e^{-s\Sigma_1^{1/2}}
(\Sigma_1^{1/2}-\Sigma_2^{1/2})
e^{-s\Sigma_2^{1/2}}$
and integrate over $s\in[0,\infty)$ to obtain
\[
 \Sigma_1^{1/2}-\Sigma_2^{1/2}
 =
 \int_0^\infty
 e^{-s\Sigma_1^{1/2}}(\Sigma_1-\Sigma_2)
 e^{-s\Sigma_2^{1/2}}\,\dd s.
\]
Indeed, $\|e^{-s\Sigma_i^{1/2}}\|_{\mathrm{op}}\le e^{-s\sqrt m}$,
so $e^{-s\Sigma_1^{1/2}}(\Sigma_1^{1/2}-\Sigma_2^{1/2})
e^{-s\Sigma_2^{1/2}}\to0$ as $s\to\infty$, while the integrand
has operator norm at most
$e^{-2s\sqrt m}\|\Sigma_1-\Sigma_2\|_{\mathrm{op}}$.
Integration gives the factor $(2\sqrt m)^{-1}$, and
Cauchy--Schwarz yields
$\mathbb E|\vec Z|\le\sqrt{\mathbb E|\vec Z|^2}=\sqrt k$.
\end{proof}
\subsection{Some useful bounds in total variation distance}
\label{appendix_total_variation_estimates}

\begin{lemma}
\label{dtv_two_binomials_I}
For $n\ge1$ and $p_1,p_2\in[0,1]$,
\[
 \dtv(\mathrm{Bin}(n,p_1),\mathrm{Bin}(n,p_2))
 \le\frac{\sqrt n}{2}
       |\arcsin(1-2p_1)-\arcsin(1-2p_2)|.
\]
If $p_1,p_2\in[\delta,1-\delta]$, with $\delta\in(0,1/2)$, this is at most
$\sqrt n\,|p_1-p_2|/(2\sqrt{\delta(1-\delta)})$.
\end{lemma}

\begin{proof}[Proof of \Cref{dtv_two_binomials_I}]
For $X_p\sim\mathrm{Bin}(n,p)$, $p\in(0,1)$, differentiation gives
\[
 \frac{\dd}{\dd p}\mathbb Ef(X_p)
 =\mathbb E\left[\frac{X_p-np}{p(1-p)}f(X_p)\right],\qquad
 \left|\frac{\dd}{\dd p}\mathbb Ef(X_p)\right|
 \le\|f\|_\infty\sqrt{\frac n{p(1-p)}}.
\]
Integration and the variational formula for total variation give
the first bound, using the antiderivative $-\arcsin(1-2p)$.
Continuity covers the endpoints, while
$p(1-p)\ge\delta(1-\delta)$ on $[\delta,1-\delta]$ gives the second.
\end{proof}

\begin{lemma}
\label{dtv_two_binomials_II}
Let $n,m\ge1$, $p\in[0,1)$, and $X_n\sim\mathrm{Bin}(n,p)$. Then
\[
 \dtv(\mathrm{Bin}(n,p),\mathrm{Bin}(n-1,p))
 =\frac{\mathbb E|X_n/n-p|}{2(1-p)}
 \le\frac12\sqrt{\frac p{n(1-p)}}.
\]
Consequently,
$\dtv(\mathrm{Bin}(n,p),\mathrm{Bin}(n+m,p))
 \le(m/2)\sqrt{p/(n(1-p))}$.
\end{lemma}

\begin{proof}[Proof of \Cref{dtv_two_binomials_II}]
For $0\le k\le n$,
$\mathrm{Bin}(n-1,p)(k)=\mathbb P(X_n=k)(n-k)/(n(1-p))$.
Cauchy--Schwarz therefore gives
\[
 \dtv(\mathrm{Bin}(n,p),\mathrm{Bin}(n-1,p))
 =\frac{\mathbb E|X_n/n-p|}{2(1-p)}
 \le\frac12\sqrt{\frac p{n(1-p)}},
\]
since $\mathbb E(X_n/n-p)^2=p(1-p)/n$.
Applying the triangle inequality over successive trial numbers yields
\[
 \dtv(\mathrm{Bin}(n,p),\mathrm{Bin}(n+m,p))
 \le\frac12\sum_{j=1}^m\sqrt{\frac p{(n+j)(1-p)}}
 \le\frac m2\sqrt{\frac p{n(1-p)}}.
\]
\end{proof}

\begin{lemma}
\label{formula_total_variation}
Let $I$ be finite, let $\lambda_i\ge0$ with $\sum_i\lambda_i=1$, and let
$\mu_i,\nu_i$ be probability measures on a common measurable space. Then
\[
 \dtv\left(\sum_i\lambda_i\mu_i,\sum_i\lambda_i\nu_i\right)
 \le\sum_i\lambda_i\dtv(\mu_i,\nu_i).
\]
If the coordinates within each of $(X_i)_{i\in I}$ and $(Y_i)_{i\in I}$
are independent, then
\[
 \dtv\left(\operatorname{Law}((X_i)_{i\in I}),
           \operatorname{Law}((Y_i)_{i\in I})\right)
 \le\sum_i\dtv(\operatorname{Law}(X_i),\operatorname{Law}(Y_i)).
\]
The same upper bound holds after applying a common measurable map to the
two vectors. No independence between the two vectors is required.
\end{lemma}

\begin{proof}[Proof of \Cref{formula_total_variation}]
For mixtures, the setwise characterisation gives
\[
 \dtv\left(\sum_i\lambda_i\mu_i,\sum_i\lambda_i\nu_i\right)
 =\sup_B\left|\sum_i\lambda_i[\mu_i(B)-\nu_i(B)]\right|
 \le\sum_i\lambda_i\dtv(\mu_i,\nu_i).
\]
For products, replacing one factor at a time and using
$\dtv(\mu\otimes\lambda,\nu\otimes\lambda)=\dtv(\mu,\nu)$ yields
$\dtv\left(\bigotimes_i\mu_i,\bigotimes_i\nu_i\right)
 \le\sum_i\dtv(\mu_i,\nu_i)$.
Independence within each vector identifies its law with the
corresponding product. Finally, for any measurable map $T$,
\[
 \dtv(\mu\circ T^{-1},\nu\circ T^{-1})
 =\sup_B|\mu(T^{-1}B)-\nu(T^{-1}B)|
 \le\dtv(\mu,\nu).
\]
\end{proof}

\section{Technical estimates for the diffusion approximation}
\label{appendix_diffusion_estimates}

This section contains the two technical lemmas used in the quantitative
diffusion approximation: the comparison between the Markov chain generator
and the diffusion generator, and the regularity estimate for the backward
Kolmogorov equation.

\subsection{Generator expansion for the local fluctuation vector}
\label{append_com_tk}

We compare the generator of the local fluctuation vector with
\(\mathscr L_t^{(n)}\) in \eqref{diffusion_operator_bkp}.

\begin{lemma}[Generator comparison]
\label{routine_computation_3}
Assume that both initial blocks are nonempty, so that \(a^0,a^1>0\),
and let \(g\in C_b^3(\mathbb R^Q;\mathbb R)\).
For every \(t\ge0\) and \(\sigma\in\Omega_n\), there exist points
\(\vec\vartheta_t^{\,i}(\sigma)\) and
\(\vec{\widehat\vartheta}_t^{\,i}(\sigma)\) on the line segment
joining \(\vec M^{(n)}(\sigma)\) and \(\vec{\rho_t^{(n)}}\) such that
\[
\begin{aligned}
&(\partial_t+L_n)
\left[g\left(\vec{\xi_t^{(n)}}\right)\right](\sigma)
-\bigl(\mathscr L_t^{(n)}g\bigr)
\left(\vec{\xi_t^{(n)}}(\sigma)\right)\\
&\quad=
\frac12\sum_{i\in Q}\sqrt{a^in}\,
\partial_i g\left(\vec{\xi_t^{(n)}}(\sigma)\right)
D^2F^{(i,n)}\left(\vec\vartheta_t^{\,i}(\sigma)\right)
\left[
\vec M^{(n)}(\sigma)-\vec{\rho_t^{(n)}},
\vec M^{(n)}(\sigma)-\vec{\rho_t^{(n)}}
\right]\\
&\qquad+
\frac12\sum_{i\in Q}
\partial_i^2g\left(\vec{\xi_t^{(n)}}(\sigma)\right)
DG^{(i,n)}\left(\vec{\widehat\vartheta}_t^{\,i}(\sigma)\right)
\cdot\left(\vec M^{(n)}(\sigma)-\vec{\rho_t^{(n)}}\right)
+\mathcal R_ng(t,\sigma),
\end{aligned}
\]
where
 $|\mathcal R_ng(t,\sigma)|
 \le
 \frac{C\|D^3g\|_\infty}
      {\sqrt{\min_{i\in Q}a^in}}$.
Consequently,
\[
\begin{aligned}
 &\left|
 (\partial_t+L_n)
 \left[g\left(\vec{\xi_t^{(n)}}\right)\right](\sigma)
 -\bigl(\mathscr L_t^{(n)}g\bigr)
 \left(\vec{\xi_t^{(n)}}(\sigma)\right)
 \right|\\
 &\quad\le
 \frac{C}{\sqrt{\min_{i\in Q}a^in}}
 \left[
 \|Dg\|_\infty
 \left\|\vec{\xi_t^{(n)}}(\sigma)\right\|^2
 +\|D^2g\|_\infty
 \left\|\vec{\xi_t^{(n)}}(\sigma)\right\|
 +\|D^3g\|_\infty
 \right].
\end{aligned}
\]
The constant \(C\) depends only on
\(\kappa_1,\kappa_2,\iota_1,\iota_2\).
\end{lemma}

\begin{proof}[Proof of \Cref{routine_computation_3}]
For \(x\in A^i\), a flip at \(x\) changes the fluctuation vector
by \((1-2\sigma_x)e_i/\sqrt{a^in}\). Taylor's formula gives
\[
\begin{aligned}
 &g\left(\vec{\xi_t^{(n)}}(\sigma^x)\right)
 -g\left(\vec{\xi_t^{(n)}}(\sigma)\right)\\
 &\quad=
 \frac{1-2\sigma_x}{\sqrt{a^in}}\,
 \partial_i g\left(\vec{\xi_t^{(n)}}(\sigma)\right)
 +\frac{1}{2a^in}\,
 \partial_i^2g\left(\vec{\xi_t^{(n)}}(\sigma)\right)
 +R_{x,i}(t,\sigma),
\end{aligned}
\]
where
\(\lvert R_{x,i}(t,\sigma)\rvert
\le\|D^3g\|_\infty/[6(a^in)^{3/2}]\).
By the definitions of the local coefficients,
\[
\begin{aligned}
 \sum_{x\in A^i}c_x(\sigma)(1-2\sigma_x)
 =a^inF^{(i,n)}\left(\vec M^{(n)}(\sigma)\right) \ \text{and} \
 \sum_{x\in A^i}c_x(\sigma)
 =a^inG^{(i,n)}\left(\vec M^{(n)}(\sigma)\right).
\end{aligned}
\]
Moreover, \eqref{mean_field_ode} gives
\(\partial_t\xi_t^{(i,n)}(\sigma)
=-\sqrt{a^in}\,F^{(i,n)}(\vec{\rho_t^{(n)}})\).
Summing over flips and including this derivative yields
\[
\begin{aligned}
(\partial_t+L_n)
\left[g\left(\vec{\xi_t^{(n)}}\right)\right](\sigma)
&=
\sum_{i\in Q}\sqrt{a^in}\,
\partial_i g\left(\vec{\xi_t^{(n)}}(\sigma)\right)
\left[
F^{(i,n)}\left(\vec M^{(n)}(\sigma)\right)
-F^{(i,n)}\left(\vec{\rho_t^{(n)}}\right)
\right]\\
&\quad+
\frac12\sum_{i\in Q}
\partial_i^2g\left(\vec{\xi_t^{(n)}}(\sigma)\right)
G^{(i,n)}\left(\vec M^{(n)}(\sigma)\right)
+\mathcal R_ng(t,\sigma).
\end{aligned}
\]
Uniform boundedness of \(G^{(i,n)}\) implies
\[
 |\mathcal R_ng(t,\sigma)|
 \le
 \frac{\|D^3g\|_\infty}{6}
 \sum_{i\in Q}
 \frac{G^{(i,n)}(\vec M^{(n)}(\sigma))}{\sqrt{a^in}}
 \le
 \frac{C\|D^3g\|_\infty}{\sqrt{\min_{i\in Q}a^in}}.
\]

Expand \(F^{(i,n)}\) to first order and \(G^{(i,n)}\) to
zeroth order along the segment joining
\(\vec{\rho_t^{(n)}}\) and \(\vec M^{(n)}(\sigma)\).
Since
\[
 \sqrt{a^in}\,
 DF^{(i,n)}\left(\vec{\rho_t^{(n)}}\right)
 \cdot\left(\vec M^{(n)}(\sigma)-\vec{\rho_t^{(n)}}\right)
 =
 \sum_{j\in Q}\beta_{ij}^{(n)}(t)\xi_t^{(j,n)}(\sigma),
\]
subtracting the diffusion generator gives the expansion in
the statement.

To estimate the Taylor terms, direct differentiation gives,
uniformly for \(\vec m\in[0,1]^2\) and \(\vec u\in\mathbb R^2\),
\[
\begin{aligned}
 \left|D^2F^{(i,n)}(\vec m)[\vec u,\vec u]\right|
 &\le C\left(
 |\vec a\cdot\vec u|^2
 +|u^i|\,|\vec a\cdot\vec u|
 \right),\\
 \left|DG^{(i,n)}(\vec m)\cdot\vec u\right|
 &\le C\left(
 |\vec a\cdot\vec u|+|u^i|
 \right).
\end{aligned}
\]
Apply these bounds with
\(u^i=M^{(i,n)}(\sigma)-\rho_t^{(i,n)}
=\xi_t^{(i,n)}(\sigma)/\sqrt{a^in}\).
The identity between global and local fluctuations gives
 $|\vec a\cdot\vec u|
 =
 \frac{|\xi_t^{(n)}(\sigma)|}{\sqrt n}
 \le
 \frac{\left\|\vec{\xi_t^{(n)}}(\sigma)\right\|}{\sqrt n}$.
Summing over \(i\in Q\) and applying Cauchy--Schwarz therefore
bounds the first Taylor term by
\(C\|Dg\|_\infty
\|\vec{\xi_t^{(n)}}(\sigma)\|^2/\sqrt n\),
and the second by
\(C\|D^2g\|_\infty
\|\vec{\xi_t^{(n)}}(\sigma)\|/
\sqrt{\min_{i\in Q}a^in}\).
Together with the remainder estimate and
\(\min_{i\in Q}a^in\le n\), this proves the final inequality.
\end{proof}

\subsection{Regularity of the backward Kolmogorov equation}
\label{appendix_reg_bFP}
We establish stability of the linearised flow and use the explicit
solution of the linear SDE to bound derivatives of the backward solution.
\begin{lemma}[Uniform stability of the linearised flow]
\label{lemma:uniform_stability_linearized_TK}
Under Assumption~\ref{assumption:bulk_initial_condition_tk}, there
exist \(C_{\varepsilon_0},c_{\varepsilon_0}>0\) such that, uniformly
over bulk initial configurations and for sufficiently large \(n\),
\[
 \|\Phi^{(n)}(t,s)\|
 \le C_{\varepsilon_0}e^{-c_{\varepsilon_0}(t-s)},
 \qquad 0\le s\le t.
\]
\end{lemma}

\begin{proof}[Proof of \Cref{lemma:uniform_stability_linearized_TK}]
Let \(\mathbf u=(\sqrt{a^0},\sqrt{a^1})^{\intercal}\), so that
\(\|\mathbf u\|=1\), and let \(I\) denote the identity matrix.
Define \(\mathbf v_t^{(n)}\) by its components
\[
 v_t^{(i,n)}
 \defeq\sqrt{a^i}\left[
 \bigl(\kappa_1-(\kappa_1+\kappa_2)\rho_t^{(i,n)}\bigr)
 \bigl(1-2\rho_t^{(n)}\bigr)
 -\frac{\iota_1(1-\rho_t^{(i,n)})+\iota_2\rho_t^{(i,n)}}{n}
 \right].
\]
With \(A_t^{(n)}\) as in \eqref{eq:A_t_def}, direct differentiation gives
\[
 \beta^{(n)}(t)
 =-A_t^{(n)}I+\mathbf v_t^{(n)}\mathbf u^{\intercal},
 \qquad
 -A_t^{(n)}+\mathbf u^{\intercal}\mathbf v_t^{(n)}
 =(F^{(n)})'\bigl(\rho_t^{(n)}\bigr).
\]
By \eqref{eq:rho_global_bulk_bound} and
\(\rho_t^{(i,n)}\in[0,1]\), there are constants \(a_0,C>0\)
such that \(A_t^{(n)}\ge a_0\) and
\(\|\mathbf v_t^{(n)}\|\le C\), uniformly in \(t\), sufficiently
large \(n\), and the bulk initial configuration.

Fix $s\ge0$ and an arbitrary $\mathbf z_s\in\mathbb R^2$.
The fundamental matrix gives
$\mathbf z_t=\Phi^{(n)}(t,s)\mathbf z_s$, which solves
$\dot{\mathbf z}_t=\beta^{(n)}(t)\mathbf z_t$ with initial
value $\mathbf z_s$ at time $s$.
Set $y_t=\mathbf u^\intercal\mathbf z_t$.
Since $\mathbf u$ is constant, multiplying by
$\mathbf u^\intercal$ and using the preceding identities gives
\[
 \dot y_t=(F^{(n)})'(\rho_t^{(n)})y_t,
 \qquad
 \dot{\mathbf z}_t=-A_t^{(n)}\mathbf z_t+\mathbf v_t^{(n)}y_t.
\]

Since \(\mathcal I\subset(x_-,x_+)\) and
\(F^{(n)}\to F^{(\infty)}\) in \(C^1([0,1])\), there exist
\(L,c_0>0\) such that
\(\lvert(F^{(n)})'(m)\rvert\le L\) on \([0,1]\) and
\((F^{(n)})'(m)\le-c_0\) on \(\mathcal I\), for all sufficiently
large \(n\).
Lemma~\ref{lemma: deterministic_entrance_time_bulk} gives
\(t_{\mathrm{ent}}^{(n)}\le T_{\varepsilon_0}\) and
\(\rho_r^{(n)}\in\mathcal J\subset\mathcal I\) thereafter. Since
 $(F^{(n)})'\bigl(\rho_r^{(n)}\bigr)
 \le -c_0+(L+c_0)\mathbf1_{\{r<t_{\mathrm{ent}}^{(n)}\}}$,
integration and \(t_{\mathrm{ent}}^{(n)}\le T_{\varepsilon_0}\) give
\[
\begin{aligned}
\int_s^t(F^{(n)})'\bigl(\rho_r^{(n)}\bigr)\,\dd r
 \le -c_0(t-s)
 +(L+c_0)\int_s^t
 \mathbf1_{\{r<t_{\mathrm{ent}}^{(n)}\}}\,\dd r\le (L+c_0)T_{\varepsilon_0}-c_0(t-s).
\end{aligned}
\]
Solving $\dot y_t=(F^{(n)})'(\rho_t^{(n)})y_t$ and using
\(\lvert y_s\rvert=\lvert\mathbf u^\intercal\mathbf z_s\rvert
\le\|\mathbf z_s\|\), we obtain
\[
 |y_t|
 =
 |y_s|\exp\left\{
 \int_s^t(F^{(n)})'\bigl(\rho_r^{(n)}\bigr)\,\dd r
 \right\}
 \le C_{\varepsilon_0}e^{-c_0(t-s)}\|\mathbf z_s\|.
\]
We also have the explicit formula
 $\mathbf z_t
 =
 e^{-\int_s^t A_q^{(n)}\,\dd q}\mathbf z_s
 +
 \int_s^t e^{-\int_r^t A_q^{(n)}\,\dd q}
 \mathbf v_r^{(n)}y_r\,\dd r$.
Using \(A_q^{(n)}\ge a_0\), \(\|\mathbf v_r^{(n)}\|\le C\),
and the preceding bound on \(y_r\), we deduce
\[
\begin{aligned}
 \|\mathbf z_t\|
 &\le
 \left(
 e^{-a_0(t-s)}
 +C_{\varepsilon_0}\int_s^t
 e^{-a_0(t-r)}e^{-c_0(r-s)}\,\dd r
 \right)\|\mathbf z_s\|\\
 &\le
 C_{\varepsilon_0}(1+t-s)
 e^{-\min\{a_0,c_0\}(t-s)}\|\mathbf z_s\|.
\end{aligned}
\]
The last inequality follows from
\(a_0(t-r)+c_0(r-s)\ge\min\{a_0,c_0\}(t-s)\)
for \(s\le r\le t\).
Set \(c_{\varepsilon_0}=\frac12\min\{a_0,c_0\}\).
Since
\((1+t-s)e^{-c_{\varepsilon_0}(t-s)}
\le1+c_{\varepsilon_0}^{-1}\), taking the supremum over
\(\|\mathbf z_s\|=1\) proves the claim.
\end{proof}

\begin{proof}[Proof of \Cref{lemma:reg_bFP}]
The linear SDE gives
\[
 g_{s,t}(\vec x)
 =\mathbb E \left[f\left(
 \Phi^{(n)}(t,s)\vec x+
 \int_s^t\Phi^{(n)}(t,r)\Sigma^{(n)}(r)\,\dd\vec B_r
 \right)\right].
\]
For \(1\le\ell\le k\), boundedness of \(D^\ell f\) permits
differentiation under the expectation. Thus
Lemma~\ref{lemma:uniform_stability_linearized_TK} gives
 $\|D^\ell g_{s,t}\|_\infty
 \le\|\Phi^{(n)}(t,s)\|^\ell\|D^\ell f\|_\infty
 \le C_{\varepsilon_0}^{\,\ell}
 e^{-\ell c_{\varepsilon_0}(t-s)}\|D^\ell f\|_\infty$.
For \(\ell=0\), the representation directly yields
\(\|g_{s,t}\|_\infty\le\|f\|_\infty\), completing the proof.
\end{proof}
\section{Approximation of the invariant measure}
\label{appendix:approx_invariant_measure}

Throughout this section, the four parameters are fixed and strictly positive.
Put $\ell_j=\iota_j/\kappa_j$ for $j=1,2$, and write
$(\ell)_{(k)}=\Gamma(\ell+k)/\Gamma(\ell)$,
with $(\ell)_{(0)}=1$. By permutation invariance, the invariant law is
uniform on each level set of $nM^{(n)}$. Its count probabilities are
\begin{equation}
\label{mean_field_markov_chain}
 \pi_n(k)\defeq\mu_{\ell_1,\ell_2}^{(n)}(k/n)
 =\binom nk\mu_{\kappa_1,\kappa_2,\iota_1,\iota_2}^{(n)}(\sigma),
 \qquad nM^{(n)}(\sigma)=k.
\end{equation}
The notation $\mu_{\ell_1,\ell_2}^{(n)}$ suppresses the dependence on
$\kappa_1/\kappa_2$. We use $\pi_n$ when comparing count laws, so that
both measures are on $\{0,\ldots,n\}$.

\subsection{The invariant count law}
The count chain has rates
\[
 q_n(k,k+1)=\frac{(n-k)^2}{n^2}(\kappa_1k+\iota_1),\qquad
 q_n(k,k-1)=\frac{k^2}{n^2}\{\kappa_2(n-k)+\iota_2\}.
\]
Define
\begin{equation}
\label{eq:tilde_mu_invariant_measure}
 \widetilde\mu_{\ell_1,\ell_2}^{(n)}(k/n)
 =\kappa_1^k\kappa_2^{n-k}\binom nk^2
   \frac{(\ell_1)_{(k)}(\ell_2)_{(n-k)}}{(\ell_1+\ell_2)_{(n)}}.
\end{equation}
For $0\le k<n$, the ratio of consecutive weights is
\[
 \frac{\widetilde\mu_{\ell_1,\ell_2}^{(n)}((k+1)/n)}
      {\widetilde\mu_{\ell_1,\ell_2}^{(n)}(k/n)}
 =\frac{\kappa_1}{\kappa_2}\left(\frac{n-k}{k+1}\right)^2
   \frac{\ell_1+k}{\ell_2+n-k-1}
 =\frac{q_n(k,k+1)}{q_n(k+1,k)}.
\]
All nearest-neighbour transitions within $\{0,\ldots,n\}$
have positive rates; in particular,
$q_n(0,1)=\iota_1>0$ and $q_n(n,n-1)=\iota_2>0$.
Detailed balance and irreducibility therefore give
\begin{equation}
\label{eq:lifted_invariant_measure_formula}
 \pi_n(k)=\frac{\widetilde\mu_{\ell_1,\ell_2}^{(n)}(k/n)}
                   {\mathcal Z_{\ell_1,\ell_2}^{(n)}},
 \qquad
 \mathcal Z_{\ell_1,\ell_2}^{(n)}
 =\sum_{j=0}^n\widetilde\mu_{\ell_1,\ell_2}^{(n)}(j/n).
\end{equation}

\subsection{Binomial approximation}
Let $B_n\sim\mathrm{Bin}(n,\alpha_*)$, with probability mass
function $p_n^{\text{Bin}}$. For $0\le k\le n$, define
\[
\begin{aligned}
 A_n(k)
 &=\binom nk
 \frac{(\ell_1)_{(k)}(\ell_2)_{(n-k)}}
      {(\ell_1+\ell_2)_{(n)}}=\int_0^1\binom nk x^k(1-x)^{n-k}
 \frac{x^{\ell_1-1}(1-x)^{\ell_2-1}}
      {\mathrm B(\ell_1,\ell_2)}\,\dd x.
\end{aligned}
\]
This representation identifies $A_n$ as a beta-binomial probability
mass function: summing over $k$ and applying the binomial theorem
gives $\sum_{k=0}^n A_n(k)=1$, and hence $0<A_n(k)\le1$.

Since $\alpha_*=\kappa_1/(\kappa_1+\kappa_2)$, the invariant
weights factor as
 \[\widetilde\mu_{\ell_1,\ell_2}^{(n)}(k/n)
 =(\kappa_1+\kappa_2)^n p_n^{\text{Bin}}(k)A_n(k).\]
Summing over $k$ and normalising therefore give
\begin{equation}
\label{eq:partition_function_binomial_expectation}
 \frac{\mathcal Z_{\ell_1,\ell_2}^{(n)}}
      {(\kappa_1+\kappa_2)^n}
 =\mathbb E A_n(B_n),
 \qquad
 \frac{\pi_n(k)}{p_n^{\text{Bin}}(k)}
 =\frac{A_n(k)}{\mathbb E A_n(B_n)}.
\end{equation}

\begin{lemma}
\label{partition_function_estimate}
As $n\to\infty$,
\[
 \mathcal Z_{\ell_1,\ell_2}^{(n)}
 =\frac{\Gamma(\ell_1+\ell_2)}{\Gamma(\ell_1)\Gamma(\ell_2)}
   \frac{\alpha_*^{\ell_1-1}(1-\alpha_*)^{\ell_2-1}}n
   (\kappa_1+\kappa_2)^n\bigl(1+O(n^{-1/2})\bigr).
\]
\end{lemma}

\begin{proof}[Proof of \Cref{partition_function_estimate}]
Put
 $h(x)=\frac{x^{\ell_1-1}(1-x)^{\ell_2-1}}
           {\mathrm B(\ell_1,\ell_2)}$,
and choose $0<\delta<\min\{\alpha_*,1-\alpha_*\}$.
The gamma-ratio expansion
$\Gamma(m+a)/\Gamma(m+b)=m^{a-b}(1+O(m^{-1}))$
gives, uniformly for $\delta n\le k\le(1-\delta)n$,
\begin{equation}
\label{eq:An_central_asymptotic}
 nA_n(k)=h(k/n)+O(n^{-1}).
\end{equation}
On $\{\delta\le B_n/n\le1-\delta\}$,
\eqref{eq:An_central_asymptotic} and the mean value theorem give
$|nA_n(B_n)-h(\alpha_*)|
\le Cn^{-1}+C|B_n/n-\alpha_*|$,
since $h'$ is bounded on $[\delta,1-\delta]$.
Hoeffding's inequality gives
$\mathbb P(B_n/n\notin[\delta,1-\delta])\le Ce^{-cn}$
for some $c>0$, while $A_n\le1$ implies
$|nA_n(B_n)-h(\alpha_*)|\le n+h(\alpha_*)$.
Splitting the expectation over these two events and using
Cauchy--Schwarz therefore yields
\begin{equation}
\label{eq:invariant_weight_L1}
 \mathbb E|nA_n(B_n)-h(\alpha_*)|
 \le \frac Cn
 +C\sqrt{\frac{\alpha_*(1-\alpha_*)}{n}}
 +(n+h(\alpha_*))Ce^{-cn}
 \le Cn^{-1/2}.
\end{equation}
Since
$|n\mathbb E A_n(B_n)-h(\alpha_*)|
\le\mathbb E|nA_n(B_n)-h(\alpha_*)|$,
we get $n\mathbb E A_n(B_n)=h(\alpha_*)+O(n^{-1/2})$.
The assertion follows from
\eqref{eq:partition_function_binomial_expectation}
and $h(\alpha_*)>0$.
\end{proof}

We write
$c_* \coloneqq h(\alpha_*)=
\frac{\alpha_*^{\ell_1-1}(1-\alpha_*)^{\ell_2-1}}
     {\mathrm{B}(\ell_1,\ell_2)}
>0$.

\begin{lemma}
\label{invariant_measure_to_bin}
There is a finite constant $C$, depending only on the four parameters, such that
\[
 \dtv\left(\pi_n,\mathrm{Bin}(n,\alpha_*)\right)\le Cn^{-1/2},
 \qquad n\ge1.
\]
\end{lemma}

\begin{proof}[Proof of \Cref{invariant_measure_to_bin}]
By \eqref{eq:partition_function_binomial_expectation},
\[
 \dtv(\pi_n,\mathrm{Bin}(n,\alpha_*))
 =\frac{\mathbb E|nA_n(B_n)-n\mathbb E A_n(B_n)|}
        {2n\mathbb E A_n(B_n)}
 \le\frac{\mathbb E|nA_n(B_n)-c_*|}{n\mathbb E A_n(B_n)}.
\]
Use \eqref{eq:invariant_weight_L1} and
$n\mathbb E A_n(B_n)\to c_*>0$; enlarging $C$ covers the finitely many
remaining values of $n$.
\end{proof}

\begin{proof}[Proof of \Cref{lemma:invariant_measure_to_product_bernoulli}]
Both configuration laws are uniform conditional on the global count. Hence
\[
 \dtv\left(\mu_{\kappa_1,\kappa_2,\iota_1,\iota_2}^{(n)},\nu_*^{(n)}\right)
 =\frac12\sum_{k=0}^n\binom nk
       \left|\frac{\pi_n(k)-p_n^{\text{Bin}}(k)}{\binom nk}\right|
 =\dtv(\pi_n,\mathrm{Bin}(n,\alpha_*)).
\]
Apply Lemma~\ref{invariant_measure_to_bin}.
\end{proof}

\section{Computations for Yau's relative entropy method}
\label{appendix_Yau_computations}

We compute $\mathbb F_t^{(n)}$ and we give the three estimates used
in the short-time and long-time arguments.

\subsection{Computation for
\texorpdfstring{$\mathbb F_t^{(n)}$}{F(t,n)}}
\label{subsec:computations_Yau}
For $x,y\in[0,1]$, set
\begin{equation}
\label{eq:F_two_variables_Yau}
 \mathcal F^{(n)}(x,y)
 \defeq [\kappa_1(1-x)-\kappa_2x]y(1-y)
 +\frac{\iota_1(1-x)(1-y)-\iota_2xy}{n}.
\end{equation}

\begin{lemma}
\label{computation_short_time_estimate_Ft}
\label{short_time_estimate_computation_Ft_I}
\label{short_time_estimate_computation_Ft_II}
For every $t>0$ and $\eta\in\Omega_n$, the following identity holds,
with the argument $\eta$ suppressed in the density and fluctuation
functions:
\[
\begin{aligned}
 \mathbb F_t^{(n)}(\eta)
 &=
 \sum_{i\in Q}
 \frac{\sqrt{a^i}\,\xi_t^{(i,n)}\xi_t^{(n)}}
 {\rho_t^{(i,n)}(1-\rho_t^{(i,n)})}
 \Bigl\{
 [\kappa_1-(\kappa_1+\kappa_2)\rho_t^{(i,n)}]\\
 &\qquad\qquad\times(1-M^{(n)}-\rho_t^{(n)})
 -\frac{\iota_1(1-\rho_t^{(i,n)})+\iota_2\rho_t^{(i,n)}}{n}
 \Bigr\}\\
 &\quad+
 \sum_{i\in Q}a^iM^{(i,n)}
 \frac{1-\rho_t^{(i,n)}}{\rho_t^{(i,n)}}
 \left[\kappa_1(2M^{(n)}-1)+\frac{\iota_1-\kappa_1}{n}\right]\\
 &\quad+
 \sum_{i\in Q}a^i(1-M^{(i,n)})
 \frac{\rho_t^{(i,n)}}{1-\rho_t^{(i,n)}}
 \left[\kappa_2(1-2M^{(n)})+\frac{\iota_2-\kappa_2}{n}\right].
\end{aligned}
\]
\end{lemma}

\begin{proof}[Proof of \Cref{computation_short_time_estimate_Ft}]
For $x\in A^i$, the product structure and the adjoint formula give
\[
\begin{aligned}
 \frac{\nu_t^{(n)}(\eta^x)}{\nu_t^{(n)}(\eta)}
 =\eta_x\frac{1-\rho_t^{(i,n)}}{\rho_t^{(i,n)}}
 +(1-\eta_x)\frac{\rho_t^{(i,n)}}{1-\rho_t^{(i,n)}},\ \
 L_n^*\mathbf1(\eta)
 =\sum_x\left[
 c_x(\eta^x)\frac{\nu_t^{(n)}(\eta^x)}{\nu_t^{(n)}(\eta)}
 -c_x(\eta)\right].
\end{aligned}
\]
The incoming birth and death rates are evaluated at
$M^{(n)}-n^{-1}$ and $M^{(n)}+n^{-1}$, respectively.
Their differences from the corresponding rates at $M^{(n)}$ are
\[
 \frac{\kappa_1(2M^{(n)}-1)}n+\frac{\iota_1-\kappa_1}{n^2},
 \qquad
 \frac{\kappa_2(1-2M^{(n)})}n+\frac{\iota_2-\kappa_2}{n^2}.
\]
Summing over the $na^iM^{(i,n)}$ occupied and
$na^i(1-M^{(i,n)})$ vacant sites in each block gives
\begin{equation}
\label{eq:Lstar_one_F_plus_error}
\begin{aligned}
 L_n^*\mathbf1(\eta)
 &=
 \sum_i
 \frac{na^i(M^{(i,n)}-\rho_t^{(i,n)})}
 {\rho_t^{(i,n)}(1-\rho_t^{(i,n)})}
 \mathcal F^{(n)}(\rho_t^{(i,n)},M^{(n)})\\
 &\quad+\text{the two correction sums in the statement of this lemma}.
\end{aligned}
\end{equation}
Differentiating the product measure and using
$\partial_t\rho_t^{(i,n)}
=\mathcal F^{(n)}(\rho_t^{(i,n)},\rho_t^{(n)})$ yields
\begin{equation}
\label{eq:dt_log_nu_F}
 \partial_t\log\nu_t^{(n)}(\eta)
 =
 \sum_i
 \frac{na^i(M^{(i,n)}-\rho_t^{(i,n)})}
 {\rho_t^{(i,n)}(1-\rho_t^{(i,n)})}
 \mathcal F^{(n)}(\rho_t^{(i,n)},\rho_t^{(n)}).
\end{equation}
Finally, direct subtraction gives
\[
\begin{aligned}
 \mathcal F^{(n)}(x,y)-\mathcal F^{(n)}(x,z)
 &=(y-z)\Bigl[
 [\kappa_1(1-x)-\kappa_2x](1-y-z)-\frac{\iota_1(1-x)+\iota_2x}{n}
 \Bigr].
\end{aligned}
\]
Subtracting \eqref{eq:dt_log_nu_F} from
\eqref{eq:Lstar_one_F_plus_error} and using
 $na^i(M^{(i,n)}-\rho_t^{(i,n)})(M^{(n)}-\rho_t^{(n)})
 =\sqrt{a^i}\,\xi_t^{(i,n)}\xi_t^{(n)}$
proves the identity.
\end{proof}

The short-time entropy estimate also requires control of the proportion of
sites whose states differ from their states in the initial block decomposition.

\begin{lemma}[Discrepancy bound for the initial blocks]
\label{lemma:defect_estimate_initial_blocks}
If $\min_{i\in Q}a^i\ge\varepsilon_0>0$, then
\[
 \mathbb E_{\eta_0}^{(n)}
 \left[M_t^{(0,n)}+1-M_t^{(1,n)}\right]
 \le C_{\varepsilon_0}t,\qquad t\ge0.
\]
\end{lemma}

\begin{proof}[Proof of \Cref{lemma:defect_estimate_initial_blocks}]
Let $N_t$ be the number of flips up to time $t$.
Every site differing from its initial state has flipped at least
once, so
\[
 \varepsilon_0n\left(M_t^{(0,n)}+1-M_t^{(1,n)}\right)
 \le\sum_{x\in\Lambda_n}
 \mathbf1_{\{\sigma_x^{(n)}(t)\ne(\eta_0)_x\}}
 \le N_t.
\]
Since $\sum_x c_x(\sigma)\le Cn$ uniformly in $\sigma$,
 $\mathbb E_{\eta_0}^{(n)}[N_t]
 =\mathbb E_{\eta_0}^{(n)}
 \left[\int_0^t\sum_x c_x(\sigma^{(n)}(s))\,\dd s\right]
 \le Cnt$.
Taking expectations in the first inequality and dividing by
$\varepsilon_0n$ proves the claim.
\end{proof}
\subsection{Initialisation at time zero}

Since $\nu_0^{(n)}=\delta_{\eta_0}$ does not have full support,
we apply Yau's inequality at positive times. The following lemma
justifies letting the initial time decrease to zero.

\begin{lemma}
\label{lemma:initial_relative_entropy_limit}
Fix $n\in\mathbb N$ and $\eta_0\in\Omega_n^*$, and construct
$\nu_t^{(n)}$ from $\eta_0$. Then
\[
 H\left(\eta_0S_t^{(n)}\,\middle|\,\nu_t^{(n)}\right)
 \longrightarrow0\qquad\text{as }t\downarrow0.
\]
\end{lemma}

\begin{proof}[Proof of \Cref{lemma:initial_relative_entropy_limit}]
The ODE gives
\begin{equation}
\label{eq:small_time_rho_expansion_entropy}
\begin{aligned}
 \rho_t^{(0,n)}
 =a^0\left(\kappa_1a^1+\frac{\iota_1}{n}\right)t+O(t^2),\quad
 1-\rho_t^{(1,n)}
 =a^1\left(\kappa_2a^0+\frac{\iota_2}{n}\right)t+O(t^2).
\end{aligned}
\end{equation}
Both leading coefficients are positive. If $\eta$ differs from
$\eta_0$ at $r$ sites, the product formula therefore gives
$\nu_t^{(n)}(\eta)=c_\eta t^r(1+O(t))$ with $c_\eta>0$.

A nonzero term connecting $\eta_0$ to $\eta$ in a matrix power
of $L_n$ requires at least $r$ off-diagonal factors.
At order $r$, such terms exist and contain only positive flip
rates. Hence $(L_n^m)(\eta_0,\eta)=0$ for $m<r$, and
 $\eta_0S_t^{(n)}(\eta)
 =\frac{(L_n^r)(\eta_0,\eta)}{r!}\,t^r(1+O(t))$, and $(L_n^r)(\eta_0,\eta)>0$.
For $\eta\ne\eta_0$, combining this with
$\nu_t^{(n)}(\eta)=c_\eta t^r(1+O(t))$ gives
 $\log\frac{\eta_0S_t^{(n)}(\eta)}{\nu_t^{(n)}(\eta)}
 =\log\frac{(L_n^r)(\eta_0,\eta)}{r!c_\eta}+O(t)$.
This logarithm is bounded, while
$\eta_0S_t^{(n)}(\eta)=O(t^r)$ with $r\ge1$. Therefore
 $\eta_0S_t^{(n)}(\eta)
 \log\frac{\eta_0S_t^{(n)}(\eta)}{\nu_t^{(n)}(\eta)}
 =O(t^r)\to 0$.
For $\eta=\eta_0$, both probabilities equal $1+O(t)$, so their
logarithmic ratio and the corresponding entropy contribution
are $O(t)$. Since $n$ is fixed and $\Omega_n$ is finite,
$H\left(\eta_0S_t^{(n)}\,\middle|\,\nu_t^{(n)}\right)
 =\sum_{\eta\in\Omega_n}\eta_0S_t^{(n)}(\eta)
 \log\frac{\eta_0S_t^{(n)}(\eta)}{\nu_t^{(n)}(\eta)}
 \to 0$.
\end{proof}

\subsection{Uniform exponential moment under the reference measure}

The following bound allows the entropy inequality to
control the quadratic fluctuation terms in $\mathbb F_t^{(n)}$.

\begin{lemma}
\label{lemma:Yau_quadratic_exponential_moment}
There exist constants $\gamma,C>0$ such that, for every
$n$, $t>0$, and choice of nonempty initial blocks,
\[
 \int_{\Omega_n}
 \exp\left\{\gamma\left[
 1+(\xi_t^{(n)})^2+\sum_{i\in Q}(\xi_t^{(i,n)})^2
 \right]\right\}\,\dd\nu_t^{(n)}
 \le C.
\]
\end{lemma}

\begin{proof}[Proof of \Cref{lemma:Yau_quadratic_exponential_moment}]
Under $\nu_t^{(n)}$, the local fluctuations are independent
normalised centred binomial variables. Hoeffding's lemma gives
$\mathbb E_{\nu_t^{(n)}}e^{\theta\xi_t^{(i,n)}}
\le e^{\theta^2/8}$ for every $\theta\in\mathbb R$.
For an independent $Z\sim\mathcal N(0,1)$, the Gaussian
moment-generating function gives
$\mathbb E_Ze^{\sqrt2 Zy}=e^{y^2}$ for every $y\in\mathbb R$.
Tonelli's theorem and Hoeffding's bound with
$\theta=\sqrt2 z$ therefore yield
\[
\begin{aligned}
 \mathbb E_{\nu_t^{(n)}}e^{(\xi_t^{(i,n)})^2}
 =\mathbb E_Z\mathbb E_{\nu_t^{(n)}}
   e^{\sqrt2 Z\xi_t^{(i,n)}}\le\mathbb E_Ze^{Z^2/4}
 =\frac1{\sqrt{2\pi}}\int_{\mathbb R}e^{-z^2/4}\,\dd z
 =\sqrt2.
\end{aligned}
\]
Since $\sum_i a^i=1$, Cauchy--Schwarz gives
$(\xi_t^{(n)})^2\le\sum_i(\xi_t^{(i,n)})^2$.
For $\gamma=1/2$, independence thus bounds the required
expectation by
 $e^{1/2}\prod_{i\in Q}
 \mathbb E_{\nu_t^{(n)}}e^{(\xi_t^{(i,n)})^2}
 \le2e^{1/2}$.
\end{proof}

\section{A technical proof on the cutoff window}
\label{appendix_profile_window}
\label{appendix_profile_window_reference}

\begin{proof}[Proof of
\Cref{lemma:profile_window_reference_approximation}]
Fix $R>0$ and $u\in(0,T_0]$, and put
$r_n(s)=t^{(n)}(s)-u$. All estimates below are uniform over
$|s|\le R$ and initial configurations in the fixed bulk class.

Since $r_n(s)=t_n+O_{R,u}(1)$, the quantitative CLTs in
Propositions~\ref{prop:CLT_mag_tk}
and~\ref{prop:CLT_mag_GD} apply at these times.
The covariance bounds
\eqref{eq:beta_sigma_close_to_limit} and
\eqref{eq:Gamma_barGamma_close}, together with
Lemma~\ref{dk_two_gaussians}, therefore show that the laws of
$\vec{\xi_{r_n(s)}^{(n)}}$ and
$\vec{\zeta_{r_n(s)}^{(n)}}$ have $d_K$-distance $o_n(1)$.
These underlying estimates also apply when $r_n(s)<t_n$.

Together with the reference variance bound
\eqref{quadratic_fluctuation_estimates_under_reference_measure}
and the deterministic estimate
\eqref{eq:sharp_bulk_deterministic_decay}, this yields the
coupling used in Subsection~\ref{estimate_in_dtv}, with all
four local densities in $[\alpha_*/2,(1+\alpha_*)/2]$
with probability tending to one.
Applying the short-restart argument leading to
\eqref{eq:unconditioned_evolved_TV_bound} with $T=r_n(s)$,
and taking $n\to\infty$ followed by $\varepsilon\downarrow0$,
gives
\[
 \limsup_{n\to\infty}\sup_{|s|\le R}
 \dtv\left(
 \sigma_0^{(n)}S_{t^{(n)}(s)}^{(n)},
 \nu_{r_n(s)}^{(n)}S_u^{(n)}
 \right)
 \le C_{\varepsilon_0}\sqrt u.
\]
The constant is independent of $u\in(0,T_0]$, since the
restarted configurations lie in the same fixed interior
density interval.

Finally, \eqref{eq:sharp_bulk_deterministic_decay} gives
$\omega_{r_n(s)+v}^{(n)}
 \le C_{\varepsilon_0,R,T_0}/\sqrt n$
for $0\le v\le u$. The reference parameters thus remain
in a fixed compact subinterval of $(0,1)$, so
\eqref{eq:shifted_time_entropy_bound} applies even if
$r_n(s)<t_n$ and yields
\[
 \sup_{|s|\le R}
 H\left(
 \nu_{r_n(s)}^{(n)}S_u^{(n)}
 \,\middle|\,\nu_{t^{(n)}(s)}^{(n)}
 \right)
 \le
 \exp\left\{\frac{C_{\varepsilon_0,R,T_0}u}{\sqrt n}\right\}
 -1
 \longrightarrow0.
\]
Pinsker's inequality and the triangle inequality give an
upper limit of $C_{\varepsilon_0}\sqrt u$ for the supremum
in the lemma. Letting $u\downarrow0$ completes the proof.
\end{proof}
\section{Auxiliary estimates for worst-case mixing}
\label{app:worst_case_tools}
We give the estimates used in Section~\ref{sec:worst_case}.

\subsection{Sharp boundary entrance}
\label{app:sharp_boundary_entrance}

\begin{proof}[Proof of \Cref{lemma:sharp_boundary_entrance}]
Put $F(x)=x(1-x)\left(\kappa_1-(\kappa_1+\kappa_2)x\right)$.
The clock in \eqref{eq:boundary_clock_definition} satisfies
$\psi'(x)=1/F(x)$ and
\begin{equation}
\label{eq:boundary_clock_at_zero}
 \psi(x)-\kappa_1^{-1}\log x\longrightarrow0
 \qquad\text{as }x\downarrow0.
\end{equation}
Choose $\gamma\in(0,1/2)$ and $j_n=\lfloor n^\gamma\rfloor$.
We first couple $K$ with the pure-birth immigration process $Y$,
started at $\ell$, with birth rate $\kappa_1k+\iota_1$.
While the chains agree at $k<j_n$, use three independent
exponential clocks. At rate $b_n(k)$, both chains jump to $k+1$;
at rate $(\kappa_1k+\iota_1)-b_n(k)$, only $Y$ jumps to $k+1$;
and at rate $d_n(k)$, only $K$ jumps to $k-1$.
These rates are nonnegative, and $d_n(0)=0$.
The marginal rates are the required ones. Restart the clocks after
each common birth, and stop at the first discrepancy or the common
visit to $j_n$. 

Only the second and third clocks create a discrepancy. Their total
rate is bounded by
 $(\kappa_1k+\iota_1)-b_n(k)+d_n(k)
 \le\frac{C(k+1)^2}{n}$.
Conditional on reaching $(k,k)$ without a discrepancy, the
probability that the next event creates one is therefore
 $\frac{(\kappa_1k+\iota_1)-b_n(k)+d_n(k)}
      {\kappa_1k+\iota_1+d_n(k)}
 \le\frac{C(k+1)}n$,
because $\kappa_1k+\iota_1\ge
\min\{\kappa_1,\iota_1\}(k+1)$.
Before a discrepancy, the common path visits the successive states
$\ell,\ldots,j_n-1$. Summing the conditional failure probabilities
shows that the probability of stopping at a discrepancy is at most
 $\frac Cn\sum_{k=\ell}^{j_n-1}(k+1)
 \le\frac{Cj_n^2}{n}=o(1)$.

The process $(Y_{t/\kappa_1}-\ell)_{t\ge0}$ is a linear
birth process with immigration, started from zero, with birth rate
$m+\ell+\iota_1/\kappa_1$ at state $m$.
By \cite[Theorem~5 and Corollary~1]{Tavare1987BirthImmigration},
\[
 e^{-\kappa_1t}Y_t\longrightarrow W_{0,\ell}
 \quad\text{almost surely},\qquad
 W_{0,\ell}\sim\operatorname{Gamma}
           (\ell+\iota_1/\kappa_1,1).
\]
Let $\tau_j^Y$ be the first visit of $Y$ to $j$.
The process is nonexplosive and has positive birth rates, so
$\tau_j^Y\to\infty$ almost surely and $Y_{\tau_j^Y}=j$.
Evaluating the preceding limit at these times and using
$W_{0,\ell}>0$ almost surely gives
 $\tau_j^Y-\kappa_1^{-1}\log j
 \longrightarrow-\kappa_1^{-1}\log W_{0,\ell}
 \ \text{almost surely}$.
Under the coupling, $\tau_{j_n}=\tau_{j_n}^Y$ unless a
discrepancy occurs before the common visit to $j_n$. Hence
 $\mathbb P(\tau_{j_n}\ne\tau_{j_n}^Y)
 \le \frac{Cj_n^2}{n}=o(1)$,
so $\tau_{j_n}-\tau_{j_n}^Y=o_{\mathbb P}(1)$.
Since $j_n\to\infty$, the preceding almost-sure limit gives
\begin{equation}
\label{eq:mesoscopic_boundary_entrance}
 \tau_{j_n}-\kappa_1^{-1}\log j_n
 \ \Longrightarrow\ -\kappa_1^{-1}\log W_{0,\ell}.
\end{equation}

For the second stage, start $K$ at $j_n$, put
$l_n=\lfloor j_n/2\rfloor$ and $u_n=\lfloor nq\rfloor$, and
stop at $\tau=\inf\{t:K_t\in\{l_n,u_n\}\}$.
For all sufficiently large $n$, uniformly over $l_n\le k\le u_n$,
 $\frac{d_n(k)}{b_n(k)}
 \le\frac{\kappa_2q}{\kappa_1(1-q)}+O(n^{-1})<1$.
Choose $h>0$, independent of $n$, so that
$d_n(k)\le e^{-h}b_n(k)$ on this interval.
Writing $\mathcal L_n$ for the count generator and $V(k)=e^{-hk}$,
we have, at every interior state,
\[
\begin{aligned}
 \mathcal L_nV(k)
 &=b_n(k)\{V(k+1)-V(k)\}
   +d_n(k)\{V(k-1)-V(k)\}\\
 &=e^{-hk}(e^h-1)\{d_n(k)-e^{-h}b_n(k)\}\le0.
\end{aligned}
\]
Thus $(e^{-hK_{t\wedge\tau}})_{t\ge0}$ is a bounded
supermartingale. The finite irreducible count chain has
$\tau<\infty$ almost surely. Optional stopping at $t\wedge\tau$
and bounded convergence give
\[
 e^{-hl_n}\mathbb P_{j_n}(K_\tau=l_n)
 \le\mathbb E_{j_n}[e^{-hK_\tau}]
 \le e^{-hj_n}.
\]
Consequently,
\begin{equation}
\label{eq:mesoscopic_lower_exit_probability}
 \mathbb P_{j_n}(K_\tau=l_n)
 \le e^{-h(j_n-l_n)}=o(1),
\end{equation}
since $j_n-l_n\sim n^\gamma/2$.

Set $g_n(k)=\psi(k/n)$ for $l_n\le k\le u_n$.
Uniformly over the interior states, we claim that
\begin{equation}
\label{eq:boundary_clock_generator_bounds}
\begin{split}
 |\mathcal L_ng_n(k)-1|\le\frac Ck,\qquad \text{and} \qquad
 \mathcal L_n(g_n^2)(k)-2g_n(k)\mathcal L_ng_n(k)
 \le\frac Ck.
\end{split}
\end{equation}
Indeed, write $x=k/n$. Since $q<\alpha_*$ is fixed,
$F(x)\ge cx$ on $(0,q]$, and hence
$|\psi'(x)|\le C/x$ and $|\psi''(x)|\le C/x^2$ there.
For all sufficiently large $n$, the points between
$(k-1)/n$ and $(k+1)/n$ lie in $[x/2,q]$.
Taylor's theorem therefore gives
 $g_n(k\pm1)-g_n(k)
 =\pm\frac{\psi'(x)}n+O(k^{-2})$.
Using $b_n(k)+d_n(k)\le C(k+1)$ and
\[
 b_n(k)-d_n(k)=nF(x)+I(x),\qquad
 I(x)=\iota_1(1-x)^2-\iota_2x^2,
\]
we obtain
 $\mathcal L_ng_n(k)
 =1+\frac{I(x)}{nF(x)}+O(k^{-1})
 =1+O(k^{-1}).$
For the second bound, use the identity
\[
\begin{aligned}
 \mathcal L_n(g_n^2)(k)-2g_n(k)\mathcal L_ng_n(k)=b_n(k)\left(g_n(k+1)-g_n(k)\right)^2  +d_n(k)\left(g_n(k-1)-g_n(k)\right)^2.
\end{aligned}
\]
Each clock jump is $O(k^{-1})$, so the right-hand side is
$O((k+1)/k^2)=O(k^{-1})$.

For large $n$, the clock drift is at least $1/2$.
Dynkin's formula and the monotonicity of $\psi$
give
\[
\begin{aligned}
 \tfrac12\mathbb E_{j_n}\left[\tau\wedge t\right]
 &\le \mathbb E_{j_n}\left[
        \int_0^{\tau\wedge t}
        \mathcal L_ng_n(K_s)\,\mathrm ds
       \right]=\mathbb E_{j_n}\left[g_n(K_{\tau\wedge t})\right]
       -g_n(j_n)\\
 &\le \psi(u_n/n)-\psi(j_n/n)=O(\log n).
\end{aligned}
\]

Here $\psi(u_n/n)\to\psi(q)$, while
\eqref{eq:boundary_clock_at_zero} implies
$\psi(j_n/n)=-(1-\gamma)\kappa_1^{-1}\log n+o(1)$.
Monotone convergence gives $\mathbb E_{j_n}\tau=O(\log n)$.

Define the stopped martingale
 $\mathcal M_t
 =g_n(K_{t\wedge\tau})-g_n(j_n)
  -\int_0^{t\wedge\tau}\mathcal L_ng_n(K_s)\,\mathrm ds$.
Before $\tau$, we have $K_s>l_n$ and hence $K_s\ge j_n/2$.
By \eqref{eq:boundary_clock_generator_bounds},
\[
\begin{aligned}
 \mathbb E_{j_n}\left[\left|
   \int_0^\tau\{\mathcal L_ng_n(K_s)-1\}\,\mathrm ds\right|\right]\le\frac C{j_n}\mathbb E_{j_n}\left[\tau\right]
 \le\frac{C\log n}{j_n}=o(1).
\end{aligned}
\]
The predictable quadratic variation satisfies the same bound:
\[
\begin{aligned}
 \mathbb E_{j_n}\left[\langle\mathcal M\rangle_\tau\right]
 =\mathbb E_{j_n}\left[\int_0^\tau
  \{\mathcal L_n(g_n^2)-2g_n\mathcal L_ng_n\}(K_s)
  \,\mathrm ds\right]\le\frac C{j_n}\mathbb E_{j_n}\left[\tau\right]
 \le\frac{C\log n}{j_n}=o(1).
\end{aligned}
\]
The martingale square identity, first at bounded times and then
by $L^2$ convergence, gives
$\mathbb E_{j_n}\left[|\mathcal M_\tau|^2\right]
=\mathbb E_{j_n}\left[\langle\mathcal M\rangle_\tau\right]=o(1)$.
Consequently,
\[
\begin{aligned}
 \psi(K_\tau/n)-\psi(j_n/n)-\tau=\mathcal M_\tau
   +\int_0^\tau\{\mathcal L_ng_n(K_s)-1\}\,\mathrm ds
   =o_{\mathbb P}(1).
\end{aligned}
\]

Let $H_n$ be the first hitting time of $u_n$ from $j_n$.
On $\{K_\tau=u_n\}$, we have $H_n=\tau$; its complement has
probability $o(1)$ by \eqref{eq:mesoscopic_lower_exit_probability}.
Thus
 $H_n=\psi(u_n/n)-\psi(j_n/n)+o_{\mathbb P}(1)$.
Apply the strong Markov property.
Since the count must pass through $j_n$ before reaching $u_n$,
\[
\begin{aligned}
 \tau_{u_n}-\kappa_1^{-1}\log n
 =\bigl(\tau_{j_n}-\kappa_1^{-1}\log j_n\bigr)
   +\psi(u_n/n)-\bigl\{\psi(j_n/n)
          -\kappa_1^{-1}\log(j_n/n)\bigr\}
   +o_{\mathbb P}(1).
\end{aligned}
\]
Now use \eqref{eq:mesoscopic_boundary_entrance},
$\psi(u_n/n)\to\psi(q)$, and
\eqref{eq:boundary_clock_at_zero}.
\eqref{eq:sharp_boundary_entrance_limit} follows.
Applying the same argument to $n-K$, with the two species
interchanged, proves the reflected assertion.
\end{proof}

\subsection{From weak convergence to total variation}

\begin{lemma}[Total variation from a point for a birth--death chain]
\label{lemma:birth_death_tv_oscillation}
Let $P_t$ be the transition semigroup of an irreducible continuous-time
birth--death chain on $\{0,\ldots,n\}$, with invariant law $\pi$.
For an initial state $k$ and $t\ge0$, define
\[
 H_t(j)=P_t(k,\{0,\ldots,j\})-\pi(\{0,\ldots,j\}),
 \qquad 0\le j\le n,
\]
and set $H_t(-1)=0$. Then
\begin{equation}
\label{eq:birth_death_tv_oscillation}
 \dtv(P_t(k,\cdot),\pi)
 =\max_{-1\le j\le n}H_t(j)-\min_{-1\le j\le n}H_t(j).
\end{equation}

For a sequence of such chains with semigroups $P_t^{(n)}$ and
invariant laws $\pi_n$, choose deterministic initial states
$k_n\in\{0,\ldots,n\}$, times $v_n\ge0$, and strictly increasing
affine maps $r_n:\mathbb R\to\mathbb R$. Define
\[
\begin{aligned}
 F_n(x)
 &=P_{v_n}^{(n)}
   \bigl(k_n,\{j\in\{0,\ldots,n\}:r_n(j)\le x\}\bigr),\\
 G_n(x)
 &=\pi_n\bigl(\{j\in\{0,\ldots,n\}:r_n(j)\le x\}\bigr).
\end{aligned}
\]
Suppose the rescaled laws converge weakly to laws with continuous
distribution functions $F$ and $G$, respectively; equivalently,
$F_n(x)\to F(x)$ and $G_n(x)\to G(x)$ for every $x\in\mathbb R$.
Then
\begin{equation}
\label{eq:birth_death_tv_weak_limit}
 \dtv(P_{v_n}^{(n)}(k_n,\cdot),\pi_n)
 \longrightarrow
 \sup_{x\in\mathbb R}\{F(x)-G(x)\}
 -\inf_{x\in\mathbb R}\{F(x)-G(x)\}.
\end{equation}
\end{lemma}

\begin{proof}[Proof of \Cref{lemma:birth_death_tv_oscillation}]
Fix $k$ and $t>0$, and let $Q$ be the generator.
We use \cite[Corollary~4.5 and the continuous-time argument
in the proof of Proposition~4]{DingLubetzkyPeres2010BirthDeath}.
For integers $m\ge1$ satisfying
$m\ge2t\max_j\{-Q(j,j)\}$, the matrix
$R_m=I+(t/m)Q$ is lazy, irreducible, and birth--death, with
invariant law $\pi$. The cited corollary shows that
$j\mapsto R_m^m(k,j)/\pi(j)$ is unimodal, and
 $R_m^m=\left(I+\frac tmQ\right)^m\longrightarrow e^{tQ}=P_t$. Passing its monotonicity inequalities to the
limit shows that $j\mapsto P_t(k,j)/\pi(j)$ is nondecreasing
up to a mode and nonincreasing thereafter.

Put $\Delta(j)=P_t(k,j)-\pi(j)$. Since $\pi(j)>0$, unimodality
implies that
$ A_+=\{j:\Delta(j)>0\}
     =\{j:P_t(k,j)/\pi(j)>1\}$
is an interval, possibly empty. Since $\sum_j\Delta(j)=0$,
 $\dtv(P_t(k,\cdot),\pi)
 =\frac12\sum_{j=0}^n|\Delta(j)|
 =\sum_{j\in A_+}\Delta(j)$.
For every $B\subseteq\{0,\ldots,n\}$,
$\left|\sum_{j\in B}\Delta(j)\right|
\le\sum_{j\in A_+}\Delta(j)$, with equality for $B=A_+$.
Thus, in either case,
\[
\begin{aligned}
 &\dtv(P_t(k,\cdot),\pi)
 =\max_{0\le a\le b\le n}
   \left|\sum_{j=a}^b\Delta(j)\right|\\
 &\quad=
 \max_{0\le a\le b\le n}|H_t(b)-H_t(a-1)|
 =\max_{-1\le j\le n}H_t(j)-\min_{-1\le j\le n}H_t(j).
\end{aligned}
\]
Here we used
$\sum_{j=a}^b\Delta(j)=H_t(b)-H_t(a-1)$ and the fact that
the largest absolute difference between two values is their
maximum minus their minimum.
At $t=0$, the likelihood ratio is
$\mathbf1_{\{j=k\}}/\pi(j)$, which is also unimodal.
This proves \eqref{eq:birth_death_tv_oscillation}.

For the final assertion, $F_n-G_n$ takes exactly the values
of the cumulative probability differences in
\eqref{eq:birth_death_tv_oscillation}. Hence
\[
 \dtv(P_{v_n}^{(n)}(k_n,\cdot),\pi_n)
 =\sup_x\{F_n(x)-G_n(x)\}
  -\inf_x\{F_n(x)-G_n(x)\}.
\]
Weak convergence to continuous distribution functions implies
$\|F_n-F\|_\infty+\|G_n-G\|_\infty\to0$, and therefore
$F_n-G_n\to F-G$ uniformly. For bounded functions $h_n,h$,
\[
 |\sup_x h_n(x)-\sup_x h(x)|\le\|h_n-h\|_\infty,
 \qquad
 |\inf_x h_n(x)-\inf_x h(x)|\le\|h_n-h\|_\infty.
\]
Applying these inequalities to $F_n-G_n$ and $F-G$ proves
\eqref{eq:birth_death_tv_weak_limit}.
\end{proof}

\subsection{Uniform entrance and growing minorities}

\begin{lemma}[Uniform boundary entrance]
\label{lem:uniform_nonzero_minority_entrance}
Put \(j_n=\lfloor an\rfloor\), where
\(a\in(0,\min\{\alpha_*,1-\alpha_*\})\) is fixed and sufficiently small.
If $\tau_{j_n}$ is the first visit to $j_n$, for all sufficiently large \(n\), then
\[
 \sup_{0\le k<j_n}
 \mathbb E_k\left[\left|\tau_{j_n}
          -\frac1{\kappa_1}\log\frac n{k+1}\right|\right]\le C_a.
\]
The reflected statement holds with $\kappa_2$.
\end{lemma}

\begin{proof}[Proof of \Cref{lem:uniform_nonzero_minority_entrance}]
Let $e_r$ be the mean passage time from $r$ to $r+1$. The first-jump
decomposition gives
\[
 e_0=\iota_1^{-1},\qquad
 e_r=b_n(r)^{-1}+\frac{d_n(r)}{b_n(r)}e_{r-1}\quad(r\ge1).
\]
Uniformly below $j_n$, $b_n(r)\ge c(r+1)$ and
$d_n(r)/b_n(r)\le Lr/n$. Choose $a$ so that $2La<1/2$.
Induction then gives $e_r\le C/(r+1)$: the second term is at most
$LaC/r\le2LaC/(r+1)$, and the first term is absorbed by enlarging
$C$. Substituting this bound back into the recurrence, and using
$(1-r/n)^{-2}-1=O(r/n)$, gives, uniformly for $0\le r<j_n$,
 $e_r=\frac1{\kappa_1r+\iota_1}+O(n^{-1})$.
By the strong Markov property at the successive levels,
 $\mathbb E_k\tau_{j_n}
 =\sum_{r=k}^{j_n-1}e_r
 =\frac1{\kappa_1}\log\frac n{k+1}+O_a(1)$,
where the final error is uniform in $k<j_n$. 

Put $m_n(k)=\kappa_1^{-1}\log(n/(k+1))$.
Since $b_n(r)\le\kappa_1r+\iota_1$, the count process can be
coupled below a pure-birth process $Y$, started at $k$, with
birth rate $\kappa_1y+\iota_1$. Its mean is
 $\mathbb E_kY_t
 =\left(k+\frac{\iota_1}{\kappa_1}\right)e^{\kappa_1t}
 -\frac{\iota_1}{\kappa_1}$.
Because $Y$ is nondecreasing,
$\{\tau_{j_n}\le t\}\subseteq\{Y_t\ge j_n\}$.
For $0\le u\le m_n(k)$, Markov's inequality therefore gives
\[
\begin{aligned}
 \mathbb P_k(\tau_{j_n}\le m_n(k)-u)
 \le\frac{\mathbb E_kY_{m_n(k)-u}}{j_n}\le\frac n{j_n}
 \frac{k+\iota_1/\kappa_1}{k+1}e^{-\kappa_1u}
 \le C_a e^{-\kappa_1u},
\end{aligned}
\]
where $n/j_n\le2/a$ for large $n$ and
$(k+\iota_1/\kappa_1)/(k+1)\le\max\{1,\iota_1/\kappa_1\}$.
We also have that
\[
\begin{aligned}
 \mathbb E_k[(m_n(k)-\tau_{j_n})_+]
 =\int_0^{m_n(k)}
 \mathbb P_k(\tau_{j_n}<m_n(k)-u)\,\dd u\le C_a\int_0^\infty e^{-\kappa_1u}\,\dd u
 \le C_a.
\end{aligned}
\]
Together with the uniform mean estimate, this yields, after
enlarging $C_a$,
\[
 \mathbb E_k|\tau_{j_n}-m_n(k)|
 =\mathbb E_k\tau_{j_n}-m_n(k)
 +2\mathbb E_k[(m_n(k)-\tau_{j_n})_+]
 \le C_a.
\]
Reflection proves the second statement.
\end{proof}

\begin{lemma}[Convergence to equilibrium on the worst-case window]
\label{lem:growing_minority_already_mixed}
Fix $s\in\mathbb R$, and let $j_n=\lfloor an\rfloor$ be as in
Lemma~\ref{lem:uniform_nonzero_minority_entrance}.
For initial configurations $\sigma_0^{(n)}$ with counts
$k_n=nM^{(n)}(\sigma_0^{(n)})$, 
\[
 \mathscr D_n(T_n+s/\lambda_*;\sigma_0^{(n)})\longrightarrow0
\]
in each of the following cases, for all sufficiently large $n$:
\begin{enumerate}
\item $k_n<j_n$, and either $k_n\to\infty$ or
      $\kappa_1>\kappa_2$;
\item $k_n>n-j_n$, and either $n-k_n\to\infty$ or
      $\kappa_2>\kappa_1$;
\item $j_n\le k_n\le n-j_n$.
\end{enumerate}
In the third case, convergence is uniform over all such initial
configurations.
\end{lemma}

\begin{proof}[Proof of \Cref{lem:growing_minority_already_mixed}]
Consider first $k_n<j_n$. Put
$c_n=\log n/\min\{\kappa_1,\kappa_2\}$ and
$m_n(k)=\kappa_1^{-1}\log(n/(k+1))$, so that $T_n=t_n+c_n$.
Under either hypothesis in the first case,
\[
 g_n\defeq c_n-m_n(k_n)
 =\left(\frac1{\min\{\kappa_1,\kappa_2\}}
       -\frac1{\kappa_1}\right)\log n
 +\frac{\log(k_n+1)}{\kappa_1}\longrightarrow\infty.
\]
Indeed, both terms are nonnegative, and the second diverges if
$k_n\to\infty$, while the first diverges if $\kappa_1>\kappa_2$.
Since $c_n-g_n/2=m_n(k_n)+g_n/2$,
Lemma~\ref{lem:uniform_nonzero_minority_entrance} and Markov's
inequality give
 $\mathbb P_{k_n}(\tau_{j_n}>c_n-g_n/2)
 \le\frac{2}{g_n}
 \mathbb E_{k_n}|\tau_{j_n}-m_n(k_n)|
 \le\frac{2C_a}{g_n}$.
At entrance the count equals $j_n$, and
$j_n/n\in[a/2,1-a/2]$ for all sufficiently large $n$.
On $\{\tau_{j_n}\le c_n-g_n/2\}$, the entrance configuration
belongs to the bulk class $[a/2,1-a/2]$, and the remaining time is
$T_n+s/\lambda_*-\tau_{j_n}\ge t_n+s/\lambda_*+g_n/2$.
By the strong Markov property, conditioning on the entrance time
and configuration gives a law started from that bulk configuration.
Contraction bounds its distance to equilibrium by
$\mathscr D_{n,a/2}(t_n+s/\lambda_*+g_n/2)$, and convexity
preserves this bound after averaging. The complementary event
contributes at most its probability, bounded by $2C_a/g_n$. Hence
\[
 \mathscr D_n(T_n+s/\lambda_*;\sigma_0^{(n)})
 \le\frac{2C_a}{g_n}
 +\mathscr D_{n,a/2}(t_n+s/\lambda_*+g_n/2).
\]

For every fixed $R>0$, we have $s+\lambda_*g_n/2\ge R$
eventually. By monotonicity of the bulk distance and
Theorem~\ref{cutoff_profile_tk},
 $\limsup_{n\to\infty}
 \mathscr D_n(T_n+s/\lambda_*;\sigma_0^{(n)})
 \le\mathcal G(r_{a/2}^*e^{-R})$.
Letting $R\to\infty$ proves the first case. Applying the same
argument to $n-K^{(n)}$, with the two species interchanged,
proves the second.

Finally, if $j_n\le k_n\le n-j_n$, the initial density belongs
to $[a/2,1-a/2]$ for all sufficiently large $n$.
Since $T_n-t_n=c_n\to\infty$, for every fixed $R>0$,
\[
 \sup_{\substack{\sigma_0\in\Omega_n:\\
 j_n\le nM^{(n)}(\sigma_0)\le n-j_n}}
 \mathscr D_n(T_n+s/\lambda_*;\sigma_0)
 \le\mathscr D_{n,a/2}(t_n+R/\lambda_*)
\]
for all sufficiently large $n$. Using the bulk profile theorem, we first let $n\to\infty$ and then $R\to\infty$ to obtain uniform convergence in the third case.
\end{proof}

\subsection{Loss of dependence on the initial occupied set}
\begin{lemma}[Loss of dependence on finitely many initially occupied sites]
\label{lem:finite_minority_label_erasure}
Fix $k\in\mathbb N_0$ and $s\in\mathbb R$. For each sufficiently
large $n$, let $\sigma_0^{(n)}$ have exactly $k$ ones. At
$t=T_{1,n}+s/\lambda_*$,
\[
 \mathscr D_n(t;\sigma_0^{(n)})
 =\dtv\bigl(\mathcal L_k(K_t^{(n)}),\pi_n\bigr)+o(1),
\]
where $\pi_n$ is the invariant count law and $\mathcal L_k$
denotes the law of the count chain started at $k$.
The reflected assertion holds for configurations with exactly
$k$ zeros, at time $T_{2,n}+s/\lambda_*$.
\end{lemma}

\begin{proof}[Proof of \Cref{lem:finite_minority_label_erasure}]
Couple the chain from the prescribed occupied set with one whose
initial occupied set is uniform among all sets of size $k$.
When their occupied sets are $A,B$, with $|A|=|B|=K$, use shared
birth and death clocks of rates $b_n(K)$ and $d_n(K)$. Put
 $\ell=|A\setminus B|=|B\setminus A|$,
and choose a bijection $\varphi:A\setminus B\to B\setminus A$
at each update.

At a birth, provided $K<n$, with probability
$(n-K-\ell)/(n-K)$ choose a site uniformly from $(A\cup B)^c$
and add it to both sets; with the remaining probability
$\ell/(n-K)$, choose $a$ uniformly from $A\setminus B$ and add
$\varphi(a)$ to $A$ and $a$ to $B$.
At a death, provided $K>0$, with probability $(K-\ell)/K$
remove a uniformly chosen site of $A\cap B$ from both sets;
with the remaining probability $\ell/K$, choose $a$ uniformly
from $A\setminus B$ and remove $a$ from $A$ and $\varphi(a)$
from $B$. 

In either chain, each vacant site is selected at a birth with
probability $1/(n-K)$, and each occupied site at a death with
probability $1/K$, giving the correct transition rates.
Each update changes both counts by the same amount, so they
remain equal. Updating a common site leaves $\ell$ unchanged,
whereas a paired update removes one site from each discrepancy
set and decreases $\ell$ by one. Hence
$0\le\ell_t\le\ell_0\le k$, and once the configurations agree,
they remain equal. They necessarily agree when $K=0$ or $K=n$.

Choose a fixed interval
 $\mathcal U=[\alpha_*-\delta,\alpha_*+\delta]\subset(0,1)$.
For $x=K/n\in\mathcal U$, the discrepancy decreases at rate
$\ell\{b_n(K)/(n-K)+d_n(K)/K\}$, and
\[
 \frac{b_n(K)}{n-K}+\frac{d_n(K)}K
 =(\kappa_1+\kappa_2)x(1-x)
 +\frac{\iota_1(1-x)+\iota_2x}{n}\ge c
\]
for some $c>0$ independent of $n$. Its drift is therefore at
most $-c\ell$ while the common density lies in $\mathcal U$.

We next show that the density remains in $\mathcal U$ during the final
interval of any fixed length $L>0$, with probability tending
to one. For all sufficiently large $n$, $t-L\ge0$, and
 $t-L=T_{1,n}+\frac{s-\lambda_*L}{\lambda_*}$.
Thus \eqref{eq:fixed_seed_count_weak_limit} gives
$M_{t-L}^{(n)}\to\alpha_*$ in probability.
Recall that $F^{(\infty)}(x)
 =x(1-x)\bigl[\kappa_1-(\kappa_1+\kappa_2)x\bigr]$.
For $0\le u\le L$, the density martingale decomposition is
\[
 M_{t-L+u}^{(n)}
 =M_{t-L}^{(n)}
 +\int_0^u
 \left\{
 F^{(\infty)}(M_{t-L+v}^{(n)})
 +\frac{I(M_{t-L+v}^{(n)})}{n}
 \right\}\,\mathrm dv+N_u,
\]
where $I(x)=\iota_1(1-x)^2-\iota_2x^2$ and $N$ is a
square-integrable martingale with $N_0=0$.
The density jumps have size $1/n$, and
$b_n(K)+d_n(K)\le Cn$ uniformly in $K$, so Doob's inequality gives
 $\mathbb E\left[\sup_{0\le u\le L}|N_u|^2\right]
 \le4\mathbb E[\langle N\rangle_L]\le\frac{CL}{n}$.
Since $F^{(\infty)}(\alpha_*)=0$, $F^{(\infty)}$ is Lipschitz
on $[0,1]$, and $I$ is bounded, Gr\"onwall's inequality yields
\[
 \sup_{0\le u\le L}|M_{t-L+u}^{(n)}-\alpha_*|
 \le e^{CL}\left(
 |M_{t-L}^{(n)}-\alpha_*|+\frac{CL}{n}
 +\sup_{0\le u\le L}|N_u|
 \right)=o_{\mathbb P}(1).
\]
Consequently, the probability that the density lies outside
$\mathcal U$ at any time during $[t-L,t]$ tends to zero.

Set
 $\zeta_n=\inf\{u\ge0:M_{t-L+u}^{(n)}\notin\mathcal U\}$ and $Z_u=\ell_{t-L+u}\mathbf1_{\{u<\zeta_n\}}$.
For $u<\zeta_n$, the density lies in $\mathcal U$, so the discrepancy
drift bound applies. At $\zeta_n$, $Z$ can only decrease when
it is set to zero, and it remains zero thereafter. Its drift is
therefore at most $-cZ_u$, giving
 $\mathbb E[Z_L]\le e^{-cL}\mathbb E[Z_0]\le ke^{-cL}$.
On $\{\zeta_n>L\}$, $Z_L=\ell_t$, and a positive discrepancy
is at least one. Hence
\[
 \mathbb P(\ell_t>0)
 \le\mathbb P(\zeta_n\le L)+\mathbb E[Z_L]
 \le o_n(1)+ke^{-cL},
\]
where $L$ is fixed in the $o_n(1)$ term. First letting
$n\to\infty$ and then $L\to\infty$ proves
$\mathbb P(\ell_t>0)\to0$.

Let $\nu_n=\sigma_0^{(n)}S_t^{(n)}$, $\bar\nu_n$ be the
time-$t$ law of the chain with uniform initial occupied set,
and $\mu_n$ denote the invariant law.
The coupling gives
$\dtv(\nu_n,\bar\nu_n)\le\mathbb P(\ell_t>0)=o(1)$.
 
Both laws have count distribution $p_n=\mathcal L_k(K_t^{(n)})$.
The uniform initial law and the dynamics are invariant under
site permutations, so $\bar\nu_n$ is exchangeable; permutation
invariance and uniqueness give the same property for $\mu_n$.
Thus, for every configuration $\eta$ with $r$ ones,
$\bar\nu_n(\eta)=p_n(r)/\binom nr$ and
$\mu_n(\eta)=\pi_n(r)/\binom nr$. Summing over configurations
at each count gives
 $\dtv(\bar\nu_n,\mu_n)
 =\frac12\sum_{r=0}^n|p_n(r)-\pi_n(r)|
 =\dtv(p_n,\pi_n)$.
Since $\mathscr D_n(t;\sigma_0^{(n)})=\dtv(\nu_n,\mu_n)$,
the triangle inequality in both directions yields
\[
 \left|
 \mathscr D_n(t;\sigma_0^{(n)})-\dtv(p_n,\pi_n)
 \right|
 \le\dtv(\nu_n,\bar\nu_n)=o(1).
\]
Exchanging zeros and ones proves the reflected assertion.
\end{proof}

\end{appendix}

\section*{Use of Artificial Intelligence (AI) tools}
The authors used AI tools to check grammar, improve the writing, and generate Figure~\ref{fig:bulk_worst_cutoff}. The mathematical results and proofs were developed independently by the authors.
AI tools also assisted with the literature review, identifying \cite{Tavare1987BirthImmigration} and \cite{DingLubetzkyPeres2010BirthDeath} as useful references for the proofs of Lemmas~\ref{lemma:sharp_boundary_entrance} and~\ref{lemma:birth_death_tv_oscillation}, respectively.
The authors take full responsibility for the content of this paper.

\section*{Acknowledgements}
This work was initiated following the American Institute of Mathematics (AIM) Square group meetings in 2022, 2023, and 2024 in which the first author participated. The first author thanks Ruth J. Williams, Hye-Won Kang, Lea Popovic, Yi Fu, and Grzegorz Rempa{\l}a for many helpful discussions about the model in the context of stochastic averaging. He also thanks the American Institute of Mathematics for their warm hospitality. The authors also thank Arnab Ganguly
for helpful discussions.

\section*{Funding}
The work was supported in part by the Engineering and Physical
Sciences Research Council (EPSRC) of UK Research and Innovation
(UKRI) [grant number EP/Y027795/1] and by the National Science
Foundation (NSF) [grant number DMS-2246815].

\bibliographystyle{plain}
\bibliography{bibliography}

@book {Rachev1991ProbabilityMetrics,
    AUTHOR = {Rachev, Svetlozar T.},
     TITLE = {Probability metrics and the stability of stochastic models},
    SERIES = {Wiley Series in Probability and Mathematical Statistics:
              Applied Probability and Statistics},
 PUBLISHER = {John Wiley \& Sons, Ltd., Chichester},
      YEAR = {1991},
     PAGES = {xiv+494},
      ISBN = {0-471-92877-1},
   MRCLASS = {60B10 (60F05)},
  MRNUMBER = {1105086},
MRREVIEWER = {R.\ M.\ Dudley},
}

@article {Zolotarev1983ProbabilityMetrics,
    AUTHOR = {Zolotarev, V. M.},
     TITLE = {Probability metrics},
   JOURNAL = {Teor. Veroyatnost. i Primenen.},
  FJOURNAL = {Akademiya Nauk SSSR. Teoriya Veroyatnoste\u i\ i ee
              Primeneniya},
    VOLUME = {28},
      YEAR = {1983},
    NUMBER = {2},
     PAGES = {264--287},
      ISSN = {0040-361X},
   MRCLASS = {60B05 (60B11 60F05)},
  MRNUMBER = {700210},
MRREVIEWER = {Jerzy\ Szulga},
}

@article {Mueller1997IntegralProbMetrics,
    AUTHOR = {M\"uller, Alfred},
     TITLE = {Integral probability metrics and their generating classes of functions},
   JOURNAL = {Adv. in Appl. Probab.},
  FJOURNAL = {Advances in Applied Probability},
    VOLUME = {29},
      YEAR = {1997},
    NUMBER = {2},
     PAGES = {429--443},
      ISSN = {0001-8678,1475-6064},
   MRCLASS = {60E99 (60B10)},
  MRNUMBER = {1450938},
MRREVIEWER = {Juan\ A.\ Cuesta-Albertos},
       DOI = {10.2307/1428011},
       URL = {https://doi.org/10.2307/1428011},
}

@book {Villani2009optimaltransport,
    AUTHOR = {Villani, C\'edric},
     TITLE = {Optimal transport},
    SERIES = {Grundlehren der mathematischen Wissenschaften [Fundamental
              Principles of Mathematical Sciences]},
    VOLUME = {338},
      NOTE = {Old and new},
 PUBLISHER = {Springer-Verlag, Berlin},
      YEAR = {2009},
     PAGES = {xxii+973},
      ISBN = {978-3-540-71049-3},
   MRCLASS = {49-02 (28A75 37J50 49Q20 53C23 58E30)},
  MRNUMBER = {2459454},
MRREVIEWER = {Dario\ Cordero-Erausquin},
       DOI = {10.1007/978-3-540-71050-9},
       URL = {https://doi.org/10.1007/978-3-540-71050-9},
}

@article{Ganguly2026StatInference,
  doi = {10.48550/ARXIV.2409.06565},
  url = {https://arxiv.org/abs/2409.06565},
  author = {Ganguly,  Arnab and KhudaBukhsh,  Wasiur R.},
  title = {Statistical inference for a multiscale stochastic model of enzyme kinetics via propagation of chaos},
  publisher = {Elsevier},
  year = {2026},
  journal = {Stochastic Processes and their Applications},
  copyright = {arXiv.org perpetual,  non-exclusive license}
}

@article {Ball2006AAM,
    AUTHOR = {Ball, Karen and Kurtz, Thomas G. and Popovic, Lea and Rempa{\l}a,
              Greg},
     TITLE = {Asymptotic analysis of multiscale approximations to reaction
              networks},
   JOURNAL = {Ann. Appl. Probab.},
  FJOURNAL = {The Annals of Applied Probability},
    VOLUME = {16},
      YEAR = {2006},
    NUMBER = {4},
     PAGES = {1925--1961},
      ISSN = {1050-5164,2168-8737},
   MRCLASS = {60J27 (60F17 60J80 80A30 92C45 92E20)},
  MRNUMBER = {2288709},
MRREVIEWER = {Didier\ Piau},
       DOI = {10.1214/105051606000000420},
       URL = {https://doi.org/10.1214/105051606000000420},
}

@incollection{Kurtz1992Averaging,
  author     = {Kurtz, Thomas G.},
  title      = {Averaging for martingale problems and stochastic
                approximation},
  booktitle  = {Applied stochastic analysis ({N}ew {B}runswick, {NJ}, 1991)},
  series     = {Lect. Notes Control Inf. Sci.},
  volume     = {177},
  pages      = {186--209},
  publisher  = {Springer, Berlin},
  year       = {1992},
  isbn       = {3-540-55296-0},
  mrclass    = {60G42 (60B10 60G57 62L20)},
  mrnumber   = {1169928},
  mrreviewer = {R.\ Sh.\ Liptser},
  doi        = {10.1007/BFb0007058},
}

@article {Ganguly2026tQSSA,
    AUTHOR = {Ganguly, Arnab and KhudaBukhsh, Wasiur R.},
     TITLE = {Asymptotic analysis of the total quasi-steady state
              approximation for the {M}ichaelis-{M}enten enzyme kinetic
              reactions},
   JOURNAL = {J. Math. Anal. Appl.},
  FJOURNAL = {Journal of Mathematical Analysis and Applications},
    VOLUME = {561},
      YEAR = {2026},
    NUMBER = {1},
     PAGES = {130551},
      ISSN = {0022-247X,1096-0813},
   MRCLASS = {60J27 (60H10 60J28 92C45)},
  MRNUMBER = {5040107},
       DOI = {10.1016/j.jmaa.2026.130551},
       URL = {https://doi.org/10.1016/j.jmaa.2026.130551},
}

@article{Hasminskii1966a,
  title     = {On Stochastic Processes Defined by Differential Equations with a Small Parameter},
  volume    = {11},
  issn      = {1095-7219},
  doi       = {10.1137/1111018},
  number    = {2},
  journal   = {Theory of Probability \& Its Applications},
  publisher = {Society for Industrial \& Applied Mathematics (SIAM)},
  author    = {Has’minskii,  R. Z.},
  year      = {1966},
  month     = {1},
  pages     = {211--228}
}

@article{Khasminskii1966b,
  title     = {A Limit Theorem for the Solutions of Differential Equations with Random Right-Hand Sides},
  volume    = {11},
  issn      = {1095-7219},
  doi       = {10.1137/1111038},
  number    = {3},
  journal   = {Theory of Probability \& Its Applications},
  publisher = {Society for Industrial & Applied Mathematics (SIAM)},
  author    = {Khas’minskii,  R. Z.},
  year      = {1966},
  month     = {1},
  pages     = {390--406}
}

@misc{Elsayed2024StochasticAveraging,
  doi = {10.48550/ARXIV.2412.16157},
  url = {https://arxiv.org/abs/2412.16157},
  author = {Elsayed,  Karim S. and Izyumtseva,  Olga and KhudaBukhsh,  Wasiur R. and Rizk,  Amr},
  title = {Stochastic Analysis of Entanglement-assisted Quantum Communication Channels},
  publisher = {arXiv},
  year = {2024},
  copyright = {arXiv.org perpetual,  non-exclusive license}
}

@article{Aldous_Diaconis_cutoff,
  author = {Aldous, David and Diaconis, Persi},
  title = {Shuffling cards and stopping times},
  journal = {Amer. Math. Monthly},
  volume = {93},
  number = {5},
  pages = {333--348},
  year = {1986},
  doi = {10.1080/00029890.1986.11971821},
  url = {https://doi.org/10.1080/00029890.1986.11971821}
}

@misc{BBL2025DensityCutoff,
  author = {Barbour, A. D. and Brightwell, Graham and Luczak, Malwina},
  title = {Convergence to equilibrium for density dependent {Markov} jump processes},
  year = {2025},
  eprint = {2505.12926},
  url = {https://arxiv.org/abs/2505.12926},
  archivePrefix = {arXiv},
  primaryClass = {math.PR},
  note = {Version 4, 1 June 2026; accepted for publication in Electron. J. Probab.}
}

@article{Bibbona2020Interface,
  author  = {Bibbona, Enrico and Kim, Jinsu and Wiuf, Carsten},
  doi     = {10.1098/rsif.2020.0243},
  journal = {J. R. Soc. Interface},
  number  = {168},
  pages   = {20200243},
  title   = {Stationary distributions of systems with discreteness-induced transitions},
  url     = {https://royalsocietypublishing.org/doi/abs/10.1098/rsif.2020.0243},
  volume  = {17},
  year    = {2020}
}

@book{Billingsley1999Convergence,
  author    = {Billingsley, Patrick},
  title     = {Convergence of probability measures},
  series    = {Wiley Series in Probability and Statistics: Probability and
               Statistics},
  edition   = {Second},
  note      = {A Wiley-Interscience Publication},
  publisher = {John Wiley \& Sons, Inc., New York},
  year      = {1999},
  pages     = {x+277},
  isbn      = {0-471-19745-9},
  mrclass   = {60B10 (28A33 60F17)},
  mrnumber  = {1700749},
  doi       = {10.1002/9780470316962},
  url       = {https://doi.org/10.1002/9780470316962}
}

@book{ChenGoldsteinShao2011Stein,
  author = {Chen, Louis H. Y. and Goldstein, Larry and Shao, Qi-Man},
  title = {Normal Approximation by {Stein}'s Method},
  series = {Probability and Its Applications},
  publisher = {Springer, Berlin, Heidelberg},
  year = {2011},
  doi = {10.1007/978-3-642-15007-4},
  url = {https://doi.org/10.1007/978-3-642-15007-4}
}

@article{Diaconis1996cutoff,
  author = {Diaconis, Persi},
  title = {The cutoff phenomenon in finite {Markov} chains},
  journal = {Proc. Natl. Acad. Sci. USA},
  volume = {93},
  number = {4},
  pages = {1659--1664},
  year = {1996},
  doi = {10.1073/pnas.93.4.1659},
  url = {https://doi.org/10.1073/pnas.93.4.1659}
}

@book{Ethier_Kurtz1986MPCC,
  author     = {Ethier, Stewart N. and Kurtz, Thomas G.},
  title      = {Markov Processes: Characterization and Convergence},
  series     = {Wiley Series in Probability and Mathematical Statistics:
                Probability and Mathematical Statistics},
  publisher  = {John Wiley \& Sons, Inc., New York},
  year       = {1986},
  pages      = {x+534},
  isbn       = {0-471-08186-8},
  mrclass    = {60J25 (60B10 60F05 60F17 60G44 60J80)},
  mrnumber   = {838085},
  mrreviewer = {S.\ R. S. Varadhan},
  doi        = {10.1002/9780470316658},
  url        = {https://doi.org/10.1002/9780470316658}
}

@article{FuEtAl2026Fragility,
  author  = {Yi Fu and Hye-Won Kang and Wasiur R. KhudaBukhsh and
             Lea Popovic and Grzegorz A. Rempa{\l}a and Ruth J. Williams},
  title   = {Fragility in a {Togashi--Kaneko} Stochastic Model with Mutations},
  journal = {SIAM Journal on Life Sciences},
  year    = {2026},
  volume  = {1},
  number  = {2},
  pages   = {202--228},
  doi     = {10.1137/25M1799544},
  url     = {https://doi.org/10.1137/25M1799544}
}

@article{gonccalves2021sharp,
  author = {Gon{\c{c}}alves, Patr{\'\i}cia and Jara, Milton and Marinho, Rodrigo and Menezes, Ot{\'a}vio},
  title = {Sharp convergence to equilibrium for the {SSEP} with reservoirs},
  journal = {Ann. Inst. Henri Poincar{\'e} Probab. Stat.},
  volume = {62},
  number = {1},
  pages = {146--165},
  year = {2026},
  doi = {10.1214/23-AIHP1419},
  url = {https://doi.org/10.1214/23-AIHP1419}
}

@article{HeLuczakRoss2025SIS,
  author = {He, Roxanne and Luczak, Malwina and Ross, Nathan},
  title = {Cutoff for the logistic {SIS} epidemic model with self-infection},
  journal = {Adv. in Appl. Probab.},
  volume = {57},
  number = {4},
  pages = {1456--1483},
  year = {2025},
  doi = {10.1017/apr.2025.16},
  url = {https://doi.org/10.1017/apr.2025.16}
}

@misc{jara2018non,
  author = {Jara, Milton and Menezes, Ot{\'a}vio},
  title = {Non-equilibrium fluctuations of interacting particle systems},
  year = {2018},
  eprint = {1810.09526},
  url = {https://arxiv.org/abs/1810.09526},
  archivePrefix = {arXiv},
  primaryClass = {math.PR},
}

@article {jara_amorim_xiang,
    AUTHOR = {Amorim, Julian and Jara, Milton and Xiang, Yangrui},
     TITLE = {Quantitative hydrodynamics for a generalized contact model},
   JOURNAL = {Stochastic Process. Appl.},
  FJOURNAL = {Stochastic Processes and their Applications},
    VOLUME = {188},
      YEAR = {2025},
     PAGES = {104680},
      ISSN = {0304-4149},
   MRCLASS = {60K35 (82C05)},
  MRNUMBER = {4906118},
       DOI = {10.1016/j.spa.2025.104680},
       URL = {https://doi.org/10.1016/j.spa.2025.104680},
}

@misc{JaraXiang,
  author = {Aljovin, Enzo and Jara, Milton and Xiang, Yangrui},
  title = {Thermalization and convergence to equilibrium of the noisy voter model},
  year = {2024},
  eprint = {2409.05722},
  url = {https://arxiv.org/abs/2409.05722},
  archivePrefix = {arXiv},
  primaryClass = {math.PR},
}

@article {Lacoin_profile,
    AUTHOR = {Lacoin, Hubert},
     TITLE = {The cutoff profile for the simple exclusion process on the
              circle},
   JOURNAL = {Ann. Probab.},
  FJOURNAL = {The Annals of Probability},
    VOLUME = {44},
      YEAR = {2016},
    NUMBER = {5},
     PAGES = {3399--3430},
      ISSN = {0091-1798,2168-894X},
   MRCLASS = {60K35 (37L60 60B10 60J10 82C20)},
  MRNUMBER = {3551201},
       DOI = {10.1214/15-AOP1053},
       URL = {https://doi.org/10.1214/15-AOP1053},
}

@article{TK2001Transitions,
  title     = {Transitions Induced by the Discreteness of Molecules in a Small Autocatalytic System},
  author    = {Togashi, Yuichi and Kaneko, Kunihiko},
  journal   = {Phys. Rev. Lett.},
  volume    = {86},
  number    = {11},
  pages     = {2459--2462},
  year      = {2001},
  month     = {Mar},
  publisher = {American Physical Society},
  doi       = {10.1103/PhysRevLett.86.2459},
  url       = {https://link.aps.org/doi/10.1103/PhysRevLett.86.2459}
}

@article{Togashi2003,
  author = {Togashi, Yuichi and Kaneko, Kunihiko},
  title = {Alteration of chemical concentrations through discreteness-induced transitions in small autocatalytic systems},
  journal = {J. Phys. Soc. Japan},
  volume = {72},
  number = {1},
  pages = {62--68},
  year = {2003},
  doi = {10.1143/JPSJ.72.62},
  url = {https://doi.org/10.1143/JPSJ.72.62}
}

@Article{Yau-relative-entropy,
 Author = {Yau, Horng-Tzer},
 Title = {Relative entropy and hydrodynamics of {Ginzburg}-{Landau} models},
 FJournal = {Letters in Mathematical Physics},
 Journal = {Lett. Math. Phys.},
 ISSN = {0377-9017},
 Volume = {22},
 Number = {1},
 Pages = {63--80},
 Year = {1991},
 Language = {English},
 DOI = {10.1007/BF00400379},
 zbMATH = {4197129},
 Zbl = {0725.60120}
}

@article{DingLubetzkyPeres2010BirthDeath,
  author = {Ding, Jian and Lubetzky, Eyal and Peres, Yuval},
  title = {Total variation cutoff in birth-and-death chains},
  journal = {Probab. Theory Related Fields},
  volume = {146},
  number = {1--2},
  pages = {61--85},
  year = {2010},
  doi = {10.1007/s00440-008-0185-3},
  url = {https://doi.org/10.1007/s00440-008-0185-3}
}

@article{Tavare1987BirthImmigration,
  author = {Tavar{\'e}, Simon},
  title = {The birth process with immigration, and the genealogical structure of large populations},
  journal = {J. Math. Biol.},
  year = {1987},
  volume = {25},
  pages = {161--168},
  doi = {10.1007/BF00276387}
}

@article{BarbourHamzaKaspiKlebaner2015Escape,
  author = {Barbour, A. D. and Hamza, K. and Kaspi, Haya and Klebaner, F. C.},
  title = {Escape from the boundary in {Markov} population processes},
  journal = {Adv. in Appl. Probab.},
  year = {2015},
  volume = {47},
  number = {4},
  pages = {1190--1211},
  doi = {10.1239/aap/1449859806}
}

@article{Barbour1990Diffusion,
  author  = {Barbour, A. D.},
  title   = {{Stein}'s method for diffusion approximations},
  journal = {Probability Theory and Related Fields},
  volume  = {84},
  number  = {3},
  pages   = {297--322},
  year    = {1990},
  doi     = {10.1007/BF01197887}
}

@book{KipnisLandim1999,
  author    = {Kipnis, Claude and Landim, Claudio},
  title     = {Scaling Limits of Interacting Particle Systems},
  series    = {Grundlehren der mathematischen Wissenschaften},
  volume    = {320},
  publisher = {Springer-Verlag},
  address   = {Berlin, Heidelberg},
  year      = {1999},
  doi       = {10.1007/978-3-662-03752-2}
}

@article{BordenaveCaputoSalez2018,
  author  = {Bordenave, Charles and Caputo, Pietro and Salez, Justin},
  title   = {Random walk on sparse random digraphs},
  journal = {Probability Theory and Related Fields},
  year    = {2018},
  volume  = {170},
  number  = {3-4},
  pages   = {933--960},
  doi     = {10.1007/s00440-017-0796-7}
}

@article{Teyssier2020Profile,
  author  = {Teyssier, Lucas},
  title   = {Limit profile for random transpositions},
  journal = {The Annals of Probability},
  year    = {2020},
  volume  = {48},
  number  = {5},
  pages   = {2323--2343},
  doi     = {10.1214/20-AOP1424}
}

\end{document}